\documentclass[leqno,11pt]{article}

\usepackage[usenames, dvipsnames]{xcolor}
\usepackage{tikz}

\usepackage{mathpazo}
\usepackage{amssymb,hyperref,amsthm}
\usepackage{euscript}
\usepackage{flafter}
\usepackage{pstricks}
\usepackage[latin1]{inputenc}
\usepackage{amsmath}
\usepackage{amsfonts}
\usepackage{pstricks-add}
\usepackage{yfonts}
\usepackage{egothic}
\usepackage{dsfont}
\usepackage{bm}
\usepackage{variations}

 \newcommand{ \un }{\mathds{1}}

\hypersetup{
    linktoc=page,
   linkcolor=red,          
  citecolor=blue,        
    filecolor=blue,      
   urlcolor=cyan,
    colorlinks=true           
}
\newcommand{\bl}[1]{{\color{blue}{#1}}}

\usepackage[refpage]{nomencl}

\usepackage{makeidx}
\makeindex

\usepackage{enumitem}

\usepackage{mathrsfs} 

\evensidemargin\oddsidemargin

\DeclareFontFamily{U}{mathx}{\hyphenchar\font45}
\DeclareFontShape{U}{mathx}{m}{n}{
      <5> <6> <7> <8> <9> <10>
      <10.95> <12> <14.4> <17.28> <20.74> <24.88>
      mathx10
      }{}
\DeclareSymbolFont{mathx}{U}{mathx}{m}{n}
\DeclareMathAccent{\widecheck}{\mathalpha}{mathx}{"71}

\newcommand{\eqnsection}{
\renewcommand{\theequation}{\thesection.\arabic{equation}}
   \makeatletter
   \csname  @addtoreset\endcsname{equation}{section}
   \makeatother}
\eqnsection

\def\Gl{\mathcal{G}}

\def\r{{\mathbb R}}
\def\e{{\mathbb E}}
\def\p{{\mathbb P}}
\def\P{{\bf P}}
\def\E{{\bf E}}
\def\Q{{\bf Q}}
\def\z{{\mathbb Z}}
\def\N{{\mathbb N}}
\def\T{{\mathbb T}}

\def\1{{\mathds{1}}}

\def\en{\mathcal{E}}
 \def\pa{\overleftarrow}

\def\cB{\mathcal{B}}

\def\mS{{\underline{S}}}
\def\MV{{\overline{V}}}
\def\MS{{\overline{S}}}

\def\uS{{\mathbb{S}}}
\def\Hinf{\mathcal{H}}
\def\Ren{\mathcal{R}}

\newtheorem{theo}{Theorem}[section]
\newtheorem{prop}[theo]{Proposition}
\newtheorem{lem}[theo]{Lemma}

\newtheorem{cor}[theo]{Corollary}

\newtheorem{rem}[theo]{Remark}
\newtheorem{fact}[theo]{Fact}

\newcommand{\R}{\mathbb{R}}
\renewcommand{\T}{\mathbb{T}}

\renewcommand{\epsilon}{\varepsilon}

\def\cb{{\bf c}}
\def\Cb{{\bf C}}

\newcommand{ \mB}{\textgoth{m}}

\newcommand{\V}{\mathbb{V}ar}
\newcommand{\Ve}{\mathbb{V}^{\mathcal{E}}\!\!ar}

\usepackage{tikz}

\title{\bf Range and critical generations of a random walk on Galton-Watson trees}

\author{Pierre Andreoletti\footnote{Laboratoire MAPMO - C.N.R.S. UMR 7349 - F\'ed\'eration Denis-Poisson, Universit\'e d'Orl\'eans
(France). 
 \newline \vspace{0.1cm} \hspace{0.2cm} $\dag$Institut Camille Jordan - C.N.R.S. UMR 5208 - Universit\'e Claude Bernard Lyon 1
(France). 
\newline \vspace{0.1cm}    MSC 2000  60K37; 60J80 ; 60G50. \newline \vspace{0.5cm} \textit{Key words :  random walks, range, random environment, branching random walk} } $\ $,$\ $ Xinxin Chen$^\dag$  }

\begin{document}

\baselineskip=17pt
\setcounter{page}{1}

\maketitle

In this paper we consider a random walk in random environment on a tree and focus on the 
boundary case for the underlying branching potential. We study the range $R_n$ of this walk up to time $n$ and obtain its correct asymptotic in probability which is of order $n/\log n$. This
result is a consequence of the asymptotical behavior of the number of visited sites at generations of order $(\log n)^2$,
which turn out to be the most visited generations. Our proof which involves a quenched analysis
gives a description of the typical environments responsible for the behavior of $R_n$.

\section{Introduction}
\label{Intro}

Let us consider a random walk with a random environment given by a branching random walk. 
{This branching random walk is governed by a point process $\mathcal{L}:=\{A_1,A_2,\cdots,A_N\}$ on the real line, where $N$ is also random in $\mathbb{N}\cup\{\infty\}$. The initial ancestor (i.e. the root), denoted by $\phi$,  gives birth to $N$ children with displacements $A_1,A_2,\cdots,A_N$ they  form  the first generation.  Then, for any integer $n\geq1$, each individual in the $n$-th generation gives birth independently of all others to its own children in the $(n+1)$-th generation. Their displacements are given by independent copies of $\mathcal{L}$. 

We thus obtain a genealogical tree, denoted by $\T$, which is a Galton-Watson tree with offspring $N$. For each vertex (individual or site) $z\in\T$, $A(z)$ denotes its displacement and $V(z)$ its position with respect to the root. If $y$ is the parent of $z$, write $\pa{z}=y$, also if $y$ is an ancestor of $z$, write $y<z$. $V$ can then be written as
\[
V(z)=\sum_{\phi<y\leq z}A(y),
\]
with $V(\phi)=0$. In particular $\mathcal{L}=\{V(z), |z|=1\}$, with $|z|$ the generation of $z$.

The branching random walk $(V(z), z\in\T)$ serves as a  random environment $\en$ (also called random potential).
Conditionally on the environment $\en=(V(z), z\in\T)$, a random walk $(X_n, n \in \N^*, X_0= \phi)$ starting from the root and taking values on the vertices of $\T$ can be defined, with  probabilities of transition:
\begin{align}\label{jump}
p^\en(z,u)=\begin{cases}
\frac{e^{-V(u)}}{e^{-V(z)}+\sum_{v: \pa{v}=z}e^{-V(v)}}, \textrm{ if $u$ is a child of }z, \vspace{0.3cm} \\  
\frac{e^{-V(z)}}{e^{-V(z)}+\sum_{v: \pa{v}=z}e^{-V(v)}}, \textrm{ if $u$ is the parent of }z.
\end{cases}
\end{align}

For convenience, we add a parent $\pa{\phi}$ to the root and assume that \eqref{jump} holds also for $z=\phi$ with $p^\en(\pa{\phi},\phi)=1$.\\
Let $\P$ be the probability measure of the environment and $\P^*$, the probability conditioned on the survival set of the tree $\T$ (which is assumed to be supercritical, see \eqref{hyp0} below). Let $\p^{\en}$, the quenched probability measure of this random walk that is $\p^{\en}(\cdot):=\p(\cdot|\en)$ and $\p(\cdot):=\int\p^{\en(w)}(\cdot)\P(dw)$ the annealed probability measure. Similarly we also define $\p^*$ with respect to $\P^*$. \\
{The walk $(X_n, n \in \N^*, X_0= \phi)$ belongs to the family of biased random walks on a tree first introduced by R. Lyons (\cite{Lyons} and \cite{Lyons2}). In our case where the bias is random, the first references go back to R. Lyons and R. Pemantle \cite{LyonPema} and M.V. Menshikov and D. Petritis \cite{MenPet}. These works give a classification of these random walks on a regular tree in term of recurrence criteria, their results are extended lately for Galton-Watson trees by G. Faraud \cite{Faraud}. This classification which can be determined from the fluctuations of the $\log$-Laplace transform $\psi$ defined below is resumed in Figure \ref{fig2}.  Assume that there exists $ \theta>0$, such that  $\forall s\in [-1,1+ \theta ]$ 
\[\qquad \psi(s):=\log \E\Big(\sum_{|z|=1} e^{-sV(z)} \Big)<+ \infty, \] }
{where $\sum_{|z|=k}$  with $k\in\N_+$ means sum over all the individuals $z$ of generation $k$.}
\begin{figure}[h]
\begin{center} 
{\scalebox{1.2} {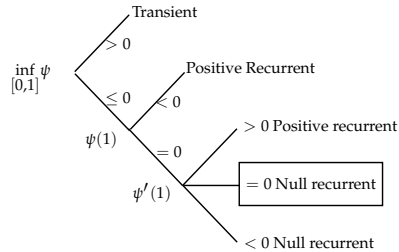 }}
\caption{Recurrence criteria for $(X_n,n)$} \label{fig2}
\end{center}
\end{figure}


In this paper we focus on the boundary case for the environment (in the sense of Biggins-Kyprianou \cite{BigKyp}), that is :
\begin{equation}\label{hyp0}
\E\left[N\right]>1,\quad {\psi(1)}= \log \E\Big[\sum_{|z|=1}e^{-V(z)}\Big]=0,\quad \psi'(1)=\E\Big[\sum_{|z|=1}V(z)e^{-V(z)}\Big]=0.
\end{equation}
{Notice that the first hypothesis $\E\left[N\right]>1$ implies that we work on a supercritical Galton-Watson tree. In particular $(X_n; n\geq0)$ can not be reduced to the one-dimensional random walk in random environment}. Also we need additional hypothesis given below : there exists $\theta>0$ such that
\begin{align}
\E\Big[\sum_{|z|=1}e^{-(1+\theta)V(z)}\Big]+\E\Big[\sum_{|z|=1}e^{\theta V(z)}\Big]<&\infty\label{hyp0+}\\
 \E\Big[\Big(\sum_{|z|=1}(1+|V(u)|)e^{-V(u)}\Big)^2\Big]<&\infty.\label{hyp1}
\end{align}
The hypothesis \eqref{hyp1} will be required in Lemma \ref{mvar}. But the hypothesis \eqref{hyp0+} is more elementary which gives finite exponential moments.

It is proved in \cite{Faraud}, see also Figure \ref{fig2}, that the random walk $X$ is null recurrent under \eqref{hyp0}. Moreover in this case $X$ is very slow, indeed Y. Hu and Z. Shi \cite{HuShi10a} (see also \cite{HuShi10b} with G. Faraud ) proved that the largest generation visited up to time $n$, $X_n^*:= \max_{k \leq n}|X_k|$ behaves in $(\log n)^3$. In fact it is the slowest null recurrent random walk in random environment on the tree, the other cases that is when $\psi'(1)<0$ being diffusive or sub-diffusive but without logarithmic behavior (see \cite{HuShi10}, \cite{Faraud}, \cite{AidDer}). One of the questions raised by the authors at this time was : is $(\log n)^3$ the typical fluctuation of  this walk, that is of $|X_n|$ for example ? If we now look at the largest generation entirely visited $M_n:= \max\{k \geq 1: \{|z|=k\}\subset\{X_i; 0\leq i\leq n\}\}$, then it is of order $\log n$ as shown in P. Andreoletti and P. Debs \cite{AndDeb1}, and we could also ask here the same question. It turns out that neither of the two is the good answer. A first result in that direction is obtained in the work of \cite{AndDeb2}. For any $z\in \T$, define
\begin{equation}
T_z=T^1_z:=\inf\{m\geq1: X_m=z\}\textrm{ and } T_z^k:=\inf\{m\geq T_z^{k-1} : X_m=z\}, \forall k\geq2.
\end{equation}
\noindent Then for any generation $\ell\geq1$, the number of sites visited at this generation up to time $n$ is given by
\[
N_n(\ell):=\sum_{|z|=\ell}\1_{T_z<n}.
\]
We also introduce the same variable stopped at the $n$-th return to the root: 
\[
K_n(\ell):=N_{T_\phi^n}(\ell).
\]

It is proved in \cite{AndDeb2} that the typical generations which maximise the number of distinct visited sites are of the order $(\log n)^2$ :
\begin{equation}\label{AnD}
\lim_{n \rightarrow + \infty} \frac{\e\left(K_n( (\log n)^{2} )\right)}{\e\left(K_n( (\log n)^{1+ \zeta} )\right)}=\infty,\ \forall \zeta \neq 1 \textrm{ and } \e\left(K_n( (\log n)^{2} \right)) \asymp n/ \log n  \ \footnote{In \cite{AndDeb2} the lower bound obtained is actually a little smaller than $n/ \log n$.}. 
\end{equation}

They also notice that only the sites such that the branching potentiel $V(\cdot)$ is high enough  (typically
larger than $\log n$) are of importance. That is to say produce the main contribution for $\e\left(K_n( (\log n)^{2} )\right)$, conversely
the sites with low potential are mostly visited but there are very few of them (typically of order
$n/(\log n)^2$ compared to $n/(\log n)$). More recently, in \cite{HuShi15}, it is proved that $(\log n)^2$ is actually the right
normalisation for the generation of $X$ at the instant $n$, this unexpected behavior makes us think to the one dimensional case of Sinai's walk \cite{Sinai}. However the walk on the tree has its own particularities, for example, contrarily to the one-dimensional case which remains in the site of low potential, it can reach height of potential of order $(\log n)^2$ (see \cite{HuShi15b}). 

Another motivation, as working on the tree, is to understand more precisely the way the walk spread on the tree so we turn back to the number of distinct visited sites. The main lack in the paper \cite{AndDeb2} is first that nothing precise is said on the behavior in probability of $N_n$ (neither for $K_n$), and that their annealed results say few things on the typical behavior of the potentials leading to this critical $(\log n)^2$-th generation. Our results here bring answers to these points.

We have split our results into two parts, the first subsection below deals with the normalization for the number of
distinct visited sites per critical generation as well as for the total number of distinct visited sites up to time $n$. The second subsection is devoted to a quenched results making a
link between the range of X and the behavior of the environment. In a the third subsection we present the key ideas of proofs.

\subsection{Annealed results}
\label{thms}

Our first theorem shows that the behavior in probability of the number of distinct visited sites  at critical generations is of order ${n}/(\log n)^3$.

\begin{theo}\label{upton}For any integers $\ell=\ell(n)$ such that $\lim_{n \rightarrow + \infty} \frac{\ell}{ (\log n)^2}= \gamma>0$, there exists a positive constant $\lambda(\gamma)>0$ such that as $n\rightarrow\infty$,
\begin{equation}
\frac{(\log n)^3}{n}N_n(\ell)\xrightarrow{in\ \p^*} \frac{\lambda(\gamma)\sigma^2}{4},
\end{equation}
where $\sigma^2:=\E\left[\sum_{\vert x\vert=1} V^2(x)e^{-V(x)}\right]\in(0,\infty)$ by \eqref{hyp0+}.
\end{theo}
The function $\lambda(\gamma)$ can be written explicitly (see below \eqref{cvgqK}), it is related to the convergence of variables  depending only on the environment. 
This theorem is the consequence of the behaviors of $K_n$ and of the local time at the root. To be more precise, let us introduce the derivative martingale $(D_m,m)$ given by
\begin{align}
D_m:= \sum_{|z|=m} V(z) e^{-V(z)}, \label{1.7}
\end{align}
and denote its almost sure limit by $D_\infty$ (see \cite{BigKyp} for its existence and \cite{Che15} for its positivity under $\P^*$). The behavior
in probability of $K_n$ is given by
 
\begin{theo}\label{uptoTn}For any $\ell=\ell(n)$ such that $\lim_{n \rightarrow + \infty} \frac{\ell}{ (\log n)^2}= \gamma>0$, 
\begin{equation}
\frac{(\log n)^2}{n}K_n(\ell)\xrightarrow{in\ \p^*}\lambda(\gamma)p^\en(\phi,\overleftarrow{\phi})D_\infty.
\end{equation}
\end{theo}

If we compare this results with the behavior in mean (see \eqref{AnD}),
a multiplicative $(\log n)$ appears. It comes from the behavior of the branching potential which typically remains positive in probability (see \ref{below}) reducing the number of possible visited sites.

\noindent Also the main difference between $N_n(\ell)$ and $K_n(\ell)$ comes essentially from the normalisation. The additional $\log n$ which appears above for $K_n(\ell)$ comes from the local time of $X$ at the root of the tree, it is indeed proved in \cite{HuShi15}: 

\begin{prop}[\cite{HuShi15}] \label{propHuShi}  
\begin{equation}
\frac{T_\phi^n}{n\log n}\xrightarrow{in\ \p^*}4D_\infty p^\en(\phi,\overleftarrow{\phi})/\sigma^2.
\end{equation}
\end{prop}

Instead of one critical generation, we now turn to consider the total number of visited sites, in other words, the range of the random walk:
\[
R_n:=\sum_{z\in\T}\un_{ \{T_z\leq n\} }.
\]

Following \eqref{AnD} and Theorem \ref{upton} we can ask wether or not critical generations contribute
mainly to $R_n$ ? The answer is yes : Proposition \ref{noncritical} below states that for non-critical generations, the total number of visited sites contributes to something negligible compared to $n/\log n$, while the range $R_n$ is of order $n/\log n$ in probability, as stated in Theorem \ref{upton2}.

\begin{prop}\label{noncritical}
For any $\delta>0$,
\[
\lim_{\varepsilon\rightarrow0}\limsup_{n\rightarrow\infty}\p\Big[\Big(\sum_{|z|\leq \varepsilon (\log n)^2}\un_{ \{T_z\leq n\} }+\sum_{|z|\geq (\log n)^2/\varepsilon}\un_{\{T_z\leq n \} }\Big)\geq \delta n/\log n\Big]=0.
\]
\end{prop}

So as the main contribution comes from generations of order $(\log n)^2$, we have that with high probability, $R_n\approx \sum_{\varepsilon(\log n)^2\leq \ell\leq (\log n)^2/\varepsilon}N_n(\ell)$ with $\varepsilon\downarrow 0$. As a consequence we obtain the following result for the range of $X$ :
\begin{theo}\label{upton2} We have
\begin{equation}
\frac{\log n}{n}R_n \xrightarrow{in\ \p^*} \frac{\sigma^2 }{4}\Lambda,
\end{equation}
where $\Lambda:=\int_0^\infty \lambda(\gamma)d\gamma\in(0,\infty)$.
\end{theo}
\begin{rem}\label{remR}
In fact, once again by Proposition \ref{propHuShi}, this theorem follows from the following convergence:
\[
\frac{R_{T_\phi^n}}{n}\xrightarrow{\textrm{ in }\p^*} \Lambda p^{\en}(\phi,\pa{\phi})D_\infty.
\]
Also, the integrability of  $\lambda$ is stated in Lemma \ref{A} of the  Appendix
\end{rem}
\noindent 
These first results give a quantitative description of the number of visited sites and of the generations involved, but no description of the underlying environment is given. In the following section we discuss what we have learnt about the typical behavior of the potential that leads to the above behavior of $R_n$.

\subsection{A quenched point of view}
\label{rough}

Like we said in the first part of the introduction, Andreoletti-Debs \cite{AndDeb2} observe that the sites where the potential remains small (always lower than $\log n$) have a negligible contribution for the number of visited sites. One of the reasons for this is the fact that the number of such sites is actually negligible on the tree (see their Proposition 1.3). Intuitively these sites are easily accessible as the potential remains low, but the set of these sites still has a low conductance.

Here we give some more details of the sites that the random walk is inclined to visit, i.e. the sites that contribute importantly to the range. 

For sites $y, \ z\in\T$, recall that $y\leq z$ means that $y$ belongs to the shortest path from the root $\phi$ to $z$. Let $\overline{V}(z):=\max_{\phi<y\leq z}V(y)$. Define for any $a_0>1$,
\[A_1:= \Big\{z\in\T:\  \frac{\log n}{a_0} \leq \max_{\phi<y\leq z} \left(\overline{V}(y)-V(y) \right) \leq \log n + g(n)  \Big\}, \]
where $\{g(n),n \}$ is a positive increasing function such that $\lim_{n \rightarrow + \infty} (g(n)- \log \log n) = + \infty$. Moreover, for any $a_1>0$, let  
\[A_2:= \left\{z\in\T:\  \log n +\log\log n  \leq \overline{V}(z) \leq a_1 \log n \sqrt{ \log \log n}   \right\}, \]
and 
\[A_3:= \Big\{z\in\T:\ \overline{V}(z)> \max_{\substack{ y\leq z;  |y|\leq |z|-|z|^{1/3}}} V(y) \Big \}.
 \]
Let us introduce a notation for truncated versions of $K_n$, $R_n$ and their quenched mean : if $A$ is an event depending only on the environment $\mathcal{E}$, then for any $\ell\geq1$,
\begin{align}
& K_n^A(\ell):=\sum_{|z|= \ell} \un_{ \{T_z< T^n_{\phi} \} } \un_{\{ z \in A\} },\ R_{T^n_\phi}^A:= \sum_{m \geq 0}K_n^A(m), \label{truncV}  \\
& \mathcal{K}_n^A(\ell):=\mathbb{E}^{\mathcal{E}}\left(K_n^A(\ell) \right),\ \mathscr{R}_{T^n_\phi}^A:= \mathbb{E}^{\mathcal{E}}\left(R_{T^n_\phi}^A \right). \label{kA}
\end{align}

Notice that the above means  are easily computable (see section \ref{sec2}), but we are not interested in their expressions for now. 
The following result proves tightness of the range up to $T_\phi^n$ minus the truncated quenched mean of $R_{T_\phi^n}$: $\mathscr{R}_n^{A_1 \cap A_2 \cap A_3}$, this makes appear  favorite environments described by potential $V$. 
\begin{prop} \label{Prop1.6} For any $\eta >0$, there exists $a_1>0$ such that
\[ \lim_{a_0 \rightarrow + \infty} \limsup_{n \rightarrow + \infty} \p^*\left( \frac{ 1}{n} \left| R_{T_{\phi}^n}- \mathscr{R}_{T_{\phi}^n}^{A_1 \cap A_2 \cap A_3} \right| \geq \eta \right)=0. \]
\end{prop}

From this result together with the well known fact in \cite{Aid13} about the potential : $\P(\inf_{z\in\T}V(z) \geq -\alpha) \geq 1-e^{-\alpha}$, we are able to draw a typical trajectory of potential that maximises the number of visited sites (see Figure \ref{fig1}). 
\begin{figure}[h]
\begin{center}
\scalebox{0.4}{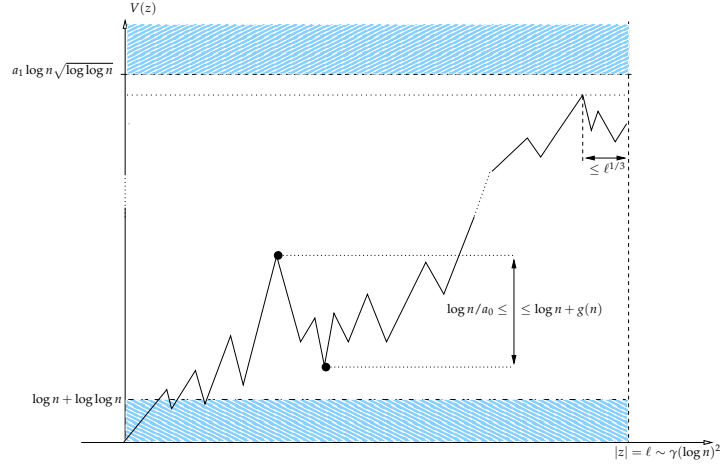}
\caption{Typical accessible environments within time $n$}\label{fig1}
\end{center}
\end{figure}

\subsection{Sketch of proofs and organization of the paper}
\label{cvgofquenched}
As we have already seen, Theorem \ref{upton}, comes from Proposition \ref{propHuShi} together with Remark \ref{remR}, so only the remark as to be proved. Also thanks to Proposition \ref{noncritical} (which is a consequence of Lemma 2.4 for which the proof is posponed in Section \ref{Sec4.3.4}) together with  Proposition \ref{propHuShi}, only the critical generations of order $(\log n)^2$ have to be considered. For that we first study individually each of these generations which is the purpose of Theorem \ref{uptoTn} : \\
{\bf Skech of proof of Theorem \ref{uptoTn}} : 
The first step for the study of $K_n(\ell)$ for $\ell\sim \gamma(\log n)^2$ is to compare it with its quenched expectation $\mathcal{K}_n(\ell):=\e^{\en}[K_n(\ell)]$. The main idea here is simple : we would like to apply  Tchebychev's inequality to the quenched probability $\p^{\en}\left(|K_n(\ell)-\mathcal{K}_n(\ell)|\geq \varepsilon \mathcal{K}_n(\ell)\right)$. Unfortunately this gives nothing usable if we do it directly. Indeed it turns out that the quenched variance $\Ve(K_n(\ell))$ which appears when applying this inequality can not be controlled properly with respect to measure $\P^*$.  In order to overcome this, we add restrictions to the environment, the first one comes from the reflecting barrier introduced by \cite{HuShi15} : let $\delta>0$ introduce
\(
 L_\delta:=\{z\in\T: \max_{\phi<y\leq z}(\overline{V}(y)-V(y))\leq \log n-(1+\delta)\log \log n\}.
\)
In other words, we consider the restriction of $K_n$ to the sites of $ L_\delta$, that is to say $K_n^{L_\delta}(\ell)=\sum_{|z|=\ell}\un_{ \{T_z<T_\phi^n \}}\un_{\{z\in L_\delta \}}$ and its quenched mean 
\begin{align}
 \mathcal{K}^{L_\delta}_n(\ell) & :=\sum_{|z|=\ell}\p^{\en}(T_z<T_\phi^n)\un_{ \{ z\in L_\delta\} }=\sum_{|z|=\ell} (1-(1-a_z)^n)\un_{ \{ z\in L_\delta\}}, \textrm{ where} \nonumber \\
  a_z & :  =\p^{\mathcal{E}}_\phi(T_z<T_\phi)=p^{\en}(\phi,z_1)\p^\en_{z_1}(T_z<T_\phi)=\frac{p^\en(\phi,\pa{\phi})}{\sum_{\phi<y\leq z}e^{V(y)}}, \label{az}
\end{align}
obtained by the strong Markov property, also the last equality in \ref{az} comes from Lemma C.1 in \cite{AndDeb2}. Then, following the ideas of \cite{AndDeb2}, we add a second restriction by defining the set $U:=\{z\in\T: \overline{V}(z)\geq \log n+\log\log n\}$. This restriction, which comes from the fact that only sites with a high level of potential count, contributes to a simplification of the expression of the quenched mean defined above: for any $z\in U$,  $a_z\leq e^{-\overline{V}(z)}\leq \frac{1}{n\log n}$, thus
\begin{equation}\label{naz}
0\leq na_z-[ 1-(1-a_z)^n] \leq n^2 a_z^2 \leq \frac{1}{\log n} na_z, 
\end{equation}
so in particular $1-(1-a_z)^n=(1+o_n(1))na_z$, and for any event $D \subseteq \{z\in\T\ :\  \overline{V}(z) \geq \log n+ \log \log n\}$ depending only on the environment 
\begin{align} 
 {\mathcal{K}}_n^{D}(\ell)= \mathbb{E}^{\mathcal{E}}\left[ \sum_{|z|= \ell} \un_{ \{T_z< T^n_{\phi}\} } \un_{\{ z \in D \} } \right]  \sim n    \sum_{|z|= \ell} a_z  \un_{\{ z \in D \} }=:\widetilde{\mathcal{K}}^D_n(\ell). \label{mathscK}
\end{align}
We prove rigorously, in Subsection \ref{Qexp}, that the cost of these restrictions $L_\delta\cap U$ is negligible for the number of distinct  visited sites before $n$ return to the origin (see Lemmata \ref{addU} and \ref{usesn}). So we are left to study the restriction ${K}_n^{U\cap L_\delta}(\ell)$. For that we apply Tchebychev's inequality (see Section \ref{sec2.2}) and, thanks to the restriction, the expectation with respect to measure $\P^*$ of the quenched variance (Section \ref{QuVar}) $\Ve({K}_n^{U\cap L_\delta}(\ell))$ is well controlled. Finally we obtain that in probability $K_n(\ell)$ can be approximated for large $n$ by $\widetilde{\mathcal{K}}_n^{U\cap L_\delta}(\ell)$ :
\begin{align} K_n(\ell) \overset{\p}{ \sim  } 
\widetilde{\mathcal{K}}_n^{U\cap L_\delta}(\ell)=\sum_{|z|=\ell}na_z\un_{ \{ z\in L_\delta\cap U \}}=np^{\en}(\phi,\overleftarrow{\phi})\times\sum_{|z|=\ell}e^{-V(z)}\frac{e^{V(z)}}{\sum_{\phi<y\leq z}e^{V(y)}}\un_{\{ z\in U\cap L_\delta \}}. \label{cvprobaK}
\end{align}
The second step is to obtain the convergence of ${(\log n)^2}\widetilde{\mathcal{K}}_n^{U\cap L_\delta}(\ell)/n$ to some non trivial limit under $\P^*$. 
For that we introduce the following martingale-like variable, for  any $m\geq1$ 
 and $a,b\geq0$,
\begin{align} 
W_m(F_{a,b}):=&\sum_{|z|=m}e^{-V(z)}F_{a,b}(z), \textrm{ where } \nonumber\\
F_{a,b}(z):=&\sqrt{m}\frac{e^{V(z)}}{\sum_{\phi<y\leq z}e^{V(y)}}\un_{ \{ \overline{V}(z)\geq b \} }\un_{\{ \max_{\phi<y\leq z  }(\overline{V}(y)-V(y))\leq a \}}. \label{defF}
\end{align}
With this notation $\widetilde{\mathcal{K}}_n^{U\cap L_\delta}(\ell)$ can be re-write, 
\begin{equation}\label{KtoW}
\widetilde{\mathcal{K}}_n^{U\cap L_\delta}(\ell)=\frac{np^{\en}(\phi,\overleftarrow{\phi})}{\sqrt{\ell}}W_\ell(F_{\log n-(1+\delta)\log\log n, \log n+\log\log n}).
\end{equation}

Notice that if $F_{a,b}(z)=1$ for any of its arguments, then $W_m(F_{a,b})$ is exactly the well-known additive martingale $W_m:=\sum_{|z|=m}e^{-V(z)}$. Aid\'ekon and Shi \cite{AidekonShi09} showed that $\sqrt{m}W_m$ converge in $\P^*$-probability   to the positive martingale $D_\infty=:\lim_{m\rightarrow\infty}\sum_{|z|=m}V(z)e^{-V(z)}$ . More recently Madaule \cite{Madaule} proved that if one chooses one site $z$ at the $m$-th generation, according to the measure ${e^{-V(z)}}/{W_m}$, the corresponding rescaled trajectory $({V(y)\un_{\{|y|=\lfloor mt\rfloor, y\leq z\}}}/{\sqrt{m}})_{ 0\leq t\leq 1}$ is asymptotically a Brownian meander. \\
Unfortunately in our case $F_{a_m,b_m}(z)$ is not simply a functional of this rescaled trajectory, so their results cannot be applied directly. However, our proof of Proposition \ref{cvgProp} below (see Section \ref{sec3.2}) is mainly inspired by their arguments. \bl{We are going to take $a=O(\sqrt{m})$, and the factor $\sqrt{m}$ is used to ``balance'' $\frac{e^{V(z)}}{\sum_{\phi<y\leq z}e^{V(y)}}\un_{\overline{V}(z)\geq b}$.}
\begin{prop} \label{cvgProp}
If $(a_m; m\geq0)$ and $(b_m; m\geq0)$ are positive sequences such that $\lim_{m\rightarrow\infty}\frac{a_m}{\sqrt{m}}=a\in\r_+^*$ and $\lim_{m\rightarrow\infty}\frac{b_m}{\sqrt{m}}=b\in\r_+$, then as $m\rightarrow\infty$, there exists $\mathscr{C}_{a,b}\in(0,\infty)$ such that
\begin{equation}\label{cvgP}
\sqrt{m} W_m(F_{a_m,b_m})\xrightarrow[m\rightarrow\infty]{in \ \P^*} \mathscr{C}_{a,b}D_\infty.
\end{equation}
see \eqref{1.7} for definition of $D_\infty$. $\mathscr{C}_{a,b}$, which definition is given in Section \ref{sec3.2}, is continuous,  increasing in $a$ and decreasing in $b$, and we state that $\mathscr{C}_{0,b}=0$.
\end{prop}
\noindent With this result we obtain the convergence of the quenched random variable $\widetilde{\mathcal{K}}_n^{U\cap L_\delta}(\ell)$ : for $\ell\sim\gamma(\log n)^2$ and any $\delta>0$, by \eqref{KtoW} and \eqref{cvgP}, as $n\rightarrow\infty$,
\begin{equation}\label{cvgqK}
\ell\frac{\widetilde{\mathcal{K}}_n^{U\cap L_\delta}(\ell)}{np^{\en}(\phi,\overleftarrow{\phi})}\xrightarrow{in \ \P^*} \mathscr{C}_{\gamma^{-1/2},\gamma^{-1/2}}D_\infty .
\end{equation}
Then \eqref{cvprobaK} and \eqref{cvgqK} yield  Theorem \ref{uptoTn}, with $\lambda(\gamma):=\gamma ^{-1}{\mathscr{C}_{\gamma^{-1/2},\gamma^{-1/2}}}$. More details about the properties of $\lambda$ are given in Lemma \ref{A}.


\noindent {\bf Final ideas for the proof of Remark \ref{remR}}
As $\mathscr{C}_{a,b}$ is continuous and monotone on $(a,b)\in\r_+^2$ and that $W_m(F_{a\sqrt{m},b\sqrt{m}})$ is also monotone on $(a,b)\in\r_+^2$. It follows that \eqref{cvgP} holds uniformly for $W_m(F_{a\sqrt{m},b\sqrt{m}})$ in $(a,b)\in\r_+^2$ in the following sense: for any $\varepsilon>0$,
\begin{equation}\label{unifcvg}
\lim_{m\rightarrow\infty}\P^*\left(\sup_{a\geq0,b\geq0}\Big\vert \sqrt{m}W_m(F_{a\sqrt{m},b\sqrt{m}})- \mathscr{C}_{a,b}D_\infty\Big\vert\geq \varepsilon\right)=0.
\end{equation}
This induces the following corollary which proof can be found Section \ref{sec3.3}.

\begin{cor}\label{sumW}
\begin{equation}
\lim_{\beta\rightarrow\infty}\sum_{m=1}^\infty \frac{W_m(F_{\beta,\beta})}{\sqrt{m}}=\Lambda D_\infty,\textrm{ in } \P^*\textrm{-probability with}\ \Lambda=\int_0^\infty \mathscr{C}_{\frac{1}{\sqrt{x}},\frac{1}{\sqrt{x}}}\frac{dx}{x}.
\end{equation}
\end{cor}
This corollary still holds if we replace $F_{\beta,\beta}$ by $F_{\beta\pm O(\log \beta),\beta}$ in the sum. This result brings out Remark \ref{remR} and therefore Theorem \ref{upton2}.

\begin{rem}
\eqref{unifcvg} suggests that uniformity may also occur in probability for $K_n(\ell)$, meaning that the "for any $\ell$" in Theorem \ref{upton} could actually be placed  inside the probability. Unfortunately, this uniformity can not be obtained by the way of our proofs and we believe in fact that this is not true and that the right normalisation for $\max_{ \ell }N_n( \ell)$ could be different from $n/ (\log n)^3$.
\end{rem}

\vspace{0.5cm}
\noindent  The rest of the paper is organized as follows :

 \textit{In Section \ref{sec2}}  we use results of Sections  \ref{sec3} and \ref{sec4} to give the main steps of the proofs of  theorems and propositions stated in Section \ref{thms}. \textit{In Section \ref{sec3}} we focus on the environment and show Proposition \ref{cvgProp} and Corollary \ref{sumW}. This section is independent of the other sections and uses only the Appendix. \textit{In Section \ref{sec4}} we compute the annealed mean of $K_n$ and give an upper bound for the mean of the quenched variance. Also we prove lemmata used in section \ref{sec2} and finish with the proof of Proposition \ref{Prop1.6}. In the \textit{Appendix} we collect and prove many estimations for centered random walk with i.i.d increments.\\[1pt]

{In this paper, we use $c$ or $c'$ for constants which may change from line to line. We write $c(x)$ when that constant depends on some parameter $x$.}
\section{Proof of the theorems}
\label{sec2}

This section is devoted to proving Theorems \ref{uptoTn} and \ref{upton2}, i.e. the convergence in probability of $K_n(\ell)$. Theorem \ref{upton} follows immediately from Theorem \ref{uptoTn} and Proposition \ref{propHuShi}, so we feel free to omit its proof. Recall that for convenience, we fixe some $\gamma\in(0,\infty)$ and always write $\ell$ for the integer sequence $\{\ell(n); n\geq1\}$ such that 
\[
\lim_{n \rightarrow + \infty} {\ell(n)}/{ (\log n)^2}= \gamma.
\]

 Our arguments are based on the study of truncated versions of $K_n$. This decomposition of $K_n$ appears naturally when computing the  mean of $K_n$ as well as the mean of its quenched variance. We therefore start with this decomposition.

\subsection{Quenched expectation and truncated versions of $K_n(\ell)$ \label{Qexp}}
For any measurable event $C$ obtained from the environment, the number of visited sites at generation $\ell$ up to the $n$-th return to $\phi$ can be written as
\[
K_n(\ell)=K^C_n(\ell)+K^{C^c}_n(\ell)=\sum_{|z|=\ell} \un_{ \{ T_z<T_\phi^n \} }\un_{\{ z\in C\}}+\sum_{|z|=\ell} \un_{\{T_z<T_\phi^n \}}\un_{\{z\in C^c \}}.
\]
To exclude the sites in $C$ that make {few} contribution to $K_n$, we add restrictions for the potentials on the above sum.  First  (see \cite{Aid13}) for any $\varepsilon>0$, we can choose $\alpha>0$ such that
\begin{equation}\label{below}
\P\left(\inf_{u\in \T}V(u)<-\alpha \right)\leq e^{-\alpha}\leq \varepsilon.
\end{equation}
Let $\underline{V}(z):=\min_{\phi<y\leq z}V(y)$, it is then natural to consider the set 
\begin{equation*}
B_1:=\{z\in \T: \underline{V}(z)\geq-\alpha\}.
\end{equation*}
{Secondly}, in \cite{HuShi15}, a reflecting barrier is introduced by
\begin{equation*}
\mathcal{L}_r:=\Big\{z\in\T:\sum_{\phi<u\leq z}e^{V(u)-V(z)}>r, \max_{\phi <y< z}\sum_{\phi <u\leq y}e^{V(u)-V(y)}\leq r\Big\}\textrm{ with } r>0.
\end{equation*}
This reflecting barrier allows to reduce the number of interesting sites for the walk in the following sense : let $f$ be a positive increasing function such that $\lim_{n \rightarrow + \infty} f(n)= + \infty$, then
\begin{align} \lim_{n \rightarrow + \infty} \p\left( \exists k \leq T_{\phi}^n, X_k  \in \mathcal{L}_{ \frac{n f(n)}{\log n}} \right)=0. \label{theoHuShi} 
\end{align}
The above result is a direct consequence of Theorem 2.8 (in \cite{HuShi15}) together with Proposition 1.3. Following this idea, we introduce the following sets 
 \[ B_2:= \left\{z\in\T: \max_{\phi <y \leq z} \sum_{\phi<u\leq y}e^{V(u)-V(y)} \leq n \right \}=:\{z\in\T: z<\mathcal{L}_{n}\}, \]
 then according to \eqref{theoHuShi} 
\begin{align*} \lim_{n \rightarrow + \infty} \P\left( \forall k \leq T_{\phi}^n, X_k \in B_2 \right)=1. 
\end{align*}
Also, for any $\delta>0$ let $s_n:=n (\log n)^{-1-\delta} $ and
\[ B_2^\delta:= \left\{z\in\T: \max_{\phi <y \leq z} \sum_{\phi<u\leq y}e^{V(u)-V(y)} \leq s_n\right \}=\{z\in\T: z<\mathcal{L}_{s_n}\}.\]
We will see that for our specific problem, we can restrict the set $B_2$ to $B_2^{\delta}$ for well chosen $\delta$. For convenience, denote $B:=B_1 \cap B_2$ and $B^{\delta}:=B_1 \cap B_2^{\delta}$. Because of \eqref{below} and \eqref{theoHuShi}, one sees that with high probability, $K_n(\ell)\sim K_n^B(\ell)=\sum_{|z|=\ell}\un_{ \{T_z<T_\phi^n\} }\un_{\{z\in B\} }$. Moreover, if $z\in B_2^\delta$, we have $z\in L_\delta$ (recall the definition of  $L_\delta$ just above \ref{az}), and conversely, if $z\in L_{\delta+2}$, then $z\in B_2^{\delta}$.

Also we add  the last restriction over the values of $\overline{V}$: $U= \{z\in\T:\ \overline{V}(z)\geq \log n+\log\log n\}$. The following lemma shows that the cost of this restriction is negligible. 
\begin{lem}\label{addU}
\begin{equation}
\e\Big[K_n^{B\setminus U}(\ell)\Big]=o\left(\e[K^{B\cap U}_n(\ell)]\right)=o\left(\frac{n}{(\log n)^2}\right). \label{sansB3} 
\end{equation}
\end{lem}
Our arguments will show indeed that $\e[K^{B\cap U}_n(\ell)]=\Theta(\frac{n}{(\log n)^2})$, so that the sites in $B\cap U$ mainly contribute. We postpone the proof of this lemma to Section \ref{Sec4.2}.

Here is our strategy to obtain Theorem \ref{upton}. We first show that for suitable $\delta>0$, with high probability, $K_n(\ell)\approx K_n^{B\cap U}(\ell)\approx K_n^{B^\delta\cap U}(\ell)$, while the last quantity can be approached by its quenched mean by bounding its quenched variance. This observation combined with the fact that the quenched mean $\mathcal{K}^{B^\delta\cap U}_n(\ell)$ converges in probability because of Proposition \ref{cvgProp}, imply our theorem.

We stress on the fact that replacement of $B$ by $B^\delta$ helps to correctly bound the quenched variance, it appears that the price of this replacement is negligible, as {shown} in the following Lemma: 

\begin{lem}\label{usesn}
For any $\delta>0$, we have
\begin{equation}
\e[K_n^{B\cap U}(\ell)-K_n^{B^\delta\cap U}(\ell)]=o \left(\frac{n}{(\log n)^2}\right). \label{C1}
\end{equation}
\end{lem}

The next step is to approach $K_n^{B^\delta\cap U}(\ell)$ by its quenched mean $\mathcal{K}_n^{B^\delta\cap U}(\ell)$, or more conveniently by $\widetilde{\mathcal{K}}_n^{B^\delta\cap U}(\ell)=\sum_{|z|=\ell}na_z\un_{\{ z\in B^\delta\cap U\} }$. Notice indeed that, in view of \eqref{naz}, we have
\begin{equation}\label{nazK}
0\leq \widetilde{\mathcal{K}}_n^{B^\delta\cap U}(\ell)-\mathcal{K}_n^{B^\delta\cap U}(\ell)\leq \frac{1}{\log n} \widetilde{\mathcal{K}}_n^{B^\delta\cap U}(\ell).
\end{equation}
\begin{prop} \label{lemCentredK}
For any $\eta>0$ and $\delta>3$,
\begin{align}
\lim_{n \rightarrow +\infty} \p\left(|K_n^{B^\delta\cap U}(\ell)-\widetilde{\mathcal{K}}_n^{B^\delta\cap U}(\ell)|\geq \eta \frac{n}{(\log n)^2}\right) = 0. \label{CentredK}
\end{align}
\end{prop}
\noindent The proof of  this proposition can be found in Section \ref{Sec4.2}, now with these restrictions introduced, we are ready to prove the theorems.

\subsection{Convergence of $K_n(\ell)$ and $R_n$ \label{sec2.2}: proofs of Theorems \ref{uptoTn} and \ref{upton2}}
We are now ready to prove Theorem \ref{uptoTn}: it suffices to show that for any $\eta>0$,
\begin{equation}\label{cvgupton}
\lim_{n \rightarrow +\infty} \p\Big(\Big\vert \frac{(\log n)^2}{n}K_n(\ell)-\lambda(\gamma)p^\en(\phi,\overleftarrow{\phi})D_\infty\Big\vert\geq \eta\Big)= 0.
\end{equation}

\begin{proof}[Proof of \eqref{cvgupton}]Let ${\bf p}_n:=\p\Big(\Big\vert \frac{(\log n)^2}{n}K_n(\ell)-\lambda(\gamma)p^\en(\phi,\overleftarrow{\phi})D_\infty\Big\vert\geq \eta\Big)$. We first add the restrictions $B_1$ and $B_2$ ( recalling that $B=B_1\cap B_2$). For that let us introduce the events $\cB_1:=\{ \inf_{u\in \T}V(u) \geq -\alpha \}$, $\cB_2:=\{ \bigcap_{i=1}^n\{X_i <\mathcal{L}_{n}\} \}$ and for any $x>0$ and random variable $H$, $\cB_3(H,x):=\Big\{ \Big\vert \frac{(\log n)^2}{n}H -\lambda(\gamma)p^\en(\phi,\overleftarrow{\phi})D_\infty\Big\vert  \geq x \Big\}$. We have $
{\bf p}_n\leq \P\left( \overline{\cB}_1\right)  +\p\left(\overline{\cB}_2 \right)+\p(\cB_3(K_n^B(\ell),\eta ))$. That is to say using \eqref{theoHuShi}, $\limsup_{n\rightarrow\infty}{\bf p}_n \leq\P\left(\overline{\cB}_1 \right)  +\limsup_n\p(\cB_3(K_n^B(\ell),\eta ))$.
For the second term on the right hand side of the previous inequality, we involve the restrictions $B_2^\delta$ and $U$, it then follows from  \eqref{sansB3} and \eqref{C1}  that  $\limsup_{n\rightarrow\infty}{\bf p}_n \leq \P\left(\overline{\cB}_1  \right)  +\limsup_n\p(\cB_3(K_n^{B^\delta\cap U}(\ell),\eta/2)$. Then by Proposition \ref{lemCentredK}, we can use $\widetilde{\mathcal{K}}_n^{B^\delta\cap U}(\ell)$ to approach $K_n^{B^\delta\cap U}(\ell)$ and obtain that  $\limsup_{n\rightarrow\infty}{\bf p}_n \leq \P\left(\overline{\cB}_1\right)  +\limsup_n\P(\cB_3(\widetilde{\mathcal{K}}_n^{B^\delta\cap U}(\ell),\eta/4 ))$.
By releasing the restriction $B_1$, one gets that $\limsup_{n\rightarrow\infty}{\bf p}_n \leq 2\P\left(\overline{\cB}_1\right)  +\limsup_n\P(\cB_3(\widetilde{\mathcal{K}}_n^{B_2^\delta\cap U}(\ell),\eta/4 ) $.
Recall that by definition (see above \eqref{az}, and below \eqref{theoHuShi}) $L_{\delta+2}\subset B_2^\delta\subset L_{\delta}$ so clearly $\widetilde{\mathcal{K}}_n^{L_{\delta+2}\cap U}(\ell)\leq \widetilde{\mathcal{K}}_n^{B_2^\delta\cap U}(\ell)\leq \widetilde{\mathcal{K}}_n^{L_\delta\cap U}(\ell)$. So by \eqref{cvgqK} $\limsup_{n\rightarrow\infty}{\bf p}_n \leq 2\P( \overline{\cB}_1)$. Letting $\alpha\uparrow \infty$, we deduce \eqref{cvgupton} from \eqref{below}. 
\end{proof}
It remains to show the convergence of the range $R_n$, that is Theorem \ref{upton2}. As mentioned in Remark \ref{remR}, by Proposition \ref{propHuShi}, we only need to prove that 
\begin{equation}\label{uptoTn2}
\frac{R_{T_\phi^n}}{n}\xrightarrow{\textrm{ in }\p^*} \Lambda p^{\en}(\phi,\pa{\phi})D_\infty, 
\end{equation}
with $R_{T_\phi^n}=\sum_{m=0}^\infty K_n(m)$. First, we claim that only the critical generations really count in this sum, and that the truncated version of $(K_n(m),m)$ make the main contribution :
{
\begin{lem}\label{othergenerations}
We have
\begin{equation}
\lim_{\epsilon\downarrow 0}\limsup_{n\rightarrow\infty}\frac{1}{n}\Big\{\e\Big[\sum_{m\leq \epsilon(\log n)^2}K_n^B(m)\Big]+\e\Big[\sum_{m\geq (\log n)^2/\epsilon}K_n^B(m)\Big]\Big\}=0, \label{outside}
\end{equation}
and for any $\epsilon>0$,
\begin{align}
\lim_{n\rightarrow\infty}\frac{1}{n}\Big\{ \e \Big[\sum_{m=\varepsilon(\log n)^2}^{(\log n)^2/\varepsilon}K_n^{B\setminus U}(m)\Big]+\e\Big[\sum_{m=\varepsilon(\log n)^2}^{(\log n)^2/\varepsilon} K_n^{(B\cap U)\setminus (B^\delta\cap U)}(m)\Big] \Big\}=0. \label{inside}
\end{align}
\end{lem}
}
The proof of this Lemma is postponed in Section \ref{Sec4.3.4}.  Notice here that Proposition \ref{noncritical} follows from \eqref{outside} and Proposition \ref{propHuShi}. As non-critical generations are negligible, we can borrow the previous arguments for $K_n(\ell)$ to show the convergence for $R_{T^n_\phi}$.

\begin{proof}[Proof of Theorem \ref{upton2} (i.e. \eqref{uptoTn2})]
For any $\eta>0$, let us consider $\p\Big(|R_{T_\phi^n}-\Lambda p^\en(\phi,\overleftarrow{\phi})D_\infty n|\geq \eta n\Big)$. Considering restrictions $B_1$ and $B_2$, one sees that for any $\alpha>0$,
\begin{align*}
\p\Big(|R_{T_\phi^n}-\Lambda p^\en(\phi,\overleftarrow{\phi})D_\infty n|\geq \eta n\Big)\leq \p(\overline \cB_1 )+\p(\overline \cB_2 )+\p\Big(\Big\vert \sum_{m=0}^\infty K_n^B(m)-\Lambda p^\en(\phi,\overleftarrow{\phi})D_\infty n\Big\vert\geq \eta n \Big).
\end{align*}
By \eqref{below} and \eqref{theoHuShi},
\begin{multline}\label{addB}
\limsup_{n\rightarrow\infty}\p\Big(|R_{T_\phi^n}-\Lambda p^\en(\phi,\overleftarrow{\phi})D_\infty n|\geq \eta n\Big)\leq e^{-\alpha}+\\
\limsup_{n\rightarrow\infty}\p\left(\Big\vert \sum_{m=0}^\infty K_n^B(m)-\Lambda p^\en(\phi,\overleftarrow{\phi})D_\infty n\Big\vert\geq \eta n \right).
\end{multline}
For the $K_n(m)$, we only need to consider the generations $m$ of order $(\log n)^2$. For any $\varepsilon>0$, define for any $x>0$ and random variables $(H(m),m \geq 0)$, $\cB_4(H,x):=\Big\{ \vert \sum_{m=\varepsilon(\log n)^2}^{(\log n)^2/\varepsilon} H(m)-\Lambda p^\en(\phi,\overleftarrow{\phi})D_\infty n\Big\vert\geq  x n \Big\}$, we have
\begin{align*}
&\p\left(\Big\vert \sum_{m=1}^\infty K_n^B(m)-\Lambda p^\en(\phi,\overleftarrow{\phi})D_\infty n\Big\vert\geq \eta n \right) 
\leq  \p\Big(\sum_{\substack{m\geq (\log n)^2/\varepsilon,\\\textrm{ or }m\leq \varepsilon(\log n)^2}}K_n^B(m)\geq \eta n/2\Big)+\p(\cB_4( K_n^B, \eta/2) ),
\end{align*}
{where the first probability on the right hand side is negligible because of \eqref{outside}}. For the second probability, we consider only the sites $z\in B^\delta\cap U$ and obtain that
\begin{align*}
\p  (\cB_4( K_n^B(m), \eta/2) ) 
& \leq \p\Big(\Big\vert\sum_{m=\varepsilon(\log n)^2}^{(\log n)^2/\varepsilon}K_n^{B\setminus U}(m)\Big\vert\geq \eta n/6\Big)+\p\Big(\Big\vert\sum_{m=\varepsilon(\log n)^2}^{(\log n)^2/\varepsilon} K_n^{(B\cap U)\setminus (B^\delta\cap U)}(m)\Big\vert\geq \eta n/6\Big)\\
& +\p\left(\cB_4( K_n^{B^\delta\cap U}, \eta/6) \right).
\end{align*}
In view of \eqref{inside} together with \eqref{outside}, we obtain that
\begin{multline}\label{sumell}
\limsup_{n\rightarrow\infty}\p\left(\Big\vert \sum_{m=0}^\infty K_n^B(m)-\Lambda p^\en(\phi,\overleftarrow{\phi})D_\infty n\Big\vert\geq \eta n \right)\leq o_\epsilon(1)+\limsup_{n\rightarrow\infty}\p(\cB_4( K_n^{B^\delta\cap U}, \eta/6) ).
\end{multline}
It remains to bound the second term on the right hand side. Recall that the quenched expectation of $K_n^{B^\delta\cap U}(m)$ is denoted $\mathcal{K}_n^{B^\delta\cap U}(m)$, introducing the variable $\Delta_n(H,G):=\sum_{m=\varepsilon(\log n)^2}^{(\log n)^2/\varepsilon}  \Big\vert H(m)-G(m)\Big\vert$ for any sequences $(H(m),G(m), m\geq 0)$, we can write 
\begin{align}\label{rhsR}
\p(\cB_4( K_n^{B^\delta\cap U}, \eta/6)) 
\leq \p(\Delta_n( K_n^{B^\delta\cap U},\mathcal{K}_n^{B^\delta\cap U}) \vert\geq {\eta n}/{12} )
+\P(\cB_4( \mathcal{K}_n^{B^\delta\cap U}, {\eta}/{12} ) ). 
\end{align}
First, by Markov inequality, $\p\left(\Delta_n( K_n^{B^\delta\cap U},\mathcal{K}_n^{B^\delta\cap U}) \vert\geq {\eta n}/{12}   \right)\leq {144}({\eta^2n^2})^{-1}\e\left[\Big(\Delta_n( K_n^{B^\delta\cap U},\mathcal{K}_n^{B^\delta\cap U})\Big)^2\right]$, which by Cauchy-Schwartz inequality is bounded by {
\[
\frac{144}{\eta^2 n^2} \sum_{m=\varepsilon(\log n)^2}^{(\log n)^2/\varepsilon} 1\sum_{m=\varepsilon(\log n)^2}^{(\log n)^2/\varepsilon}\E\left(\V^\mathcal{E}(K_n^{B^\delta\cap U}(m))\right).
\] }
Applying Lemma \ref{mvar} with $\delta>5$ to this term implies that
\begin{equation}\label{rhsR1}
\limsup_{n\rightarrow\infty}\p\left(\Delta_n( K_n^{B^\delta\cap U},\mathcal{K}_n^{B^\delta\cap U}) \geq \frac{\eta n}{12}\right)=0.
\end{equation}
Second, by {replacing} $\mathcal{K}_n^{B^\delta\cap U}$ by $\widetilde{\mathcal{K}}_n^{B^\delta\cap U}$ {(recall the definition of $\widetilde{\mathcal{K}}$  in \eqref{mathscK})}, one sees that
\begin{align*}
\P(\cB_4( \mathcal{K}_n^{B^\delta\cap U}, {\eta}/{12} ) ) 
\leq &\P\left( \Delta_n( \mathcal{K}_n^{B^\delta\cap U}(m),\widetilde{\mathcal{K}}_n^{B^\delta\cap U}) \geq {\eta n}/{24}\right)
+\P(\cB_4( \widetilde{\mathcal{K}}_n^{B^\delta\cap U}, {\eta}/{24} ) ) \\
\leq& \P\left( \sum_{m=\varepsilon(\log n)^2}^{(\log n)^2/\varepsilon}\frac{1}{\log n}\widetilde{\mathcal{K}}_n^{B^\delta\cap U}(m)\geq \frac{\eta n}{24}\right)
+\P(\cB_4( \widetilde{\mathcal{K}}_n^{B^\delta\cap U}, {\eta}/{24} ) ), \\
 =: & RH_1
+\P(\cB_4( \widetilde{\mathcal{K}}_n^{B^\delta\cap U}, {\eta}/{24} ) )
\end{align*}
where the last inequality follows from \eqref{nazK}. Plugging this inequality and \eqref{rhsR1} into \eqref{rhsR} yields 
\begin{align*}
\p(\cB_4( K_n^{B^\delta\cap U}(m), \eta/6)) \leq
o_n(1)+RH_1+\P\left(\cB_4( \widetilde{\mathcal{K}}_n^{B^\delta\cap U}, {\eta}/{24})\right).
\end{align*}
Now, observe that
\begin{align*}
RH_1 \leq \P\left(\cB_4( \widetilde{\mathcal{K}}_n^{B^\delta\cap U}, {\eta}/{24}) \right)+\P\left(\Lambda p^\en(\phi,\overleftarrow{\phi})D_\infty n \geq \eta n (\log n-1)/24 \right),
\end{align*}
where the second probability on the right hand side vanishes as $n\rightarrow\infty$ because $p^\en(\phi,\overleftarrow{\phi})D_\infty$ is finite $\P$-a.s. So {moving back} to \eqref{sumell}, we deduce that
\begin{multline}\label{sumellB}
\limsup_{n\rightarrow\infty}\p\left(\Big\vert \sum_{m=1}^\infty K_n^B(m)-\Lambda p^\en(\phi,\overleftarrow{\phi})D_\infty n\Big\vert\geq \eta n \right)\leq o_\epsilon(1)+2\P\left(\cB_4( \widetilde{\mathcal{K}}_n^{B^\delta\cap U}, {\eta}/{24})\right). \\
\end{multline}
So in view of \eqref{addB} and \eqref{sumellB}, we have
\begin{align}\label{addB+}
\limsup_{n\rightarrow\infty}\p\Big(|R_{T_\phi^n}-\Lambda p^\en(\phi,\overleftarrow{\phi})D_\infty n|\geq \eta n\Big)
\leq e^{-\alpha}+o_\varepsilon(1)+2\P\left(\cB_4( \widetilde{\mathcal{K}}_n^{B^\delta\cap U}, {\eta}/{24})\right).
\end{align}
We claim here that
\begin{equation}\label{cvgsumK}
\limsup_{\epsilon\downarrow0}\limsup_{n\rightarrow\infty}2\P\left(\cB_4( \widetilde{\mathcal{K}}_n^{B^\delta\cap U}, {\eta}/{24})\right) \leq e^{-\alpha}, 
\end{equation}
which, together with \eqref{addB+} concludes the convergence in probability of $R_{T_\phi^n}/n$ by letting $\alpha\rightarrow\infty$ and $\epsilon \rightarrow 0$. It remains to show \eqref{cvgsumK}. We observe that on $\{\inf_{u\in\T}V(u)\geq-\alpha\}$, $\widetilde{\mathcal{K}}_n^{B^\delta\cap U}(m)=\widetilde{\mathcal{K}}_n^{ B_2^\delta\cap U}(m)$ for any $m\geq0$, hence
 $\P\left(\cB_4( \widetilde{\mathcal{K}}_n^{B^\delta\cap U}, {\eta}/{24})\right)  \leq\P(\overline \cB_1 )+\P\left(\cB_4( \widetilde{\mathcal{K}}_n^{B_2^\delta\cap U}, {\eta}/{48})\right)$ 
where the first term on the right hand side is bounded by $e^{-\alpha}$. So again by Lemma \ref{othergenerations}, we have
\begin{multline*}
\P\left(\cB_4( \widetilde{\mathcal{K}}_n^{B^\delta\cap U}(m), {\eta}/{24})\right) 
\leq e^{-\alpha}+o_{n,\varepsilon}(1)+
\P\left(\Big\vert\sum_{m=1}^\infty \widetilde{\mathcal{K}}_n^{B_2^\delta\cap U}(m)-\Lambda p^\en(\phi,\pa{\phi})D_\infty n\Big\vert\geq \eta n/96\right).
\end{multline*}
 Recall that $L_{\delta+2}\subset B_2^\delta\subset L_\delta$, by \eqref{KtoW}, we have
 \begin{equation}\label{KtoW+}
\frac{W_m(F_{\log n-(3+\delta)\log\log n, \log n+\log\log n})}{\sqrt{m}}\leq  \frac{\widetilde{\mathcal{K}}_n^{B_2^\delta\cap U}(m)}{np^\en(\phi,\pa{\phi})}\leq \frac{W_m(F_{\log n-(1+\delta)\log\log n, \log n+\log\log n})}{\sqrt{m}}.
 \end{equation}
 Finally, \eqref{cvgsumK} follows immediately from Corollary \ref{sumW} .
\end{proof} 

\section{Convergence of martingale-like variables $(W_m(F_{a_m, b_m}),\ m \geq 1)$}
\label{sec3}
This section is devoted to proving Proposition \ref{cvgProp} and Corollary \ref{sumW} which only concern the environment. The main idea is borrowed from \cite{AidekonShi09}, on the Seneta-Heyde norm of the additive martingale $W_m$ in the boundary case \eqref{hyp0}. To do so, we need to introduce a change of measure and the corresponding spinal decomposition. 

\subsection{Lyons' change of measures and spinal decomposition}

We begin with  the following Biggins-Kyprianou \cite{BigginsKyprianou} identity usually called many-to-one Lemma :
\begin{lem}
In the boundary case \eqref{hyp0}, there exists a sequence of i.i.d. real-valued random variables $(S_i-S_{i-1}, i\geq0)$ with $S_0=0$ such that for any $n\geq1$ and any Borel function $g: \r^n\rightarrow\r_+$,
\begin{equation}
\E\left[\sum_{|x|=n}g\Big(V(x_i), 1\leq i\leq m\Big)\right]=\E\Big[e^{S_n}g(S_i; 1\leq i\leq n)\Big]. \label{M2O}
\end{equation}
\end{lem}
It immediately follows from \eqref{hyp0} and \eqref{hyp0+} that the sequence $(S_n, n\geq0)$ is a centered random walk of finite variance $\sigma^2:=\E[\sum_{|z|=1}V(z)^2e^{-V(z)}]$. For notational simplicity, let 
\[
\mS_n:=\min_{1\leq i\leq n}S_i,\ \MS_n:=\max_{1\leq i\leq n}S_i.
\]
Also let $\Ren(\cdot)$ be the renewal function associated with the strict descending ladder heights of $(S_i,i)$, it can be expressed as
\begin{equation}\label{R}
\Ren(u)=\sum_{k=0}^\infty\P(S_k<\mS_{k-1}, S_k\geq-u), \ \forall u\geq0.
\end{equation}
Obviously, $\Ren(u)\geq\Ren(0)=1$. The renewal theorem implies the existence of $\bf{c}_0 \in (0,+ \infty)$ such that  
\begin{align}\label{c0}
\cb_0:= \lim_{u \rightarrow +\infty} \frac{\Ren(u)}{u}.
\end{align}
Moreover, there exist two constants $0<C_-<C_+<\infty$ such that for any $u\geq0$,
\begin{equation}\label{Ren}
C_-(1+u)\leq \Ren(u)\leq C_+(1+u).
\end{equation}

For $\alpha>0$, define the truncated variables adapted to $\{\mathcal{F}_n:=\sigma((z, V(z)); |z|\leq n); n\geq0\}$, the natural filtration of the branching random walk, for any $n\geq0$ :
\begin{align*}
W_n^{(\alpha)}(F_{a_n,b_n}):=\sum_{|z|=n}e^{-V(z)}F_{a_n, b_n}(z)\un_{\{ \underline{V}(z)\geq-\alpha \}},\quad 
D_n^{(\alpha)}:=\sum_{|z|=n}\Ren(\alpha+V(z))e^{-V(z)}\un_{\{\underline{V}(z)\geq-\alpha\} }.
\end{align*}
See \eqref{defF} for the definition of $F_{a_n, b_n}(z)$. For any $a\in\r$, let $\P_a$ be the probability measure such that $\P_a(\{V(z), z\in\T\}\in\cdot)=\P(\{a+V(z),z\in\T\}\in\cdot)$. For $a\geq-\alpha$, we introduce the change of measure as follows:
\begin{equation}\label{changeprobab}
\Q^{(\alpha)}_a\vert_{\mathcal{F}_n}:=\frac{D_n^{(\alpha)}}{\Ren(\alpha+a)e^{-a}} \P_a\vert_{\mathcal{F}_n}.
\end{equation}
 The fact that $D_n^{(\alpha)}$ is a non-negative martingale which converges a.s. to some limit $D_\infty^{(\alpha)}$ has been proved by Biggins and Kyprianou \cite[Th 5.1]{BigKyp}.  So $\Q^{(\alpha)}_a$ is well define. Following their ideas, we present a spinal decomposition of the branching random walk under $\Q^{(\alpha)}_a$: we start with one individual $w_0$ (i.e., the root $\phi$ ), located at $V(w_0)=a$. Then for any $n\geq0$, 
 \begin{enumerate}
\item in the $n$-th generation, each individual $u$ except $w_n$, gives birth independently of all others to its children of the $n+1$-th generation whose positions constitute a point process distributed as $(V(z), |z|=1)$ under $\P_{V(u)}$;
\item $w_n$ produces, independently, its children in the $n+1$-th generation, whose positions are given by a point process distributed as $(V(z), |z|=1)$ under $\Q^{(\alpha)}_{V(w_n)}$;
\item Among the children of $w_n$, $w_{n+1}$ is chosen to be $z$ with probability proportional to \\$\Ren(\alpha+V(z))e^{-V(z)}\un_{\{\underline{V}(z)\geq-\alpha\}}$.
\end{enumerate}
In this description, the infinite ray $(w_n, n\geq0)$ is called the spine under $\Q^{(\alpha)}_a$. For simplicity, we write $\Q^{(\alpha)}$ for $\Q^{(\alpha)}_0$. The following fact makes explicit the distribution of $\omega_n$ and $(V(w_k), 1\leq k\leq n)$ under $\Q^{(\alpha)}$.

 \begin{fact}[\cite{BigKyp}]
Assume \eqref{hyp0}. Let $\alpha\geq0$,  
for any $n\geq1$ and $|z|=n$,
\begin{equation}\label{probable}
\Q^{(\alpha)}(w_n=z\vert\mathcal{F}_n)=\frac{\Ren(\alpha+V(z))e^{-V(z)}\un_{\{\underline{V}(z)\geq-\alpha \}}}{D_n^{(\alpha)}}.
\end{equation}
 The spine process $(V(w_n), m\geq0)$ under $\Q^{(\alpha)}$ is distributed as the random walk $(S_n, n\geq0)$ under $\P$ conditioned to stay above $-\alpha$. In other words, for any $n\geq1$ and any measurable function $g:\r^n\rightarrow\r_+$,
\begin{equation}\label{rw+}
\E_{\Q^{(\alpha)}}\Big[g(V(w_k), 1\leq k\leq n)\Big]=\frac{1}{\Ren(\alpha)}\E\Big[g(S_k, 1\leq k\leq n)\Ren(\alpha+S_n); \mS_n\geq-\alpha\Big].
\end{equation}
\end{fact}

\subsection{Convergence in probability of ${W_n^{(\alpha)}(F)}/{D_n^{(\alpha)}}$ under $\Q^{(\alpha)}$ \label{sec3.2}}

In this section we prove that if $a_n=a\sqrt{n}+o(\sqrt{n})$ and $b_n=b\sqrt{n}+o(\sqrt{n})$ for some $a,b>0$, then there exists some constant $\Cb_{a,b}\in(0,\infty)$ such that under $\Q^{(\alpha)}$,
\begin{equation}\label{cvgQ}
\sqrt{n}\frac{W_n^{(\alpha)}(F_{a_n, b_n})}{D_n^{(\alpha)}}\longrightarrow \Cb_{a,b},\ \textrm{in  probability}.
\end{equation}
This convergence also holds for $b=0$. When $a=0$, $\Cb_{a,b}$ is trivially zero by first moment estimation. 

It is known that $\lim_{n\rightarrow\infty}\min_{|z|=n}V(z)=\infty$, $\P$-a.s. As a consequence of \eqref{c0}, $D_\infty^{(\alpha)}=\cb_0D_\infty$ on $\{\inf_{z\in T}V(z)\geq-\alpha\}$. As it is shown in \cite[Th 5.1]{BigKyp} and  \cite{Che15}, $D_n^{(\alpha)}$ converges $\P$-a.s and in $L^1$ to $D_\infty^{(\alpha)}$ which is positive under $\P^*$. So $\Q^{(\alpha)}$ is absolutely continuous with respect to $\P$. We thus deduce Proposition \ref{cvgProp} from \eqref{cvgQ} with $\mathscr{C}_{a,b}=\cb_0\Cb_{a,b}$ (one can refer to \cite[Section 5]{AidekonShi09} for more details).

The proof of \eqref{cvgQ} is based on the computations of the  first and second moments of $\frac{W_n^{(\alpha)}(F_{a_n, b_n})}{D_n^{(\alpha)}}$. By \eqref{probable}, for any measurable function $F:\r^n\rightarrow \r_+$ of the trajectory of $V$ that is that $F(z)=F(V(y); \phi<y\leq z)$, we have
\begin{equation}\label{condQ}
\E_{\Q^{(\alpha)}}\Big[\frac{F(w_n)}{\Ren(\alpha+V(w_n))}\Big\vert\mathcal{F}_n\Big]=\sum_{|z|=n}\frac{e^{-V(z)}F(z)\un_{ \{ \underline{V}(z)\geq-\alpha \} }}{D_n^{(\alpha)}}=:\frac{W_n^{(\alpha)}(F)}{D_n^{(\alpha)}}.
\end{equation}
Taking expectation under $\Q^{(\alpha)}$ then applying \eqref{rw+} yields that
\begin{equation}\label{meanQ}
\E_{\Q^{(\alpha)}}\Big[\frac{W_n^{(\alpha)}(F)}{D_n^{(\alpha)}}\Big]=\E_{\Q^{(\alpha)}}\Big[\frac{F(w_n)}{\Ren(\alpha+V(w_n))}\Big]=\frac{1}{\Ren(\alpha)}\E\Big[F(S_k; 1\leq k\leq n); \mS_n\geq -\alpha\Big]
\end{equation}

\noindent Recall that for $|z|=n$ (see \eqref{defF})
\[
F_{a_n,b_n}(z)=\sqrt{n}\frac{e^{V(z)}}{\sum_{\phi<y\leq z}e^{V(y)}}\un_{\{ \overline{V}(z)\geq b_n\}}\un_{\{\max_{\phi<y\leq z}(\overline{V}(y)-V(y))\leq a_n\}}.
\]
In order to deal with the factor $\frac{e^{V(z)}}{\sum_{\phi<y\leq z}e^{V(y)}}$, we have to add some restrictions to the sites. Observe that if $V(z)\ll\overline{V}(z)$, then 
\[
\frac{e^{V(z)}}{\sum_{\phi<y\leq z}e^{V(y)}}\leq e^{V(z)-\overline{V}(z)}\ll1.
\]
So it is reasonable to count only the sites $|z|=n$ such that $V(z)\approx\MV(z)$. \bl{And this choice gives an extra factor $\frac{1}{\sqrt{n}}$. That is why we multiply $\sqrt{n}$ in the definition of $F_{a_n,b_n}(z)$.} Let us introduce the following notations. For any $|z|=n$ and $0\leq m\leq n$, let $z_m$ be the ancestor of $z$ in the $m$-th generation and define
\begin{align*}
\Upsilon_z:=&\inf\{k: V(z_k)=\overline{V}(z)=\max_{0\leq m\leq n }V(z_m)\}.
\end{align*}
Similarly, we also define $\Upsilon_S=\inf\{k: S_k=\MS(k):=\max_{0\leq m\leq n }S_m\}$ and $\mS_{[m,n]}:=\min_{m \leq k \leq n} S_k$ 
for one-dimensional random walk $(S_k,k)$. Instead of $F_{a_n,b_n}(z)$, it is more convenient to consider 
\begin{align}
G(z):=&\sqrt{n}\frac{e^{V(z)}}{\sum_{\phi<y\leq z}e^{V(y)}}\un_{\{\overline{V}(z)\geq b_n,\ \max_{y\leq z}(\overline{V}(y)-V(y))\leq a_n\}}\un_{\{\Upsilon_z>n_0\}}\label{defG},
\end{align}
with  $n_0:=\lfloor n-n^{1/3}\rfloor$. 

Moreover, following \cite{AidekonShi09}, let us introduce the events   $E^z_n$ for $|z|=n$ as follows. Let $\Omega(y):=\{u\in\T: u\neq y, \pa{u}=\pa{y}\}$ be the collection of brothers of $y$. If $(k_n,n)$ is a positive sequence such that $k_n=o(n^{1/2})$ and $(\log n)^6=o(k_n)$, let
\[
E^z_n := E^z_{n,1}\cap E^z_{n,2}\cap E^z_{n,3},
\]
where
\begin{align}
E^z_{n,1}&=\{k_n^{1/3}\leq V(z_{k_n})\leq k_n\}\cap\bigcap_{i=k_n}^n\{V(z_i)\geq k_n^{1/6}\}; \nonumber\\
E^z_{n,2}&=\bigcap_{i=k_n}^n\{\sum_{y\in\Omega(z_{i+1})}[1+(V(y)-V(z_i))_+]e^{-[V(y)-V(z_i)]}\leq e^{V(z_i)/2}\}; \nonumber \\
E^z_{n,3}&=\{\sum_{i=k_n}^n\sum_{y\in\Omega(z_{i+1})}\sum_{|u|=n, u\geq y}\Ren(\alpha+V(u))e^{-V(u)}\un_{\{ \underline{V}(u)\geq-\alpha\}}\leq \frac{1}{n^2}\}, \label{defE3}
\end{align}
with $x_+:=\max(x,0)$. In particular, for $w_n$, write $E_n$ (resp. $E_{n,i}$) instead of $E_n^{w_n}$ (resp. $E^{w_n}_{n,i}$). Let $H(z):=G(z)\un_{E^z_n}$. Here we choose $k_n=o(n^{1/2})$ so that $E_{n,1}$ happens with high probability and $(\frac{V(w_i)}{\sqrt{n}}; k_n\leq i\leq n)$ is still asymptotically Brownian meander. At the same time, we take $(\log n)^6=o(k_n)$ to make sure that the probability in \eqref{E3c} is $o_n(1)$. Moreover, it is proved in Lemma 4.7 of \cite{AidekonShi09} that for $(k_n, n)$ chosen as stated above,
\begin{align}
\lim_{n\rightarrow\infty}\Q^{(\alpha)}(E_n)=&1, \nonumber \\
\lim_{n\rightarrow\infty}\inf_{u\in [k_n^{1/3}, k_n]}\Q^{(\alpha)}(E_n\vert V(w_{k_n})=u)=&1\label{Qu}.
\end{align}
One will see later that involving $E_n$ helps us to control the second moment of $\frac{W_n^{(\alpha)}(F)}{D_n^{(\alpha)}}$ without influencing its first moment. Let us now state the main lemma of this section.

\begin{lem}\label{mom}
 Let $\alpha\geq0$, we have
\begin{align}
\lim_n\sqrt{n}\E_{\Q^{(\alpha)}}\Big[\frac{W_n^{(\alpha)}(H)}{D_n^{(\alpha)}}\Big]&=\lim_n\sqrt{n}\E_{\Q^{(\alpha)}}\Big[\frac{W_n^{(\alpha)}(F_{a_n,b_n})}{D_n^{(\alpha)}}\Big]=\Cb_{a,b}\label{F},\\
\lim_n \E_{\Q^{(\alpha)}}\Big[\Big(\sqrt{n}\frac{W_n^{(\alpha)}(H)}{D_n^{(\alpha)}}\Big)^2\Big]&=\Cb_{a,b}^2.\label{H2}
\end{align}
\end{lem}
This lemma shows immediately that under $\Q^{(\alpha)}$, $\sqrt{n}\frac{W_n^{(\alpha)}(H)}{D_n^{(\alpha)}}$ converges in probability towards $\Cb_{a,b}$ while $\sqrt{n}\frac{W_n^{(\alpha)}(F_{a_n,b_n}-H)}{D_n^{(\alpha)}}=o_n(1)$ in probability. We hence conclude the convergence \eqref{cvgQ}.

Moreover, by the change of measures \eqref{changeprobab}, this means that 
\begin{equation}\label{expWF}
\sqrt{n}\E[W_n^{(\alpha)}(F_{a_n,b_n})]\rightarrow \Cb_{a,b}R(\alpha).
\end{equation}

Before starting the proof of Lemma \ref{mom}, let us state a useful result on the random walk $\{S_k; k\geq0\}$ and the definition of constants $\Cb_{a,b}$, $\mathscr{C}_{a,b}$ and $\lambda(\cdot)$.

It is proved in \cite{Afa93} that the following joint convergence in law holds
\begin{equation}\label{mSjointcvg}
\left\{\left( \frac{S_{\lfloor nt\rfloor}}{\sqrt{\sigma^2 n}}, t\in[0,1]\right); \sum_{i=0}^ne^{-S_i} \Big \vert \underline{S}_n>0 \right\}\Longrightarrow \{(\mB_t,t\in[0,1]), \Hinf_\infty\},
\end{equation}
where $(\mB_t,t\in[0,1])$ is a Brownian meander independent of $\Hinf_\infty\in[1,\infty)$. In fact, in the sense of \cite{BertoinDoney}, the associated random walk conditioned to stay positive, denoted $(\zeta_n, n\geq0)$, is a Markov chain with probabilities of transition $p(x,dy):=\frac{\Ren(y)}{\Ren(x)}\un_{ \{y>0\} }\P_x(S_1\in dy)$, with $\P(\zeta_0=0)=1$. Consequently $ \Hinf_\infty$ can be defined as 
\[
\Hinf_\infty:=\sum_{j=0}^\infty e^{-\zeta_j}.
\]
Also we denote 
\begin{align}\label{c+}
\cb_1^+:= \lim_{n\rightarrow\infty}\sqrt{n}\P(\mS_n\geq0),\ \cb_2^+:=\lim_{n\rightarrow\infty}\sqrt{n}\P(\mS_n>0),
\end{align}
where the existence and positivity of $\cb_1^+$ and $\cb_2^+$ have been proved in \cite[Th.1 in XII. 7 \& Th.1 in XVIII.5]{Feller}.
We also introduce two functions which appears in the definition of  $\lambda(\cdot)$.  The first one involves the discrete random walk $(S_j,j)$. For any $j \geq 1$ and $x\geq 1$, define  
\begin{align}
\Gl_j(x):= \E\left[\frac{e^{S_j}}{x+\sum_{1\leq i\leq j}e^{S_i}}; \MS_j\leq0\right], \textrm{ with } \Gl_0(x):= \frac{1}{x}. \label{mG}
\end{align}
The second function depends on Brownian meander $(\mB_s, 0\leq s\leq1)$. Let $\overline{\mB}_s:=\sup_{0\leq t\leq s}\mB_t$ and $\underline{\mB}_{[s,1]}:=\inf_{s\leq t\leq 1}\mB_t$ for any $s\in[0,1]$.  Take $a>0$ and $b\geq0$, for any $(x,h) \in \R_+^2$, let
\begin{align*} 
\Psi^{a,b}(x,h):= 
\cb_2^+\P\left(\sigma \mB_1 > (\sqrt{2}b-x)\vee h, \sigma(\overline{\mB}_1- \mB_1)\leq (\sqrt{2} a-h)_+\wedge x, \max_{0\leq s\leq 1}\sigma(\mB_s-\underline{\mB}_{[s,1]})\leq \sqrt{2}a\right).
\end{align*}
Finally, let
\begin{align}
\mathcal{C}_{a,b}&:=2\cb_1^+\cb_2^+\E\left(\Psi^{a,b}(\sigma \mB_1, \sigma( \overline{\mB}_1- \mB_1)); \max_{0\leq s\leq 1}\sigma(\overline{\mB}_s-\mB_s)\leq \sqrt{2} a\right), \textrm{ and} \label{cab}\\
\Cb_{a,b}&:=\mathcal{C}_{a,b} \sum_{j=0}^{+\infty} \E\left[\Gl_j(\Hinf_\infty)\right].\label{Cab}
\end{align}
$\mathcal{C}_{a,b}$ is well defined positive and finite [see Lemma \ref{A} and its proof in Appendix \ref{A1}], also we set $\Cb_{0,b}=0$. Note also that $\Gl_j(x)\leq \Gl_j:=\E[e^{S_j}\1_{\MS_j\leq0}]$ for any $j\geq1$ and $x\geq1$ so $\Cb_{a,b}$ is finite [see \eqref{eSMSbd}]. This implies that for any $\gamma>0$,
\[ \lambda(\gamma):= \frac{\mathscr{C}_{\gamma^{-1/2},\gamma^{-1/2}}}{\gamma}=\cb_0  \frac{\Cb_{\gamma^{-1/2},\gamma^{-1/2}}}{\gamma}\in (0,\infty). \]
The integrability of  $\lambda$ is stated in Lemma \ref{A} of Appendix, so $\Lambda$ in Theorem \eqref{upton2} is well defined, i.e.
\begin{equation}\label{Lambda}
\Lambda= \int_{0}^{+\infty} \lambda(x)dx\in(0,\infty). 
\end{equation}

\subsubsection{First moment estimate: proof of \eqref{F}}
Let us turn to the proof of Lemma \ref{mom}. First of all, note that $0\leq H\leq G\leq F_{a_n, b_n}$. \eqref{F} follows from the following lemma.
\begin{lem}If $\lim_{n\rightarrow\infty}\frac{a_n}{\sqrt{n}}=a\in(0,\infty)$ and $\lim_{n\rightarrow\infty}\frac{b_n}{\sqrt{n}}=b\in(0,\infty)$, then
\begin{align}
\lim_n\sqrt{n}\E_{\Q^{(\alpha)}}\Big[\frac{W_n^{(\alpha)}(G)}{D_n^{(\alpha)}}\Big]=&\Cb_{a,b}\label{G},\\
\lim_n\sqrt{n}\E_{\Q^{(\alpha)}}\Big[\frac{W_n^{(\alpha)}(F_{a_n,b_n}-G)}{D_n^{(\alpha)}}\Big]=&0,\label{FG}\\
\lim_n\sqrt{n}\E_{\Q^{(\alpha)}}\Big[\frac{W_n^{(\alpha)}(G-H)}{D_n^{(\alpha)}}\Big]=&0.\label{GH}
\end{align}
\end{lem}
\begin{proof} 

\noindent \\ \textbf{Proof of \eqref{FG}}: For any $|z|=n$, comparing \eqref{defF} and \eqref{defG}, we define
\begin{align*}
r(z)&:=\sqrt{n}\frac{e^{V(z)}}{\sum_{\phi<y\leq z}e^{V(y)}}\un_{\{\overline{V}(z)\geq b_n,\ \max_{y\leq z}(\overline{V}(y)-V(y))\leq a_n\}}\un_{\{\Upsilon_z\leq n_0\}},
\end{align*}
with recall $n_0=\lfloor n-n^{1/3}\rfloor$, also it is clear that $0\leq F_{a_n,b_n}-G\leq r$. 
So to obtain \eqref{FG}, it suffices to show that 
\begin{equation}\label{h}
\E_{\Q^{(\alpha)}}\Big[\frac{W_n^{(\alpha)}(r)}{D_n^{(\alpha)}}\Big]=\frac{o_n(1)}{\sqrt{n}}.
\end{equation}
Applying \eqref{meanQ} for $r$ yields that
\begin{align*}
\E_{\Q^{(\alpha)}}\Big[\frac{W_n^{(\alpha)}(r)}{D_n^{(\alpha)}}\Big]=\frac{1}{\Ren(\alpha)}\E\Big[\frac{\sqrt{n}e^{S_n}}{\sum_{1\leq j\leq n}e^{S_j}};  \MS_n\geq b_n, \max_{j\leq n}(\MS_j-S_j)\leq a_n, \Upsilon_S\leq n_0,\mS_n\geq-\alpha\Big].
\end{align*}
Partitioning on the values of $\Upsilon_S$ gives that
\begin{align*}
\E_{\Q^{(\alpha)}}\Big[\frac{W_n^{(\alpha)}(r)}{D_n^{(\alpha)}}\Big]&=\sum_{k=0}^{n_0}\frac{1}{\Ren(\alpha)}\E\Big[\frac{\sqrt{n}e^{S_n}}{\sum_{1\leq j\leq n}e^{S_j}}; \Upsilon_S=k, \MS_n\geq b_n, \max_{j\leq n}(\MS_j-S_j)\leq a_n,\mS_n\geq-\alpha\Big]\\
&\leq \sum_{k=0}^{n_0}\frac{\sqrt{n}}{\Ren(\alpha)}\E\Big[e^{S_n-S_k}\un_{ \{ \Upsilon_S=k, \mS_k\geq-\alpha \} }\Big].
\end{align*}
Notice that $\{\Upsilon_S=k\}=\{S_k=\MS_k>\MS_{k-1}\}\cap\{\max_{k<j\leq n}S_j-S_k\leq 0\}$. By Markov property at time $k$,
\begin{align*}
\E_{\Q^{(\alpha)}}\Big[\frac{W_n^{(\alpha)}(r)}{D_n^{(\alpha)}}\Big]&\leq \sum_{k=0}^{n_0}\frac{\sqrt{n}}{\Ren(\alpha)}\P\Big(\mS_k\geq-\alpha, S_k=\MS_k\Big)\E\Big[e^{S_{n-k}}\un_{\{ \MS_{n-k}\leq0 \}}\Big],
\end{align*}
which by \eqref{Ren}, \eqref{mSMSbd} and \eqref{eSMSbd} implies that
\begin{align*}
\E_{\Q^{(\alpha)}}\Big[\frac{W_n^{(\alpha)}(r)}{D_n^{(\alpha)}}\Big]\leq &\sum_{k=0}^{n_0}\frac{\sqrt{n}}{\Ren(\alpha)}\frac{c (1+\alpha)}{(k+1)(n-k)^{3/2}}  \leq  c'\sqrt{n}\sum_{k=0}^{n-n^{1/3}} \frac{1}{(k+1)(n-k)^{3/2}}=O(\frac{1}{n^{2/3}}).
\end{align*}
 Observe that $\sum_{0\leq k\leq n/2}\frac{1}{(k+1)(n-k)^{3/2}}=\frac{O(1)}{n^{3/2}}\sum_{0\leq k\leq n/2}\frac{1}{k+1}=O(\frac{\log n}{n^{3/2}})$. And observe also that $\sum_{n/2\leq k\leq n-n^{1/3}}\frac{1}{(k+1)(n-k)^{3/2}}=\frac{O(1)}{n}\sum_{n/2\leq k\leq n-n^{1/3}}\frac{1}{(n-k)^{3/2}}=O(\frac{(n^{1/3})^{-1/2}}{n})$. Thus \eqref{h} holds.  \\

\noindent \textbf{Proof of \eqref{G}}:
 It follows from \eqref{meanQ} that 
 \[
 \E_{\Q^{(\alpha)}}\Big[\frac{W_n^{(\alpha)}(G)}{D_n^{(\alpha)}}\Big]=\frac{\sqrt{n}}{\Ren(\alpha)}\E\Big[\frac{e^{S_n}}{\sum_{1\leq j\leq n}e^{S_j}}; \Upsilon_S> n_0, \MS_n\geq b_n, \max_{j\leq n}(\MS_j-S_j)\leq a_n,\mS_n\geq-\alpha\Big].
 \]
Partioning over the values of $\Upsilon_S$ implies that $\E_{\Q^{(\alpha)}}\Big[\frac{W_n^{(\alpha)}(G)}{D_n^{(\alpha)}}\Big]=\frac{\sqrt{n}}{\Ren(\alpha)}\sum_{k=n_0+1}^n\sigma_k$ where
 \[
 \sigma_k:=\E\Big[\frac{e^{S_n}}{\sum_{1\leq j\leq n}e^{S_j}}; \Upsilon_S=k, \MS_n\geq b_n, \max_{j\leq n}(\MS_j-S_j)\leq a_n,\mS_n\geq-\alpha\Big].
 \]
Let $T_{i}=S_{i+k}-S_k$, and notice that $\{S_j; 0\leq j\leq k\}$ and $\{T_i; 0\leq i\leq n-k\}$ are independent, we have
\begin{multline*}
\sigma_k=\E\Big[\frac{e^{T_{n-k}}}{\sum_{1\leq j\leq k}e^{S_j-S_k}+\sum_{1\leq j\leq n-k}e^{T_{j}}};\overline{S}_{k-1}<S_k, S_k\geq b_n,\underline{S}_k\geq -\alpha, \max_{i\leq k}(\overline{S}_i-S_i)\leq a_n,\\
\underline{T}_{n-k}\geq (-\alpha-S_k)\vee (-a_n), \overline{T}_{n-k}\leq 0 \Big].
\end{multline*}
Note that $\{(-\alpha-S_k)\vee (-a_n)=-O(\sqrt{n})\}$ as $S_k\geq b_n$, while with high probability, $\underline{T}_{n-k}=O(n^{1/6})$ for $n-k\leq n^{1/3}$. The next step is to approximate $\sigma_k$ by $\sigma_k'$ which is defined as follows
\begin{multline*}
\sigma_k':=\E\Big[\frac{e^{T_{n-k}}}{\sum_{1\leq j\leq k}e^{S_j-S_k}+\sum_{1\leq j\leq n-k}e^{T_{j}}}; \overline{S}_{k-1}<S_k, S_k\geq b_n,\underline{S}_k\geq -\alpha, \max_{i\leq k}(\overline{S}_i-S_i)\leq a_n,\overline{T}_{n-k}\leq 0 \Big].
\end{multline*}
Observe that $0\leq\sigma_k'-\sigma_k\leq\P( \overline{S}_{k-1}<S_k, \underline{S}_k\geq -\alpha, \underline{T}_{n-k}\leq (-a_n)\vee(-\alpha-b_n))$.  By independence of $S$ and $T$, then using \eqref{mSMSbd} and \eqref{LDP}, one sees that
\begin{align*}
\sigma_k'-\sigma_k\leq  &\P\left(\underline{S}_k\geq -\alpha, S_k=\overline{S}_k\right)\P\Big(\underline{T}_{n-k}\leq (-a_n)\vee(-\alpha-b_n)\Big)\\
\leq & {c(1+\alpha)}k^{-1}e^{-c'\sqrt{n}}.
\end{align*}
 Hence, $\frac{\sqrt{n}}{\Ren(\alpha)}\sum_{k=n_0+1}^{n-1}(\sigma_k'-\sigma_k)=\frac{o_n(1)}{\sqrt{n}}$.
This implies that
\begin{equation}\label{sum}
\sqrt{n} \E_{\Q^{(\alpha)}}\Big[\frac{W_n^{(\alpha)}(G)}{D_n^{(\alpha)}}\Big]= \frac{1}{\Ren(\alpha)}\sum_{k=n_0+1}^{n}n\sigma_k'+o_n(1).
\end{equation}
We now turn to consider $\sigma_k'$. By independence of $S$ and $T$ again,
\begin{align*}
\sigma_k'&=\E\Big[\Gl_{n-k}(\sum_{1\leq j\leq k}e^{S_j-S_k});  \overline{S}_{k-1}<S_k, S_k\geq b_n,\underline{S}_k\geq -\alpha,\max_{i\leq k}(\overline{S}_i-S_i)\leq a_n\Big]
\end{align*}
where $\Gl_.(x)$ is defined in \eqref{mG}. Observe that for $k=n-i$ with $i\in\N$ fixed,
\[
\sigma_k'=\sigma_{n-i}'=\E\Big[\Gl_{i}(\sum_{1\leq j\leq n-i}e^{S_j-S_{n-i}});  \overline{S}_{n-i-1}<S_{n-i}, S_{n-i}\geq b_n,\underline{S}_{n-i}\geq -\alpha,\max_{i\leq n-i}(\overline{S}_i-S_i)\leq a_n\Big],
\]
which by \eqref{sumcvg}, is asymptotically, $\frac{\mathcal{C}_{a,b}\Ren(\alpha)\E[\Gl_{i}(\Hinf_\infty)]}{n}+\frac{o_n(1)}{n}$. Moreover, as $\sup_{x\geq1}\Gl_i(x)\leq \E[e^{S_i}; \MS_i\leq 0]\leq {c}{i^{-3/2}}$,
by \eqref{mSMSbd}, one sees that for $n_0\leq n-i\leq n$, 
\[
\sigma_{n-i}'\leq \frac{c}{i^{3/2}}\P( \overline{S}_{n-i-1}<S_{n-i}, S_{n-i}\geq b_n,\underline{S}_{n-i}\geq -\alpha)\leq \frac{c(1+\alpha)}{i^{3/2} (n-i)}.
\]
As a result, for any integer $K\geq1$ fixed, 
\[
\sum_{k=n_0+1}^{n}n\sigma_k'=\sum_{i=0}^Kn\sigma_{n-i}'+\sum_{K< i\leq n^{1/3}}n\sigma_{n-i}'=\Cb_{a,b}\Ren(\alpha)\sum_{i=0}^K \E[\Gl_i(\Hinf_\infty)]+o_n(1)+o_{K}(1),
\]
where $\sum_{k=0}^{K}\E[\Gl_{i}(\Hinf_\infty)]=\sum_{i=0}^\infty \E[\Gl_i(\Hinf_\infty)]+o_K(1)$.

Plugging this into \eqref{sum}, letting $n\rightarrow\infty$ then $K\rightarrow\infty$ implies that, 
{\[
\lim_{n\rightarrow\infty} \sqrt{n}\E_{\Q^{(\alpha)}}\Big[\frac{W_n^{(\alpha)}(G)}{D_n^{(\alpha)}}\Big]=\mathcal{C}_{a,b}\sum_{j=0}^\infty \E[\Gl_j(\Hinf_\infty)]=\Cb_{a,b},
\]}
 which ends the proof of \eqref{G}.

\noindent \textbf{Proof of \eqref{GH}}:
by \eqref{meanQ}, we only need to prove that
\begin{equation}\label{estimate1}
\sqrt{n}\E_{\Q^{(\alpha)}}\Big[\frac{W_n^{(\alpha)}(G-H)}{D_n^{(\alpha)}}\Big]=\E_{\Q^{(\alpha)}}\Big[\frac{\sqrt{n}G(w_n)}{\Ren(\alpha+V(w_n))}\un_{E_n^c}\Big]=:LHS=o_n(1).
\end{equation}
First, we have
\begin{align*}
LHS\leq& n\E_{\Q^{(\alpha)}}\Big[\frac{e^{V(w_n)-\overline{V}(w_n)}\un_{E_n^c}}{\Ren(\alpha+V(w_n))};  \Upsilon_{w_n} >n_0
\Big] \leq  LHS_1+LHS_2+LHS_3
\end{align*}
where
\begin{align*}
LHS_1:=&n\E_{\Q^{(\alpha)}}\Big[\frac{e^{V(w_n)-\overline{V}(w_n)}\un_{E_{n,1}^c}}{\Ren(\alpha+V(w_n))}; \Upsilon_{w_n}>n_0\Big],
LHS_2:=&n\E_{\Q^{(\alpha)}}\Big[\frac{e^{V(w_n)-\overline{V}(w_n)}\un_{E_{n,1}\cap E_{n,2}^c}}{\Ren(\alpha+V(w_n))}; \Upsilon_{w_n} >n_0\Big],\\
LHS_3:=&n\E_{\Q^{(\alpha)}}\Big[\frac{e^{V(w_n)-\overline{V}(w_n)}\un_{E_{n,1}\cap E_{n,2}\cap E_{n,3}^c}}{\Ren(\alpha+V(w_n))}\Big].
\end{align*} 
Each term $LHS_i,\ i=1,2, 3,$ are treated separately. \\ 
\textbf{For $LHS_1$}, by \eqref{rw+}, we have
\begin{align*}
LHS_1\leq& \frac{n}{\Ren(\alpha)}\E[e^{S_n-\overline{S}_n}; S_{k_n}\in[k_n^{1/3},k_n], \mS_{[k_n,\Upsilon_S]}\leq k_n^{1/6}, \Upsilon_S>n_0, \underline{S}_n\geq-\alpha]\\
&+ \frac{n}{\Ren(\alpha)}\E[e^{S_n-\overline{S}_n}; S_{k_n}\notin[k_n^{1/3},k_n], \Upsilon_S>n_0, \underline{S}_n\geq-\alpha]=:\xi_1+\xi_1'.
\end{align*}
Arguing over the values of $\Upsilon_S$ then using Markov property at $\Upsilon_S=k$,
\begin{align}
\xi_1\leq & \frac{n}{\Ren(\alpha)}\sum_{k=n_0+1}^n \E\Big[e^{S_{n-k}}\un_{\{\MS_{n-k}\leq0\} }\Big]\P\Big( S_{k_n}\in[k_n^{1/3},k_n],\mS_{[k_n,k]}\leq k_n^{1/6}, S_k>\overline{S}_{k-1}, \underline{S}_k\geq-\alpha\Big)\nonumber\\
\leq & \frac{n}{\Ren(\alpha)}\sum_{k=n_0+1}^n\frac{c}{(n-k+1)^{3/2}}\P\Big( S_{k_n}\in[k_n^{1/3},k_n],\mS_{[k_n,k]}\leq k_n^{1/6}, S_k>\overline{S}_{k-1}, \underline{S}_k\geq-\alpha\Big)\label{xi1},
\end{align}
where the second inequality holds because of \eqref{eSMSbd}. Moreover, by \eqref{mSMSkbd}, uniformly on $k\in [n_0,n]\cap\z$,
\[
\P\Big( S_{k_n}\in[k_n^{1/3},k_n],\mS_{[k_n,k]}\leq k_n^{1/6}, S_k>\overline{S}_{k-1}, \underline{S}_k\geq-\alpha\Big)=\frac{o_n(1)}{n}.
\]
We hence deduce that $\xi_1=o_n(1)$ since $\sum_{k=n_0+1}^n\frac{c}{(n-k+1)^{3/2}}$ is finite.

For $\xi_1'$, similarly, applying Markov property at time $\Upsilon_S=k$ then \eqref{eSMSbd}, we have
\begin{align*}
\xi_1'&\leq \frac{n}{\Ren(\alpha)}\sum_{k=n_0+1}^n \E[e^{S_{n-k}}; \MS_{n-k}\leq0]\P\Big(S_{k_n}\notin[k_n^{1/3},k_n], \mS_k\geq-\alpha, S_k=\MS_k\Big)\\
&\leq \frac{cn}{\Ren(\alpha)}\sum_{k=n_0+1}^n \frac{1}{(n-k+1)^{3/2}}\P\Big(S_{k_n}\notin[k_n^{1/3},k_n], \mS_k\geq-\alpha, S_k=\MS_k\Big)
\end{align*}
which by \eqref{kmSMSbd} yields that
\[
\xi_1\leq \frac{c'n}{\Ren(\alpha)}\sum_{k=n_0+1}^n \frac{1}{(n-k+1)^{3/2}nk_n^{1/2}}=o_n(1).
\]

\textbf{For $LHS_3$}, let $\mathscr{G}_\infty$ be the sigma-field generated by the spine and all siblings of the spine. 
We know from (\cite{AidekonShi09} eq. (4.9)) that
\begin{equation}\label{E3c}
\Q^{(\alpha)}\Big(E_{n,1}\cap E_{n,2}\cap E_{n,3}^c\Big\vert \mathscr{G}_\infty\Big)\leq O(n^3e^{-k_n^{1/6}/3}),
\end{equation}
wich implies that
\begin{align*}
LHS_3\leq &n\E_{\Q^{(\alpha)}}\Big[\frac{e^{V(w_n)-\overline{V}(w_n)}}{\Ren(\alpha+V(w_n))}\times\Q^{(\alpha)}\Big(E_{n,1}\cap E_{n,2}\cap E_{n,3}^c\Big\vert \mathscr{G}_\infty\Big)\Big] 
\leq  O(n^4e^{-k_n^{1/6}/3})=o_n(1).
\end{align*}

\textbf{For $LHS_2$}, we follow the same lines as in (\cite{AidekonShi09} page 18, below (4.8)) using the same notations. 
For any $1\leq i\leq n$,
\[
\Q^{(\alpha)}(E_{n,i}^c\vert V(w_k); 0\leq k\leq n)\leq c' h(V(w_i)),
\]
where for any $u\geq -\alpha$,
\(
h(u):=\E[X\un_{ \{X+\widetilde{X}>e^{u/2}\}}+\frac{\widetilde{X}\un_{X+\widetilde{X}>e^{u/2}}}{u+\alpha+1}],
\)
with $X:=\sum_{|z|=1}e^{-V(z)}$ and $\widetilde{X}:=\sum_{|z|=1}V_+(z)e^{-V(z)}$. Note that $\E[(X+\widetilde{X})^2]<\infty$ because of \eqref{hyp1}. 
Markov inequality gives that $h(u)\leq e^{-u/2}$ . Recall that $V(w_i)\geq k_n^{1/6}$ on $E_{n,i}$. Therefore,
\begin{align*}
LHS_2\leq &c'\sum_{i=k_n}^n n\E_{\Q^{(\alpha)}}\Big[\frac{e^{V(w_n)-\overline{V}(w_n)}}{\Ren(\alpha+V(w_n))}h(V(w_i))\un_{E_{n,1}}; \Upsilon_{w_n}>n_0\Big]\\
\leq & c' n(n-k_n)e^{-k_n^{1/6}/2}\E_{\Q^{(\alpha)}}\Big[\frac{e^{V(w_n)-\overline{V}(w_n)}}{\Ren(\alpha+V(w_n))}; \Upsilon_{w_n}>n_0\Big]
\end{align*}
 Applying \eqref{rw+} then partitioning on the values of $\Upsilon_S$ yields 
\begin{align*}
LHS_2\leq& c' n(n-k_n)e^{-k_n^{1/6}/2} \sum_{k=n_0+1}^n\frac{1}{\Ren(\alpha)}\E\left[e^{S_n-S_k}; \Upsilon_S=k,  \underline{S}_n\geq-\alpha\right] \\
\leq &c' n^2e^{-k_n^{1/6}/2} \sum_{k=n_0+1}^n\frac{1}{\Ren(\alpha)} \E\left[e^{S_{n-k}}\un_{\MS_{n-k}\leq 0}\right]\P\left(\mS_k\geq-\alpha, S_k=\MS_k\right),
\end{align*}
by Markov property. By \eqref{Ren}, \eqref{eSMSbd} and \eqref{mSMSbd}, 
\[
LHS_2\leq c n^2e^{-k_n^{1/6}/2}  \sum_{k=n_0+1}^n \frac{1}{k(n-k+1)^{3/2}}=o_n(1),
\]
since $(\log n)^6=o(k_n)$.

Collecting all the estimations for the $LHS_i$, this ends the proof of \eqref{GH}. 
\end{proof}

\subsubsection{Second moment estimate: proof of \eqref{H2}}
Recall the definitions of $G$ in \eqref{defG} and $H$ below \eqref{defE3}. In view of \eqref{F}, it suffices to show that
\begin{equation}\label{limH2}
\limsup_{n\rightarrow\infty}\E_{\Q^{(\alpha)}}\bigg[\Big(\frac{\sqrt{n}W_n^{(\alpha)}(H)}{D_n^{(\alpha)}}\Big)^2\bigg]\leq \Cb_{a,b}^2.
\end{equation}
By \eqref{probable},
\begin{align}
LHS_\eqref{limH2}:=&\E_{\Q^{(\alpha)}}\bigg[\Big(\frac{\sqrt{n}W_n^{(\alpha)}(H)}{D_n^{(\alpha)}}\Big)^2\bigg]=\E_{\Q^{(\alpha)}}\bigg[\frac{\sqrt{n}W_n^{(\alpha)}(H)}{D_n^{(\alpha)}}\times \frac{\sqrt{n}H(w_n)}{\Ren(\alpha+V(w_n))}\bigg]\nonumber\\
\leq & \E_{\Q^{(\alpha)}}\bigg[\frac{\sqrt{n}W_n^{(\alpha)}(G)}{D_n^{(\alpha)}}\times \frac{\sqrt{n}G(w_n)\un_{E_n}}{\Ren(\alpha+V(w_n))}\bigg].\label{LHSlimH2}
\end{align}
For convenience, let
\begin{align*}
 W_n^{(\alpha),[k_n,n]}(G):&=e^{-V(w_n)}G(w_n)\un_{\{ \underline{V}(w_n)\geq-\alpha \}}+\sum_{i=k_n}^{n-1}\sum_{y\in\Omega(w_{i+1})}\sum_{|z|=n, z\geq y}e^{-V(z)}G(z)\un_{\{ \underline{V}(z)\geq-\alpha\}},\\
  W_n^{(\alpha),[0,k_n)}(G):&=\sum_{i=0}^{k_n-1}\sum_{y\in\Omega(w_{i+1})}\sum_{|z|=n, z\geq y}e^{-V(z)}G(z)\un_{ \{\underline{V}(z)\geq-\alpha\}}, 
\end{align*}
with $\Omega(\omega_{i+1})=\{|x|=i+1 : \pa{x}= \omega_{i}, \ x \neq  \omega_{i+1}\}$. In the similar way, we define $D_n^{(\alpha),[0,k_n)}$ and $D_n^{(\alpha),[k_n,n]}$. Recall \eqref{defE3}, the event $E_{n,3}$ means that $D_n^{(\alpha),[k_n,n]}\leq {n^{-2}}$.
So under $\Q^{(\alpha)}$, the descendants of the $(\omega_i; {k_n}\leq i\leq n)$ make little contribution to $D_n^{(\alpha)}$. The same thing happens to $W_n^{(\alpha)}$. We thus approximate $\frac{\sqrt{n}W_n^{(\alpha)}(G)}{D_n^{(\alpha)}}$ by $\frac{\sqrt{n}W_n^{(\alpha)[0,k_n)}(G)}{D_n^{(\alpha),[0,k_n)}}$  on the right hand side of \eqref{LHSlimH2}. Then Markov property at $k_n$ makes it possible to deal with these two terms in the product separately. Clearly 
\begin{align*}
LHS_\eqref{limH2}\leq  \E_{\Q^{(\alpha)}}\left[{\sqrt{n}W_n^{(\alpha),[k_n,n]}(G)}{(D_n^{(\alpha)})^{-1}}\times \widetilde{G}_n \right]+\E_{\Q^{(\alpha)}}\left[\widetilde W_n \times \widetilde{G}_n \right],
\end{align*}
with $\widetilde W_n:={\sqrt{n}W_n^{(\alpha),[0,k_n)}(G)}/{D_n^{(\alpha),[0,k_n)}}$ and $ \widetilde{G}_n:= {\sqrt{n}G(w_n)\un_{E_n}}/{\Ren(\alpha+V(w_n))}$. For the first expectation above, as  $G\leq \sqrt{n}\un_{\{V(w_n)\geq b_n/2\} }$, it is clear that given $E_n$,
\[
W_n^{(\alpha),[k_n,n]}(G)\leq \sqrt{n} W_n^{(\alpha),[k_n,n]}\leq \sqrt{n} D_n^{(\alpha),[k_n,n]}\leq n^{-3/2}.
\]
In view of \eqref{Ren}, it follows that
\begin{align*}
 \E_{\Q^{(\alpha)}}\left[{\sqrt{n}W_n^{(\alpha),[k_n,n]}(G)}{(D_n^{(\alpha)})^{-1}}\times \widetilde{G}_n \right] & \leq  \E_{\Q^{(\alpha)}}\bigg[\frac{n^{-1}}{D_n^{(\alpha)}}\times \frac{n}{\Ren(\alpha+b_n/2)}{ \bf 1}_{E_n}\bigg] 
& \leq  \frac{c}{1+\alpha+b_n/2}\E_{\Q^{(\alpha)}}\Big[\frac{1}{D_n^{(\alpha)}}\Big] \\
& \leq {c'} n^{-1/2},
\end{align*}
since $\E_{\Q^{(\alpha)}}[(D_{n}^{(\alpha)})^{-1}]={\Ren(\alpha)}^{-1}\leq1$. As a consequence,
\begin{align*}
LHS_{\eqref{limH2}}\leq\frac{c'}{\sqrt{n}}+\E_{\Q^{(\alpha)}}\bigg[\widetilde W_n \widetilde{G}_n \bigg] 
\leq \frac{c'}{\sqrt{n}}+ {\E_{\Q^{(\alpha)}}\bigg[\widetilde{W}_n\times \un_{\{V(w_{k_n})\in[k_n^{1/3},k_n]\}}\bigg]} \times\sup_{u\in[k_n^{1/3},k_n]}{\E_{\Q^{(\alpha)}}\bigg[ \widetilde{G}_n \Big\vert V(w_{k_n})=u\bigg]},
\end{align*}
where the second inequality follows from Markov property at $k_n$. Let 
\begin{align*}
RHS_1:={\E_{\Q^{(\alpha)}}\bigg[\widetilde{W}_n \times \un_{\{V(w_{k_n})\in[k_n^{1/3},k_n]\}}\bigg]}, \quad  
RHS_2(u):=\E_{\Q^{(\alpha)}}\bigg[\widetilde{G}_n \Big\vert V(w_{k_n})=u\bigg].
\end{align*}
 Next we are going to show that
\begin{align}
&\limsup_{n\rightarrow\infty}RHS_1\leq \Cb_{a,b}, \textrm{ and}\label{rhs1}\\
&\limsup_{n\rightarrow\infty}\sup_{u\in[k_n^{1/3},k_n]}RHS_2(u)\leq \Cb_{a,b}.\label{rhs2}
\end{align}
For $RHS_1$, note that by Markov property  
\[
RHS_1\times{ \inf_{u\in[k_n^{1/3},k_n]}\Q^{(\alpha)}(E_n\vert V(w_{k_n})=u)} \leq \E_{\Q^{(\alpha)}}\left[\widetilde{W}_n \times \un_{E_n}\right].
\]
By \eqref{Qu}, $\inf_{u \in [k_n^{1/3},k_n]}\Q^{(\alpha)}(E_n\vert V(w_{k_n})=u) =1+o_n(1)$, therefore,
\begin{align*}
RHS_1\leq& (1+o_n(1)) \E_{\Q^{(\alpha)}}\left[\widetilde{W}_n \times \un_{E_n}\right]\\
\leq &(1+o_n(1)) \E_{\Q^{(\alpha)}}\left[\widetilde{G}_n \times \un_{E_n}\un_{ \{ D_n^{(\alpha)}\geq n^{-3/2}\}}\right]+2n\Q^{(\alpha)}\left((D_{n}^{(\alpha)})^{-1}>n^{3/2}\right),
\end{align*}
since $W_n^{(\alpha),[0,k_n)}(G)\leq \sqrt{n} D_n^{(\alpha),[0,k_n)}$. Again by Markov inequality with $\E_{\Q^{(\alpha)}}[(D_{n}^{(\alpha)})^{-1}]={\Ren(\alpha)}^{-1}\leq1$, 
\[
2n\Q^{(\alpha)}\left((D_{n}^{(\alpha)})^{-1}>n^{3/2}\right)\leq {2} n^{-1/2}.
\]
On the other hand, given $E_n\cap\{D_n^{(\alpha)}\geq n^{-3/2}\}$, $D_n^{(\alpha),[k_n, n]}\leq n^{-2}\leq D_n^{(\alpha)}/\sqrt{n}$. So, 
\[
D_n^{(\alpha),[0,k_n)} = D_n^{(\alpha)}- D_n^{(\alpha),[k_n, n]}\geq (1-1/\sqrt{n})D_n^{(\alpha)}.
\] Consequently,
\[ RHS_1\leq  (1+o_n(1)) \E_{\Q^{(\alpha)}}\left[\widetilde{W}_n \times\un_{E_n}\un_{\{D_n^{(\alpha)}\geq n^{-3/2}\}}\right]+\frac{2}{\sqrt{n}}
\leq (1+o_n(1)) \E_{\Q^{(\alpha)}}\Big[\frac{\sqrt{n}W_n^{(\alpha)}(F_{a_n,b_n})}{D_n^{(\alpha)}} \Big]+\frac{2}{\sqrt{n}}. \]
So \eqref{rhs1} follows from \eqref{F}.\\
 It remains to prove \eqref{rhs2}.  Let $m:=n-k_n$ and $m_0:=n_0-k_n$, for any $u\in[k_n^{1/3},k_n]$, $RHS_2(u)$ is bounded by
 \[
 \E_{\Q^{(\alpha)}}\left[ \frac{n}{\Ren(\alpha+V(w_m))}\frac{e^{V(w_m)}}{\sum_{0<j\leq m}e^{V(w_m)}}\un_{\{ \Upsilon_{w_n}>m_0, \overline{V}(w_m)\geq b_n,\max_{k\leq n}(\overline{V}(w_k)-V(w_k))\leq a_n\}}\Big\vert V(w_{0})=u\right]
 \]
 which by Markov property and \eqref{rw+} is less than
 \[
 \frac{n}{\Ren(\alpha+u)}\E\left[\frac{e^{S_m}}{\sum_{1\leq j\leq m}e^{S_m}}; \max_{i\leq m}(\overline{S}_i-S_i)\leq a_n, \Upsilon_S>m_0, \overline{S}_m\geq b_n-u, \underline{S}_m\geq -\alpha-u\right].
 \]
Following the same arguments used for \eqref{G}, one obtains that for all $u\in[k_n^{1/3},k_n]$, $RHS_2(u)\leq \Cb_{a,b}+o_n(1)$,
which completes the proof of \eqref{rhs2}  and conclude \eqref{limH2}. \hfill $\square$

 \subsection{Proof of Corollary \ref{sumW} \label{sec3.3}}
 In this subsection, we show that as $\beta\rightarrow\infty$,
 \begin{equation}\label{sumWo}
\sum_{m=1}^\infty \sum_{|z|=m}\frac{1}{\sum_{\phi<y\leq z}e^{V(y)}}\un_{\{ \max_{\phi<y\leq z(\overline{V}(y)-V(y))\leq \beta, \overline{V}(z)\geq \beta\pm O(\log \beta)}\}}\xrightarrow{in\ \P^* probability} \Lambda D_\infty.
 \end{equation}
 \begin{proof}
Denote 
\[
W^*_m(\beta):=\sum_{|z|=m}\frac{1}{\sum_{\phi<y\leq z}e^{V(y)}}\un_{\{\max_{\phi<y\leq z(\overline{V}(y)-V(y))\leq \beta, \overline{V}(z)\geq \beta}\}}=W_m(F_{\beta,\beta})/\sqrt{m}.
\]
 In fact, only those $m$ that are comparable to $\beta^2$ really contribute to the sum. First, for $m\leq \varepsilon\beta^2$ and $m\geq \beta^2/\varepsilon$ with $\varepsilon\downarrow0$, we claim that for any $\eta>0$,
\begin{align}
\lim_{\varepsilon\rightarrow 0}\limsup_{\beta\rightarrow\infty}\P\left(\sum_{m\leq \varepsilon\beta^2}W_m^*(\beta)\geq \eta\right)&=0\label{smallsum}\\
\lim_{\varepsilon\rightarrow 0}\limsup_{\beta\rightarrow\infty}\P\left(\sum_{m\geq \beta^2/\varepsilon}W_m^*(\beta)\geq \eta\right)&=0\label{largesum}.
\end{align}
We postpone the proof of the above facts to Subsection \ref{Sec4.3.4} as the arguments are similar to the proof of \eqref{outside}.

For any $\varepsilon>0$ fixed, by \eqref{unifcvg}, as $\beta\rightarrow\infty$, \bl{$mW^*_m(a\sqrt{m})$} converges in probability to $D_\infty\mathscr{C}_{a,a}$ uniformly on $a\in\r_+$. Moreover, $\sum_{\varepsilon\beta^2\leq m\leq \beta^2/\varepsilon} \frac{1}{m}<\infty$ for any $\varepsilon>0$.  { Therefore,
\begin{align*}
\sum_{\varepsilon\beta^2\leq m\leq \beta^2/\varepsilon} W^*_m(\beta)= &D_\infty \sum_{\varepsilon\beta^2\leq m\leq \beta^2/\varepsilon} \frac{1}{m}\left( \mathscr{C}_{\beta/\sqrt{m},\beta/\sqrt{m}}+o_{\P^*}(1)\right)\\
=&D_\infty \sum_{\varepsilon\beta^2\leq m\leq \beta^2/\varepsilon} \frac{1}{m} \mathscr{C}_{\beta/\sqrt{m},\beta/\sqrt{m}}+o_{\P^*}(1).
\end{align*} 
where $o_{\P^*}(1)$ denotes a term such that $\lim_{\beta\rightarrow\infty}o_{\P^*}(1)=0$ in $\P^*$-probability.} On the other hand, by change of variables $m=\gamma \beta^2$, 
\[
\int_{\varepsilon\beta^2}^{\beta^2/\varepsilon} \mathscr{C}_{\beta/\sqrt{m},\beta/\sqrt{m}}\frac{dm}{m}=\int_\varepsilon^{1/\varepsilon} \mathscr{C}_{\gamma^{-1/2},\gamma^{-1/2}}\frac{d\gamma}{\gamma}.
\]
As $\mathscr{C}_{a,b}$ is continuous and monotone, we get that
\[
\sum_{\varepsilon\beta^2\leq m\leq \beta^2/\varepsilon} \frac{1}{m} \mathscr{C}_{\beta/\sqrt{m},\beta/\sqrt{m}}=\int_\varepsilon^{1/\varepsilon} \mathscr{C}_{\gamma^{-1/2},\gamma^{-1/2}}\frac{d\gamma}{\gamma}+o_\beta(1).
\]
When $\varepsilon\rightarrow 0$, $\int_\varepsilon^{1/\varepsilon} \mathscr{C}_{\gamma^{-1/2},\gamma^{-1/2}}\frac{d\gamma}{\gamma}\rightarrow \Lambda$ because of Lemma \ref{A}. In view of \eqref{smallsum} and \eqref{largesum}, we conclude that in $\P^*$-probability,
\[
\lim_{\beta\rightarrow\infty}\sum_{m=1}^\infty W^*_m(\beta)=\Lambda D_\infty.
\]
Note that if we replace $W_m(F_{\beta,\beta})$ by $W_m(F_{\beta,(1\pm\epsilon)\beta})$ with $\beta\in(0,1)$, these arguments still work. By monotonicity of $F$, we have
\begin{multline*}
\sum_{m=1}^\infty W_m(F_{\beta,(1-\epsilon)\beta})\leq\sum_{m=1}^\infty\sum_{|z|=m}\frac{1}{\sum_{\phi<y\leq z}e^{V(y)}}\un_{ \{ \max_{\phi<y\leq z}(\overline{V}(y)-V(y))\leq \beta, \overline{V}(z)\geq \beta\pm O(\log \beta) \} }\leq \sum_{m=1}^\infty W_m(F_{\beta,(1+\epsilon)\beta}).
\end{multline*}
By integrability and continuity of $\mathscr{C}$, as $\beta\rightarrow\infty$, $\sum_{m=1}^\infty W_m(F_{\beta,(1-\epsilon)\beta})=(\Lambda+o_\epsilon(1)) D_\infty+o_{\P^*}(1)$.
Consequently, the convergence \eqref{sumWo} holds.
\end{proof}
 
 \section{Variance of $K_n$ and secondary results}
 \label{sec4}
 
 In this section, we {complete} the proof of the main theorems by proving {Lemmata} \ref{addU}, \ref{usesn}, \ref{othergenerations} and Proposition \ref{lemCentredK}.

 \subsection{Variance of $K^{B^\delta\cap U}_n(\ell)$ and Proof of Proposition \ref{lemCentredK} \label{QuVar}}

Recall the definition of $B^{\delta}\cap U$ in Section \ref{Qexp}, in this section we focus on the mean of the quenched variance of $K_n^{B^{\delta}\cap U}(\ell)$ which is a key step in the proof of Proposition \ref{lemCentredK}.

\subsubsection{Quenched expression for the variance}

\begin{lem} \label{varq} Recall that $a_z=\p^{\mathcal{E}}(T_z<T_\phi)$ and let $a_{v,z}:=\p^{\mathcal{E}}(T_v\wedge T_z<T_\phi)$. For every event $A$ measurable with respect to $\mathcal{E}$, denote
{ the quenched variance of $K_n^A(\ell)$ as follows:
\[
\Ve(K_n^A(\ell)):=\e^\en\left[ \left(K_n^A(\ell)-\mathcal{K}_n^A(\ell) \right)^2\right],
\]}
then 
\begin{multline*}
\Ve(K_n^A(\ell)) = \sum_{|z|= \ell}\left[(1-a_z)^n-(1-a_z)^{2n}\right]\un_{\{z \in A\} }\\
+  \sum_{|z|=\ell,|v|=\ell, z \neq v} \left[(1-a_{v,z})^n-(1-a_z)^n(1-a_v)^n\right]\un_{\{z \in A\}}\un_{\{v \in A\}}  .
\end{multline*}
\end{lem}

\begin{proof}
{Note that 
\begin{align*}
K_n^A(\ell)-\mathcal{K}_n^A(\ell)&=\sum_{|z|= \ell} \left(\un_{T_z<T_\phi^n}-(1-(1-a_z)^n)\right) \un_{\{z \in A\} } =\sum_{|z|= \ell} \left((1-a_z)^n -\un_{T_z\geq T_{\phi}^n} \right) \un_{\{z \in A\}}.
\end{align*} So the lemma  comes directly.}
\end{proof}

\noindent A corollary of this Lemma is the following result, which gives a simple upper bound of the quenched variance when $A=B^\delta\cap U$ :  

\begin{lem} Recall the definition of $B^{\delta}\cap U$ in Section \ref{Qexp}, we have :
\begin{multline}\label{veKbd}
\Ve\left({K}_n^{B^\delta\cap U}(\ell) \right) \leq \sum_{\substack{|z|=|v|= \ell, \\ z\neq v}}\left[n a_z \p_{v\wedge z}^{\mathcal{E}}(T_v<T_\phi)\un_{\{z, v\in B^\delta\cap U\}}+n a_v \p_{v\wedge z}^{\mathcal{E}}(T_z<T_\phi)\un_{\{z,v\in B^\delta\cap U\}}\right]\\
+\sum_{|z|=\ell}na_z\un_{\{z\in B^\delta\cap U\}},
\end{multline}
where $v\wedge z$ is the latest common ancestor of $v$ and $z$ in the tree $\T$, and $\p^{\en}_y$ is the quenched probability of the random walk started from $y$.
\end{lem}
\begin{proof}
This upper bound is actually true for every truncated version of $K_n(\ell)$, however it is optimized here for events included in $U$, in particular for $B^\delta\cap U$. {For $a_{v,z}$ one sees that}
\begin{align*}
a_{v,z}=\p^{\mathcal{E}}(T_v < T_z \wedge T_\phi)+\p^{\mathcal{E}}(T_z < T_v \wedge T_\phi)=:d_{v,z}+d_{z,v}.
\end{align*} 
We have, 
\begin{align*}
(1-a_{v,z})^n-(1-a_z)^n(1-a_v)^n & \leq (1-d_{v,z}-d_{z,v})^n-(1-a_z-a_v)^n
 \leq n (a_z-d_{z,v}+a_v-d_{v,z}) .
\end{align*}
Observe that 
\begin{align*}
a_z-d_{z,v}+a_v-d_{v,z}=&\p^{\mathcal{E}}(T_v \vee T_z < T_\phi)
 \leq \p^{\mathcal{E}}(T_z < T_\phi)\p^{\mathcal{E}}_{z \wedge v}(T_v < T_\phi)+\p^{\mathcal{E}}(T_v< T_\phi)\p^{\mathcal{E}}_{z \wedge v}(T_z < T_\phi)\\
 =& a_z\p^{\mathcal{E}}_{z \wedge v}(T_v < T_\phi)+a_v\p^{\mathcal{E}}_{z \wedge v}(T_z < T_\phi).
\end{align*}
 This together with Lemma \ref{varq} yields that
 \begin{multline*}
   \sum_{\substack{|z|=\ell,|v|=\ell,\\ z \neq v}} \left[(1-a_{v,z})^n-(1-a_z)^n(1-a_v)^n\right]\un_{\{z \in B^\delta\cap U\}}\un_{\{v \in B^\delta\cap U\}}\\
\leq    \sum_{\substack{|z|=|v|= \ell,\\ z\neq v}}n a_z \p_{v\wedge z}^{\mathcal{E}}(T_v<T_\phi)\un_{\{z\in B^\delta\cap U,v\in B^\delta\cap U\}}+ \sum_{\substack{|z|=|v|=\ell,\\ z\neq v}}n a_v \p_{v\wedge z}^{\mathcal{E}}(T_z<T_\phi)\un_{\{z\in B^\delta\cap U,v\in B^\delta\cap U\}}
 \end{multline*}
 Moreover, we have $(1-a_z)^n-(1-a_z)^{2n}\leq na_z$. This leads to \eqref{veKbd}.
\end{proof}

\subsubsection{Upper bound for the mean of the quenched variance}

In this section we obtain an upper bound of the mean $\E\left(\Ve\left({K}^{B^\delta\cap U}_n(\ell) \right)\right) $.

\begin{lem} \label{mvar} For $\ell\sim \gamma(\log n)^2$, every $\delta>0$ and $n$ large enough,
\[\E\left(\Ve\left({K}^{B^\delta\cap U}_n(\ell) \right)\right) \leq  c {n^2}{(\log n)^{-\delta+1/2}}.
 \]
\end{lem}

\begin{proof} 
Because of \eqref{veKbd}, we only have to bound the means of 
\[
t_n:= \sum_{|z|=|v|= \ell,\ z\neq v} a_z \p_{v\wedge z}^{\mathcal{E}}(T_v<T_\phi)\un_{\{z,v\in B^{\delta}\cap U\}}, \widetilde{\mathcal{K}}_n^{B^\delta\cap U}(\ell)=\sum_{|z|=\ell}na_z\un_{\{ z\in B^\delta\cap U\} }
\]
since the second term on the RHS of \eqref{veKbd} is its symmetric. We begin with $\widetilde{\mathcal{K}}_n^{B^\delta\cap U}(\ell)$. As $B_2^\delta\subset L_\delta$, recalling \eqref{az} and \eqref{KtoW}, one sees that $\widetilde{\mathcal{K}}_n^{B^\delta\cap U}(\ell)= \frac{n}{\sqrt{\ell}}p^\en(\phi,\pa{\phi})W_\ell^{(\alpha)}(F_{\log s_n, \log n+\log\log n})$ since the truncated martingale-like variable is obtained by adding the restriction $B_1$. By \eqref{expWF}, one gets that
\begin{equation}\label{expKBU}
\E[\widetilde{\mathcal{K}}_n^{B^\delta\cap U}(\ell)]=\Theta(\frac{n}{\ell})=\Theta(\frac{n}{(\log n)^2}).
\end{equation}

The main idea of the rest proof, is to decompose the double sum $\sum_{|z|=|v|=\ell}$ according to the latest common ancestor $z\wedge v$. 

Define $\sum_1 := \sum_{\phi<s\leq v \wedge z}e^{V(s)-V(z \wedge v)}$,  
$ \sum_2 := \sum_{v \wedge z<s\leq z}e^{V(s)-V(z \wedge v)}$ and 
$ \sum_3 := \sum_{v \wedge z<s\leq v}e^{V(s)-V(z \wedge v)}.$
We then have \[a_z= \frac{p^\en(\phi,\overleftarrow{\phi})e^{-V(v\wedge z)}}{\sum_1+ \sum_2}  \textrm{ and }   \p^{\mathcal{E}}_{z \wedge v}(T_v < T_\phi)=\frac{\sum_1}{\sum_1+ \sum_3}.\]
By comparing $\Sigma_1,\ \Sigma_2,\ \Sigma_3$, we get $ t_n  \leq  t_n^1+t_n^2+t_n^3+t_n^4$, with 
\begin{align*}
t_n^1 & :=\sum_{\substack{|z|=|v|= \ell\\ z\neq v}}  \frac{e^{-V(z \wedge v)}}{ \sum_1} \un_{\{z\in B^\delta\cap U,v\in B^\delta\cap U, \sum_1 \geq  \sum_2 \vee \sum_3 \}}, \\ 
t_n^2 &  := \sum_{\substack{|z|=|v|= \ell\\ z\neq v}}  \frac{e^{-V(z \wedge v)}}{ \sum_3}  1_{\{z\in B^\delta\cap U,v\in B^\delta\cap U, \sum_2  \leq \sum_1 \leq  \sum_3 \}}, t_n^3:=\sum_{\substack{|z|=|v|= \ell\\ z\neq v}}  \frac{e^{-V(z \wedge v)}}{ \sum_2} 1_{\{z\in B^\delta\cap U,v\in B^\delta\cap U, \sum_3  \leq \sum_1 \leq  \sum_2 \}}, \\
t_n^4 & := \sum_{\substack{|z|=|v|= \ell\\ z\neq v}} e^{-V(z \wedge v)}  \frac{ \sum_1}{ \sum_2 * \sum_3 }  \un_{\{z\in B^\delta\cap U,v\in B^\delta\cap U, \sum_1 \leq  \sum_2  \wedge  \sum_3 \}}. \end{align*}
We treat each term separately. Notice that by symmetry $\E(t_n^2)=\E(t_n^3)$, so we  only estimate $\E(t_n^1)$, $\E(t_n^2)$ and $\E(t_n^4)$. 

Recall that for every $z \in U$, $\overline{V}(z)\geq \log n+\log\log n$ and $a_z \leq (n \log n)^{-1}$. Clearly, $\{\Sigma_1+\Sigma_2\geq e^{\MV(z)-V(z\wedge v)}\geq n\log ne^{-V(z\wedge v)}\}$. In addition, if $\{\Sigma_1\geq \Sigma_2\}$, we have
\[
\MV(z\wedge v) \geq \log \frac{\Sigma_1e^{V(z\wedge v)}}{|z\wedge v|} \geq \log \left(\frac{n\log n}{2\ell}\right).
\]
\emph{$*$ Upper bound for $\E(t_n^1)$}, as $\Sigma_1$ is the largest term here, using the above remark  we have $\{z,v\in U, \Sigma_1 \geq \Sigma_2 \vee \Sigma_3 \} \subset \{ \overline{V}(z \wedge v) > \log n+ \log \log n-\log 2\ell=:m_n \} $, also as $z \in  B_2^{\delta} $,  $\Sigma_1 \leq s_n = n / (\log n)^{1+ \delta} $, so 
\begin{align*}
 t_n^1  & \leq 
\sum_{
\substack{|z|=|v|=\ell\\ z\neq v}} \frac{e^{-V(z \wedge v)}}{ \sum_1} \un_{\{ \overline{V}(z \wedge v) > m_n , \Sigma_1 \leq s_n ,\sum_1 \geq  \sum_2 \vee \sum_3, \ \underline{V}(z) \wedge  \underline{V}(v) \geq - \alpha  \}}\\
&\leq \sum_{j=0}^{\ell-1}\sum_{|u|=j}\frac{e^{-V(u)}}{\sum_1^u}\un_{\{ \overline{V}(u) > m_n , \Sigma_1^u \leq s_n, \underline{V}(u)\geq-\alpha \} }\sum_{\substack{\overleftarrow{x}=u=\overleftarrow{y} \\ x\neq y}}\sum_{|z|=\ell, z\geq x}\un_{\{\sum_2^{x,z}e^{V(x)-V(u)}\leq \sum_1^u\}}\sum_{|v|=\ell, v\geq y}\un_{\{\sum_2^{y,v}e^{V(y)-V(u)}\leq \sum_1^u\}},
\end{align*} 
where $\Sigma_1^u:=\sum_{\phi<s\leq u}e^{V(s)-V(u)}$ and $\Sigma_2^{x,z}:=\sum_{x\leq s\leq z}e^{V(s)-V(x)}$ and recall that $\pa{x}$ is the parent of $x$.

Applying Markov property at time $|u|+1$ and then Many-to-one equation \eqref{M2O} yields
\begin{align*}
\E[t_n^1]\leq & \sum_{j=0}^{\ell-1}\E\Big[\sum_{|u|=j}\frac{e^{-V(u)}}{\sum_1^u}\un_{ \{ \overline{V}(u) > m_n , \Sigma_1^u \leq s_n, \underline{V}(u)\geq-\alpha \} }\sum_{\substack{\overleftarrow{x}=u=\overleftarrow{y}\\ x\neq y}}f_{j,\ell}(\Sigma_1^u e^{V(u)-V(x)})f_{j,\ell}(\Sigma_1^u e^{V(u)-V(y)})\Big],
\end{align*}
where $f_{j,\ell}(t):=\E\big[e^{S_{\ell-1-j}}; \sum_{i=0}^{\ell-1-j}e^{S_i}\leq t\big]$.
By \eqref{eSMSbd+}, $f_{j,\ell}(t)\leq \E(e^{S_{\ell-1-j}}; \MS_{\ell-1-j}\leq \log_+t)\leq c  (\log_+t +1) t / (\ell-j)^{3/2}$. Plugging this into the previous inequality yields 
\begin{align*}
&\E[t_n^1]\leq  \sum_{j=0}^{\ell-1}\E\Big[\sum_{|u|=j}e^{-V(u)}\Sigma_1^u(1+\log \Sigma_1^u)^2\un_{ \{ \overline{V}(u) > m_n , \Sigma_1^u \leq s_n, \underline{V}(u)\geq-\alpha \}} \\
&\qquad \times\frac{c}{(\ell-j)^{3}} \sum_{\substack{\overleftarrow{x}=u=\overleftarrow{y}\\ x\neq y}}[1+(V(u)-V(x))_+]e^{V(x)-V(u)}[1+(V(u)-V(y))_+]e^{V(y)-V(u)}\Big]\\
&\leq \sum_{j=0}^{\ell-1}\frac{c}{(\ell-j)^{3}}\E\Big[\sum_{|u|=j}e^{-V(u)}\Sigma_1^u(1+\log \Sigma_1^u)^2\un_{ \{ \overline{V}(u) > m_n , \Sigma_1^u \leq s_n, \underline{V}(u)\geq-\alpha \} }\Big]\E\Big[\Big(\sum_{|x|=1}[1+(-V(x))_+]e^{-V(x)}\Big)^2\Big].
\end{align*}
By  \eqref{M2O} and hypothesis \eqref{hyp1}, we get  $\E(t_n^1) \leq   \sum_{j=0}^{\ell-1} \frac{c}{(\ell-j)^{3}} s_n(1+\log_+ s_n)^2\P\left[\overline{S}_j > m_n,\  \Sigma_1^{S} \leq s_n,\  \underline{S}_j \geq - \alpha\right]$, 
with  $ \sum_1^S:= \sum_{i=1}^je^{S_i-S_j}$. Also by \eqref{mSbd}, $\P\left[  \underline{S}_j \geq - \alpha  \right]\leq c(1+ \alpha)j^{-1/2}$, so
\begin{align*}
\E(t_n^1) \leq & \sum_{j=0}^{\ell-1} \frac{c}{(\ell-j)^{3}} s_n(1+\log_+ s_n)^2 \frac{(1+\alpha)}{(j+1)^{1/2}}\leq \frac{c(1+\alpha) n}{(\log n)^{\delta-1/2}}.
\end{align*} 
\\[2pt]
\emph{$*$ Upper bound for $\E(t_n^2)$}, with the same ideas as for the upper bound of $t_n^1$, we have 
\begin{align*}
t_n^2 \leq &\sum_{\substack{|z|=|v|= \ell\\ z\neq v}} \frac{e^{-V(z \wedge v)}}{ \sum_3}  \un_{\{ \overline{V}(z \wedge v) > m_n , \Sigma_1 \leq s_n ,\Sigma_2 \leq  \Sigma_1, \underline{V}(z \wedge v)  \geq - \alpha\} }\\
\leq&\sum_{j=0}^{\ell-1}\sum_{|u|=j}e^{-V(u)}\un_{ \{ \overline{V}(u)\geq m_n, \Sigma_1^u\leq s_n,\underline{V}(u)\geq-\alpha \}}\sum_{\substack{\overleftarrow{x}=u=\overleftarrow{y}\\ x\neq y}}\sum_{|z|=\ell, z\geq x}\un_{\{ \Sigma_2^{x,z}e^{V(x)-V(u)} \leq s_n \} }\sum_{|v|=\ell, v\geq y}\frac{1}{\Sigma_2^{y,v}e^{V(y)-V(u)}}. 
\end{align*}
By Markov property then by \eqref{M2O}, it follows that
\begin{equation*}
 \E(t_n^2) \leq  \sum_{j=0 }^{\ell-1} \E\Big[\sum_{|u|=j}e^{-V(u)}\un_{\{ \underline{V}(u)\geq-\alpha \}}\sum_{\substack{\overleftarrow{x}=u=\overleftarrow{y}\\ x\neq y}}e^{V(u)-V(y)}\E\left[\frac{e^{S_{\ell-j-1}}}{\sum_{i=0}^{\ell-j-1}e^{S_i}}\right]f_{j,\ell}\Big( s_ne^{V(u)-V(x)}\Big) \Big],
\end{equation*}
which by \eqref{eSMSbd+} and \eqref{eSbd}, is less than
\begin{align*}
&\sum_{j=0 }^{\ell-1} \frac{c s_n(1+\log_+s_n)}{(\ell-j)^2}\E\Big[\sum_{|u|=j}e^{-V(u)}\un_{\{ \underline{V}(u)\geq-\alpha \} }\sum_{\substack{\overleftarrow{x}=u=\overleftarrow{y}\\ x\neq y}}e^{V(u)-V(y)}[1+(V(u)-V(x))_+]e^{V(u)-V(x)}\Big]\\
\leq& \sum_{j=0 }^{\ell-1} \frac{c s_n(1+\log_+s_n)}{(\ell-j)^2}\E\Big[\sum_{|u|=j}e^{-V(u)}\un_{\{ \underline{V}(u)\geq-\alpha \} }\Big]\E\Big[\Big(\sum_{|x|=1}[1+(-V(x))_+]e^{-V(x)}\Big)^2\Big].
\end{align*}
Applying again \eqref{M2O}, \eqref{hyp1}, and then \eqref{mSbd} we have,
\begin{align*}
\E(t_n^2)\leq & \sum_{j=0 }^{\ell-1} \frac{c s_n(1+\log_+s_n)}{(\ell-j)^2}\P\left(\mS_j\geq-\alpha\right) \leq \frac{c(1+\alpha)n}{(\log n)^{1+	\delta}}.
\end{align*}
\emph{$*$ Upper bound for $\E(t_n^4)$}, we have :
\begin{align*}
t_n^4 & \leq \sum_{\substack{|z|=|v|= \ell\\ z\neq v}}e^{-V(z \wedge v)}  \frac{ \sum_1}{ \sum_2 * \sum_3 }  \un_{\{\underline{V}(z \wedge v)  \geq - \alpha,  \sum_1 \leq s_n\}}\\
&\leq \sum_{j=0}^{\ell-1}\sum_{|u|=j}e^{-V(u)}s_n\un_{\{ \underline{V}(u)\geq-\alpha \}}\sum_{\substack{\overleftarrow{x}=u=\overleftarrow{y}\\ x\neq y}}\sum_{|z|=\ell, z\geq x}\frac{1}{\Sigma_2^{x,z}e^{V(x)-V(u)}}\sum_{|v|=\ell, v\geq y}\frac{1}{\Sigma_2^{y,v}e^{V(y)-V(u)}}.
\end{align*}
 With the same arguments as above, one sees that
\begin{align*}
\E(t_n^4)\leq& \sum_{j=0}^{\ell-1}\E\Big[\sum_{|u|=j}e^{-V(u)}s_n\un_{\{ \underline{V}(u)\geq-\alpha \}}\sum_{\substack{\overleftarrow{x}=u=\overleftarrow{y}\\ x\neq y}}e^{2V(u)-V(x)-V(y)}\E\Big[\frac{e^{S_{\ell-j-1}}}{\sum_{i=0}^{\ell-j-1}e^{S_i}}\Big]^2\Big],
\end{align*}
which by \eqref{eSbd} and \eqref{hyp1} is less than $\sum_{j=0}^{\ell-1}\frac{cs_n}{(\ell-j)}\E\Big[\sum_{|u|=j}e^{-V(u)}\un_{\{ \underline{V}(u)\geq-\alpha\}}\Big]$. Once again by \eqref{M2O} and \eqref{mSbd}, we end up with
\[
\E(t_n^4)\leq \sum_{j=0}^{\ell-1}\frac{cs_n}{(\ell-j)}\P(\mS_j\geq-\alpha)\leq \sum_{j=0}^{\ell-1}\frac{cs_n(1+\alpha)}{(\ell-j)(j+1)^{1/2}}\leq \frac{c(1+\alpha)n(\log\log n)}{(\log n)^{2+\delta}}.
\]
Consequently, we have  $t_n\leq c{n}/{(\log n)^\delta}$, which concludes the proof.
\end{proof}

\subsection{Complementary arguments: Proofs of Lemmas \ref{addU}, \ref{usesn}, \ref{othergenerations} and Proposition \ref{lemCentredK} \label{Sec4.2}}

\subsubsection{Proof of Lemma \ref{addU}}

In fact, as in \eqref{nazK}, $\e[K_n^{B\cap U}(\ell)]=\Theta(\e[\widetilde{\mathcal{K}}_n^{B\cap U}(\ell)])$. Then similar to \eqref{expKBU}, $\e[K_n^{B\cap U}(\ell)]=\Theta(n/(\log n)^2)$. Let us show that $\e[K_n^{B\setminus U}(\ell)]=o(n/(\log n)^2)$.

Note that $\p^\en(T_z<T_\phi^n)\leq na_z\wedge 1$. We have
\begin{align*}
\e[K_n^{B\setminus U}(\ell)] & \leq \E\Big( \sum_{|z|= \ell} (na_z\wedge 1)\un_{\{ \overline{V}(z) \leq \log n+ \log \log n, \underline{V}(z) \geq - \alpha \} }\Big)\\
&\leq \E\Big( \sum_{|z|= \ell} na_z \un_{\{ \log n-3\log \log n\leq \overline{V}(z) \leq \log n+ \log \log n, \underline{V}(z) \geq - \alpha \}}\Big)+\E\Big( \sum_{|z|= \ell} \un_{\{ \overline{V}(z) \leq \log n-3 \log \log n \} }\Big).
\end{align*}
It follows immediately from \eqref{M2O} that
\[
\E\Big( \sum_{|z|= \ell} \un_{\{ \overline{V}(z) \leq \log n-3 \log \log n\} }\Big)=\E\left(e^{S_{\ell}}\un_{ \{ \overline{S}_\ell \leq  \log n-3\log\log n \} }  \right)\leq e^{\log n-3\log\log n}=\frac{n}{(\log n)^3}=o(\frac{n}{(\log n)^2}).
\]
On the other hand, for the second term, as $a_z\leq e^{-\overline{V}(z)}$,
\begin{align*}
&\E\Big[\sum_{|z|=\ell}na_z\un_{\{ \log n-3\log\log n\leq \overline{V}(z)\leq \log n+\log\log n, \underline{V}(z)\geq-\alpha \} }\Big]\\
\leq &n\E\Big[\sum_{|z|=\ell}e^{-V(z)}e^{V(z)-\overline{V}(z)}\un_{\{ \log n-3\log\log n\leq \overline{V}(z)\leq \log n+\log\log n,\underline{V}(z)\geq-\alpha\} }\Big],
\end{align*}
which by \eqref{M2O} is equal to
\[
n\E\Big[e^{S_\ell-\MS_\ell}, \log n-3\log\log n\leq \MS_\ell\leq \log n+\log\log n,\mS_\ell\geq-\alpha\Big].
\]
By Markov property at the {first hitting time}  $\MS_\ell$, one sees that
\begin{align*}
&\E\Big[\sum_{|z|=\ell}na_z\un_{\{ \log n-3\log\log n\leq \overline{V}(z)\leq \log n+\log\log n,\mS_\ell\geq-\alpha\} }\Big]\\
\leq&  n\sum_{j=1}^\ell\E\Big[e^{S_\ell-S_j}; \MS_{j-1}<S_j, S_j=\MS_\ell\in[\log n-3\log\log n, \log n+\log \log n], \mS_j\geq-\alpha\Big]\\
\leq& n\sum_{j=1}^\ell\P\Big(\mS_j\geq-\alpha, \MS_j=S_j\in[\log n-3\log\log n, \log n+\log \log n]\Big)\E\Big[e^{S_{\ell-j}};\MS_{\ell-j}\leq0\Big].
\end{align*}
By \eqref{mSMSSbd} and \eqref{eSMSbd}, one obtains that
\begin{align*}
\E\Big( \sum_{|z|= \ell} na_z \un_{\{ \log n-3\log \log n\leq \overline{V}(z) \leq \log n+ \log \log n, \underline{V}(z) \geq - \alpha \} }\Big) \leq c \frac{n \log \log n}{ (\log n)^3}=o(\frac{n}{(\log n)^2}),
\end{align*}
which completes the proof. \hfill $\square$
\subsubsection{proof of Lemma \ref{usesn}}
The quenched mean $\mathcal{K}_n^{(B\cap U)\setminus (B^\delta\cap U)}(\ell)$ of $K_n^{B\cap U}(\ell)-K_n^{B^\delta\cap U}(\ell)$ satisfies that
\begin{align*}
0\leq\mathcal{K}_n^{(B\cap U)\setminus (B^\delta\cap U)}(\ell)\leq & \widetilde{\mathcal{K}}_n^{(B\cap U)\setminus (B^\delta\cap U)}(\ell)=\sum_{|z|=\ell}na_z\un_{\{ z\in (B\setminus B^\delta)\cap U\}}.
\end{align*}
As $\{z\in B\setminus B^\delta\}$ implies that $\log s_n-\log \ell< \max_{\phi<y\leq z}(\overline{V}(y)-V(y))\leq \log n$, similarly to \eqref{KtoW}, we have
\begin{align*}
\widetilde{\mathcal{K}}_n^{(B\cap U)\setminus (B^\delta\cap U)}(\ell)\leq 
\frac{n}{\sqrt{\ell}}\Big(W_\ell^{(\alpha)}(F_{\log n, \log n+\log\log n})-W_\ell^{(\alpha)}(F_{\log s_n-\log \ell, \log n+\log\log n})\Big),
\end{align*}
with $s_n=n/(\log n)^{1+\delta}$. Taking expectation and using change of measures \eqref{changeprobab} yields that
\begin{align}\label{KlessF}
\E\left[\widetilde{\mathcal{K}}_n^{(B\setminus B^\delta)\cap U}(\ell)\right]\leq& 
\frac{n}{\sqrt{\ell}} \Ren(\alpha)\left( \E_{\Q^{(\alpha)}}\Big[\frac{W_\ell^{(\alpha)}(F_{\log n,\log n+\log\log n})}{D_\ell^{(\alpha)}}\Big]- \E_{\Q^{(\alpha)}}\Big[\frac{W_\ell^{(\alpha)}(F_{\log s_n-\log \ell, \log n+\log\log n})}{D_\ell^{(\alpha)}}\Big]\right)
\end{align}
In view of \eqref{F}, as $\ell\sim \gamma(\log n)^2$, we have
\[
\E_{\Q^{(\alpha)}}\Big[\frac{\sqrt{\ell}W_\ell^{(\alpha)}(F_{\log n,\log n+\log\log n})}{D_\ell^{(\alpha)}}\Big]- \E_{\Q^{(\alpha)}}\Big[\frac{\sqrt{\ell}W_\ell^{(\alpha)}(F_{\log s_n-\log \ell, \log n+\log\log n})}{D_\ell^{(\alpha)}}\Big]=o_n(1),
\]
and $\frac{n}{\ell}=O(\frac{n}{(\log n)^2})$. This implies that
\[
\e[K_n^{B\cap U}(\ell)-K_n^{B^\delta\cap U}(\ell)]\leq \E\left[\widetilde{\mathcal{K}}_n^{(B\cap U)\setminus (B^\delta\cap U)}(\ell)\right]=o\Big(\frac{n}{(\log n)^2}\Big).
\]
This ends the proof of Lemma	\ref{usesn}. 

\subsubsection{Proof of Proposition \ref{lemCentredK}} 
Observe that
\begin{multline}\label{bdcentredK}
\p\left(|K_n^{B^\delta\cap U}(\ell)-\widetilde{\mathcal{K}}_n^{B^\delta\cap U}(\ell)|\geq \eta \frac{n}{(\log n)^2}\right)\\\leq 
\p\left(|K_n^{B^\delta\cap U}(\ell)-\mathcal{K}_n^{B^\delta\cap U}(\ell)|\geq \eta \frac{n}{2(\log n)^2}\right)+\P\left(|\mathcal{K}_n^{B^\delta\cap U}(\ell)-\widetilde{\mathcal{K}}_n^{B^\delta\cap U}(\ell)|\geq \eta \frac{n}{2(\log n)^2}\right).
\end{multline}
For the second term on the right hand side, by Markov inequality and \eqref{nazK}, we have
\begin{align*}
\P\left(|\mathcal{K}_n^{B^\delta\cap U}(\ell)-\widetilde{\mathcal{K}}_n^{B^\delta\cap U}(\ell)|\geq \eta {n}{(2(\log n)^2)^{-1}}\right)\leq &{2(\log n)^2}{(\eta n)^{-1}}\E\left[\widetilde{\mathcal{K}}_n^{B^\delta\cap U}(\ell)-\mathcal{K}_n^{B^\delta\cap U}(\ell)\right]\\
\leq &{2\log n}{(\eta n)^{-1}}\E\left[\widetilde{\mathcal{K}}_n^{B^\delta\cap U}(\ell)\right].
\end{align*}
In view of \eqref{KtoW}, $\E\left[\widetilde{\mathcal{K}}_n^{B^\delta\cap U}(\ell)\right] \leq  {n}{\ell^{-1/2}} \E\left[W_\ell^{(\alpha)}(F_{\log s_n, \log n+\log\log n})\right] $ which implies
\begin{align*}
\P\left(|\mathcal{K}_n^{B^\delta\cap U}(\ell)-\widetilde{\mathcal{K}}_n^{B^\delta\cap U}(\ell)|\geq \eta {n}{(2(\log n)^2)^{-1}}\right)
\leq & \frac{2c}{\eta } \Ren(\alpha) \E_{\Q^{(\alpha)}}\Big[\frac{W_\ell^{(\alpha)}(F)}{D_\ell^{(\alpha)}}\Big]=O(\frac{1}{\sqrt{\ell}}),
\end{align*}
where the last equalities come from the change of measures \eqref{changeprobab} and \eqref{F}.

For the first term on the right hand side of \eqref{bdcentredK}, using Tchebychev inequality on the quenched probability yields that
\[ \p^{\mathcal{E}} \left(|K_n^{B^\delta\cap U}(\ell)-\mathcal{K}_n^{B^\delta\cap U}(\ell)|\geq \eta \frac{n}{2(\log n)^2}\right) \leq  \frac{4(\log n)^4}{\eta^2 n^2}  \Ve\left(K_n^{B^\delta\cap U}(\ell) \right).\]  
So, 
\[
\p\left(|K_n^{B^\delta\cap U}(\ell)-\mathcal{K}_n^{B^\delta\cap U}(\ell)|\geq \eta \frac{n}{2(\log n)^2}\right)\leq \frac{4(\log n)^4}{\eta^2 n^2} \E\left(\Ve\left(K_n^{B^\delta\cap U}(\ell) \right)\right)
\]
Using Lemma \ref{mvar} with $\delta=5$ gives what we need. \hfill $\square$

\subsubsection{Proofs of Lemma \ref{othergenerations}, \eqref{smallsum} and \eqref{largesum} \label{Sec4.3.4}}

\begin{proof}[Proof of \eqref{outside}] Let us show that
\begin{align}
\lim_{\epsilon\downarrow0}\limsup_{n\rightarrow\infty}\frac{1}{n}\e\Big[\sum_{\ell\geq (\log n)^2/\varepsilon}\sum_{|z|=\ell}K_n^B(\ell)\Big]&=0,\label{largegenerations}\\
\lim_{\epsilon\downarrow0}\limsup_{n\rightarrow\infty}\frac{1}{n}\e\Big[\sum_{\ell\leq \varepsilon(\log n)^2}\sum_{|z|=\ell}K_n^B(\ell)\Big]&=0.\label{smallgenerations}
\end{align}

$\bullet$  \textit{Proof of \eqref{largegenerations} and \eqref{largesum}.} Recall that $a_z\leq e^{-\MV(z)}$. One sees that
\begin{align}
LHS_\eqref{largegenerations}:=&\e\Big[\sum_{\ell\geq (\log n)^2/\varepsilon}\sum_{|z|=\ell}K_n^B(\ell)\Big]\leq\E\Big[\sum_{\ell\geq (\log n)^2/\varepsilon}\sum_{|z|=\ell}(na_z\wedge 1)\un_{z\in B}\Big]\nonumber\\
\leq &\sum_{\ell\geq (\log n)^2/\varepsilon}n\E\left[\sum_{|z|=\ell}e^{-\overline{V}(z)}\un_{ \{ \overline{V}(z)\geq \log n, z\in B\} }\right]+\sum_{\ell\geq (\log n)^2/\varepsilon}\E\left[\sum_{|z|=\ell}\un_{\{ \overline{V}(z)\leq \log n, \underline{V}(z)\geq-\alpha\} }\right]\nonumber\\
=: &R_I+R_{II}\label{RI+RII}.
\end{align}
Applying \eqref{M2O} to $R_{II}$ yields that
\begin{align*}
R_{II}=&\sum_{\ell\geq (\log n)^2/\varepsilon}\E\left[e^{S_\ell}; \MS_\ell\leq \log n, \mS_\ell\geq-\alpha\right],
\end{align*}
which by \eqref{eSMSmSbd} is bounded by
\[
\sum_{\ell\geq (\log n)^2/\varepsilon} \frac{cn(1+\alpha)(1+\log n+\alpha)}{\ell^{3/2}}\leq c\frac{n(1+\alpha)(1+\log n+\alpha)}{\sqrt{(\log n)^2/\varepsilon}}.
\]
For any $\alpha>0$ fixed, letting $\varepsilon\downarrow 0$ implies that $\lim_{\varepsilon\downarrow 0}\limsup_{n\rightarrow\infty}{{n^{-1}} R_{II}}=0$.
Also by \eqref{M2O}, $R_{I}= \sum_{\ell\geq (\log n)^2/\varepsilon}n\E\left[\sum_{|z|=\ell}e^{-\overline{V}(z)}\un_{\{ \overline{V}(z)\geq \log n, z<\mathcal{L}_{n},\underline{V}(z)\geq-\alpha \}}\right]$ equals to
\[
\sum_{\ell\geq (\log n)^2/\varepsilon}n\E\left[e^{S_\ell-\MS_\ell}; {\MS_\ell\geq \log n, \max_{1\leq k\leq \ell}\sum_{i=1}^ke^{S_i-S_k}\leq n,\mS_\ell\geq-\alpha}\right].
\]
Observe that $e^{\MS_k-S_k}\leq \sum_{i=1}^ke^{S_i-S_k}$. It then follows that
\[
R_I\leq\sum_{\ell\geq (\log n)^2/\varepsilon}n\E\left[e^{S_\ell-\MS_\ell}; \max_{1\leq k\leq \ell}(\MS_k-S_k)\leq \log n,\mS_\ell\geq-\alpha\right],
\]
which by \eqref{eSMSMMSmSbd} is less than
\[
cn(1+\alpha)\sum_{\ell\geq (\log n)^2/\varepsilon}\Big( \frac{1}{\ell^{7/6}}+\frac{1}{\ell}e^{-c' \frac{\ell}{(\log n)^2}}\Big).
\]
Clearly, $\sum_{\ell\geq (\log n)^2/\varepsilon} \frac{1}{\ell^{7/6}}\leq c\varepsilon^{1/6}(\log n)^{-1/3}$, so by monotonicity and change of variables, one sees that
\begin{align*}
\sum_{\ell\geq (\log n)^2/\varepsilon} \frac{1}{\ell}e^{-c' \frac{\ell}{(\log n)^2}}\leq &\int_{(\log n)^2/\varepsilon}^\infty e^{-c' \frac{t}{(\log n)^2}}\frac{dt}{t} = \int_{1/\varepsilon}^\infty e^{-c' s}\frac{ds}{s}\leq \varepsilon/c'.
\end{align*}
Consequently, $R_I\leq (1+\alpha)\varepsilon n/c'+c(1+\alpha)\varepsilon^{1/6}(\log n)^{-1/3}n$. We hence deduce that \\ $
\lim_{\varepsilon\rightarrow 0}\limsup_{n\rightarrow \infty}R_I/n=0$.
Collecting estimates for $R_I$ and $R_{II}$ together with \eqref{RI+RII}, \eqref{largegenerations} follows immediately.

Moreover, observe that $W^*_m(\beta)\leq \sum_{|z|=m}e^{-\overline{V}(z)}\un_{\{ \max_{\phi<y\leq z(\overline{V}(y)-V(y))\leq \beta, \overline{V}(z)\geq \beta} \} }$. So, for any $\varepsilon,\ \eta>0$,
\[
\P\left(\sum_{m\geq \beta^2/\varepsilon}W_m^*(\beta)\geq \eta\right)\leq \P(\inf_{z\in\T}V(z)\leq-\alpha)+\P\left(\sum_{m\geq \beta^2/\varepsilon}\sum_{|z|=m}e^{-\overline{V}(z)}\un_{ \{ \max_{\phi<y\leq z(\overline{V}(y)-V(y))\leq \beta, \overline{V}(z)\geq \beta,\underline{V}(z)\geq-\alpha}\} }\geq \eta\right)
\]
where the second probability on the right hand side vanishes as $\beta\rightarrow\infty$ then $\varepsilon\rightarrow0$ because of the convergence of $R_I/n$ by replacing $\log n$ by $\beta$. The first probability on the right hand side is negligible in view of \eqref{below}.

$ \bullet$ \textit{Proof of \eqref{smallgenerations} and \eqref{smallsum}}. Similarly as above,
\begin{align*}
LHS_\eqref{smallgenerations}:=\e\Big[\sum_{\ell\leq \varepsilon(\log n)^2}\sum_{|z|=\ell}K_n^B(\ell)\Big]
\leq  \E\Big[\sum_{\ell\leq \varepsilon(\log n)^2}\sum_{|z|=\ell} (ne^{-\overline{V}(z)}\wedge 1)\un_{ \{ z<\mathcal{L}_{n}, \underline{V}(z)\geq-\alpha \} }\Big]\leq R'_I+R'_{II},
\end{align*}
where
\begin{equation*}
R'_I:= \sum_{\ell\leq \varepsilon(\log n)^2}\E\Big[\sum_{|z|=\ell} ne^{-\overline{V}(z)}\un_{ \{ \overline{V}(z)\geq \log n/2,  \underline{V}(z)\geq-\alpha \} }\Big],\ R'_{II}:=\sum_{\ell\leq \varepsilon(\log n)^2}\E\Big[\sum_{|z|=\ell}\un_{ \{ \overline{V}(z)\leq \log n/2, \underline{V}(z)\geq-\alpha \} }\Big].
\end{equation*}
Again by \eqref{M2O},
\begin{align*}
R'_I=& n \sum_{\ell\leq \varepsilon(\log n)^2}\E\left[e^{S_\ell-\MS_\ell}; \MS_\ell\geq \log n/2, \mS_\ell\geq-\alpha\right]
\end{align*}
which by \eqref{eSMSMSmSbd} is bounded by
\[
n \sum_{\ell\leq \varepsilon(\log n)^2}\frac{c(1+\alpha)}{\ell^{1/2}\log n}\leq c'(1+\alpha)\sqrt{\varepsilon} n.
\]
Therefore, $\lim_{\varepsilon\downarrow0}\limsup_{n\rightarrow \infty} {R'_I}{n^{-1}}=0$ .
It remains to bound $R'_{II}=\sum_{\ell\leq \varepsilon(\log n)^2}\E\left[\sum_{|z|=\ell}\un_{\{ \overline{V}(z)\leq \log n/2, \underline{V}(z)\geq-\alpha\} }\right]$. By \eqref{M2O},
\begin{align*}
R'_{II}=&\sum_{\ell\leq \varepsilon(\log n)^2}\E\left[ e^{S_\ell}; \MS_\ell\leq \log n/2, \mS_\ell\geq-\alpha\right] 
\leq  \sum_{\ell\leq \varepsilon(\log n)^2}e^{\log n/2}\leq \varepsilon (\log n)^2\sqrt{n},
\end{align*}
 so $\lim_{\varepsilon\downarrow0}\limsup_{n\rightarrow \infty} {R'_{II}}{n^{-1}}=0$.
This completes the proof of \eqref{smallgenerations}.

Similarly to the proof of \eqref{largesum}, the convergence \eqref{smallsum} follows from \eqref{below} and the convergence of $R_I'/n$.
\end{proof}

\begin{proof}[Proof of \eqref{inside}]
We now prove that for any $\epsilon>0$,
\begin{align}
&\sum_{m=\varepsilon(\log n)^2}^{(\log n)^2/\varepsilon}\e\left[K_n^{B\setminus U}(m)\right]=o(n),\label{BsansC}\\
&\sum_{m=\varepsilon(\log n)^2}^{(\log n)^2/\varepsilon}\e\left[K_n^{(B\cap U)\setminus (B^\delta\cap U)}(m)\right]=o(n)\label{Csans}.
\end{align}

As shown in the proof of Lemma \ref{addU}, for any $m\geq \epsilon(\log n)^2$, there exists some constant $c(\varepsilon)\in\r_+$ such that
\[
\e\left[K_n^{B\setminus U}(m)\right]\leq c(\epsilon)n \frac{\log\log n}{(\log n)^3},
\]
so \eqref{BsansC} follows. It remains to show \eqref{Csans}.  Observe that
\begin{align*}
\e\left[K_n^{(B\cap U)\setminus (B^\delta\cap U)}(m)\right]\leq & \E\left[\widetilde{\mathcal{K}}_n^{(B\cap U)\setminus (B^\delta\cap U)}(m)\right]=\E\left[\sum_{|z|=m}na_z\un_{\{ z\in (B\setminus B^\delta)\cap U \}}\right].
\end{align*}
Take $\beta=\log n+\log\log n$. Because of \eqref{KlessF}, for any $\epsilon>0$ fixed, there exists $c_1>0$ such that when $n\geq 10$, for any $m\in [\epsilon\beta^2/2, \beta^2/\epsilon]\cap \z$, 
\begin{align*}
\E\left[\widetilde{\mathcal{K}}_n^{(B\cap U)\setminus (B^\delta\cap U)}(m)\right]\leq \frac{n}{\sqrt{m}}\Ren(\alpha)\left( \E_{\Q^{(\alpha)}}\Big[\frac{W_m^{(\alpha)}(F_{\beta,\beta})}{D_m^{(\alpha)}}\Big]- \E_{\Q^{(\alpha)}}\Big[\frac{W_m^{(\alpha)}(F_{\beta-c_1\log\beta, \beta})}{D_m^{(\alpha)}}\Big]\right).
\end{align*}
It follows immediately that
\begin{multline}\label{sumEW}
\sum_{m=\varepsilon(\log n)^2}^{(\log n)^2/\varepsilon}\e\left[K_n^{(B\cap U)\setminus (B^\delta\cap U)}(m)\right]\\
\leq n\Ren(\alpha)
\times\left(\sum_{m=\epsilon\beta^2/2}^{\beta^2/\epsilon} \frac{1}{\sqrt{m}} \E_{\Q^{(\alpha)}}\Big[\frac{W_m^{(\alpha)}(F_{\beta,\beta})}{D_m^{(\alpha)}}\Big]-
\sum_{m=\epsilon\beta^2/2}^{\beta^2/\epsilon} \E_{\Q^{(\alpha)}}\Big[\frac{W_m^{(\alpha)}(F_{\beta-c_1\log\beta, \beta})}{D_m^{(\alpha)}}\Big]\right).
\end{multline}
Similarly to \eqref{unifcvg}, the convergence \eqref{F} holds uniformly. So following the arguments used to prove Corollary \ref{sumW}, we deduce that for any $\varepsilon>0$, as $\beta\rightarrow\infty$,
\[
\sum_{\varepsilon\beta^2/2\leq m\leq \beta^2/\varepsilon}\frac{1}{\sqrt{m}} \E_{\Q^{(\alpha)}}\Big[\frac{W_m^{(\alpha)}(F_{\beta,\beta})}{D_m^{(\alpha)}}\Big]=\int_{\varepsilon/2}^{1/\varepsilon}\Cb_{\gamma^{-1/2},\gamma^{-1/2}}\frac{d\gamma}{\gamma}+o_\beta(1).
\]
Similarly, we also have
\[
\sum_{\varepsilon\beta^2/2\leq m\leq \beta^2/\varepsilon}\E_{\Q^{(\alpha)}}\frac{1}{\sqrt{m}}\Big[\frac{W_m^{(\alpha)}(F_{\beta-c_1\log\beta,\beta})}{D_m^{(\alpha)}}\Big]=\int_{\varepsilon/2}^{1/\varepsilon} \Cb_{\gamma^{-1/2},\gamma^{-1/2}}\frac{d\gamma}{\gamma}+o_\beta(1).
\]
As a consequence, \eqref{sumEW} becomes
\[
\sum_{m=\varepsilon(\log n)^2}^{(\log n)^2/\varepsilon}\e\left[K_n^{(B\cap U)\setminus (B^\delta\cap U)}(m)\right]\leq o_n(1)n\Ren(\alpha),
\]
which ends the proof.
\end{proof}
\subsection{Proof of Proposition \ref{Prop1.6}}

Most of the arguments are already present in the proof  of Theorem \ref{upton2} in section \ref{sec2.2}. Indeed we have stressed on the fact that the main contribution of visited sites comes from the set of individuals of the tree truncated by $B^\delta\cap U$. \\
\noindent {Similarly to the proof of \eqref{inside}, the restriction $A_3:=\{z\in\T: \overline{V}(z)>\max_{ |y|\leq |z|-|z|^{1/3},\  y\leq z}V(y) \}$ follows easily from \eqref{h}. So it remains to consider $D:= \left\{z\in\T:\ \max_{\phi<y\leq z}(\overline{V}(y)-V(y)) \leq \frac{\log n}{a_0}  \right\},$ and $F:=\left\{z\in\T:\  \overline{V}(z) \geq a_1 \log n \sqrt{ \log \log n}   \right\}$. We only need to show that }
\begin{align}
\lim_{a_0 \rightarrow +\infty}\lim_{n \rightarrow+ \infty}  \e\Big[n^{-1} \sum_{m=\varepsilon(\log n)^2}^{(\log n)^2/\varepsilon}K_n^{B\cap D\cap A_3}(m)\Big]=0, \label{4.21}
\end{align}
and 
\begin{align}
\lim_{n \rightarrow+ \infty}\e\Big[n^{-1} \sum_{m=\varepsilon(\log n)^2}^{(\log n)^2/\varepsilon} K_n^{B\cap F}(m)\Big]=0.  \label{4.22}
\end{align}
For \eqref{4.21} we do as usual and get that the expectation is smaller than
\begin{align*}
& \sum_{m=\varepsilon(\log n)^2}^{(\log n)^2/\varepsilon} \E\left[e^{S_m- \overline{S}_m} ; { \max_{1\leq i\leq m}(\MS_i-S_i)\leq  \log n/a_0,\ \underline{S}_m \geq - \alpha, \Upsilon_S>m-m^{1/3}} \right] \\
&  \leq \sum_{m=\varepsilon(\log n)^2}^{(\log n)^2/\varepsilon} \sum_{j=m-m^{1/3}}^m \E\left[e^{S_{m-j}} \un_{ \overline{S}_{m-j} \leq 0} \right]\P\left[\max_{1\leq i\leq j}(\MS_i-S_i)\leq  \log n/a_0,\ \overline{S}_{j-1}<S_j,\ \underline{S}_j \geq - \alpha \right].
\end{align*}
{Similarly to \eqref{bdRI} and the lines that follow, the above sum is bounded by \\
 $\sum_{m=\varepsilon(\log n)^2}^{(\log n)^2/\varepsilon} \sum_{j=m-m^{1/3}}^m \frac{c(1+\alpha)}{(m-j+1)^{3/2} m}e^{-c'ma_0/(\log n)^2}\leq -2(\log \epsilon) e^{-c'\epsilon a_0}$ which goes to zero as $a_0\rightarrow\infty$.}

\noindent Also for the expectation in \eqref{4.22} we have that it is smaller than
\begin{align*}
&\sum_{m=\varepsilon(\log n)^2}^{(\log n)^2/\varepsilon} \E\left[e^{S_m- \overline{S}_m} ; { \overline{S}_m \geq a_1 \log n \sqrt{ \log \log n},\ \underline{S}_m \geq - \alpha} \right] \\
& \leq \sum_{m=\varepsilon(\log n)^2}^{(\log n)^2/\varepsilon} \sum_{j=1}^m \E\left[e^{S_{m-j}} ; { \overline{S}_{m-j} \leq 0} \right]\P\left[\underline{S}_j \geq - \alpha,\  \overline{S}_{j} \geq a_1 \log n \sqrt{ \log \log n} \right]
\end{align*}
which by \eqref{eSMSbd} and \eqref{Last1} is bounded by $c   \sum_{m=\varepsilon(\log n)^2}^{(\log n)^2/\varepsilon} \sum_{j=1}^m (m-j+1)^{-3/2} m^{- c' *a_1^2} =o_n(1) $ by choosing $a_1$ properly. \hfill $\square$


 \appendix
  \renewcommand{\appendixname}{Appendix~\Alph{section}}

\section{Appendix}

\subsection{Finiteness of $\Lambda$ [see \eqref{Lambda}] \label{A1}}

\begin{lem}\label{A}
The function $\lambda: (0,\infty)\rightarrow(0,\infty)$ is well defined and integrable, i.e.,
\begin{equation}
\Lambda=\int_0^\infty \lambda(\gamma)d\gamma=\cb_0\int_0^\infty \frac{\Cb_{\gamma^{-1/2},\gamma^{-1/2}}}{\gamma} d\gamma<\infty.
\end{equation}
Further, for any $a,b>0$, $\Cb_{a\gamma^{-1/2},b\gamma^{-1/2}}/\gamma$ is integrable.
\end{lem}

\begin{proof} Recalling \eqref{Cab}, it suffices to show that $\mathcal{C}_{a,b}\in(0,\infty)$ and that ${\mathcal{C}_{\gamma^{-1/2},\gamma^{-1/2}}}/{\gamma}$ is integrable. Recall that for any $a, b>0$, 
\[
\mathcal{C}_{a,b}=2\cb_1^+\cb_2^+\E\left(\Psi^{a,b}(\sigma \mB_1, \sigma( \overline{\mB}_1- \mB_1)); \max_{0\leq s\leq 1}\sigma(\overline{\mB}_s-\mB_s)\leq \sqrt{2} a\right)
\]
with
\[
\Psi^{a,b}(x,h):= 
\cb_2^+\P\left(\sigma \mB_1 > (\sqrt{2}b-x)\vee h, \sigma(\overline{\mB}_1- \mB_1)\leq (\sqrt{2} a-h)_+\wedge x, \max_{0\leq s\leq 1}\sigma(\mB_s-\underline{\mB}_{[s,1]})\leq \sqrt{2}a\right).
\]
Obviously, $\mathcal{C}_{a,b}\leq 2\cb_1^+\cb_2^+\cb_2^+<\infty$. Moreover,
\begin{multline*}
\mathcal{C}_{a,b}\geq 2\cb_1^+\cb_2^+\cb_2^+\P\left(\sigma \mB_1>\sqrt{2} b, \sigma(\overline{\mB}_1-\mB_1)\leq \frac{\sqrt{2}}{2}a, \max_{0\leq s\leq 1}\sigma(\mB_s-\underline{\mB}_{[s,1]})\leq \sqrt{2}a\right)\\
\times \P\left(\sigma \mB_1\geq \frac{\sqrt{2}}{2}a, \max_{0\leq s\leq 1}\sigma(\overline{\mB}_s-\mB_s)\leq  \frac{\sqrt{2}}{2}(a\wedge b)\right).
\end{multline*}
On the one hand, Biane and Yor \cite{BianeYor} showed that
\[
\left((\mB_s, \underline{\mB}_{[s,1]}); 0\leq s\leq 1\right)=_d \left((|b_s|+\lambda_s^0, \lambda_s^0); 0\leq s\leq 1\right),
\]
where $(b_s,s\in[0,1])$ is a Brownian bridge and $(\lambda_s^0, 0\leq s\leq 1)$ its local time at $0$. On the other hand, if $(R_s; 0\leq s\leq 1)$ a Bessel(3) process,  Imhof \cite{Imhof} showed that for any $x>0$,
\[
(\mB_s, s\in[0,1])\textrm{ given } \{\mB_1=x\} =_d (R_s, s\in[0,1])\textrm{ given }\{ R_1=x\},
\]
where $\P(\mB_1\in dx)=xe^{-x^2/2}\un_{ \{ x\geq0 \}}dx$. By the continuity of the distribution of $(b, \lambda^0)$ and $R$, one sees that $\mathcal{C}_{a,b}$ is continuous and strictly positive.

 Let $(\mB_s, s\in[0,1])$ and $(\widetilde{\mB}_s, s\in[0,1])$ be two independent Brownian meanders. Then
\[
\mathcal{C}_{a,b}\leq \begin{cases}
c\P\left(\sigma(\mB_1+\widetilde{\mB}_1)\geq \sqrt{2}b\right)\\
c\P\left(\max_{s\in[0,1]}\sigma(\mB_s-\underline{\mB}_{[s,1]})\leq \sqrt{2}a\right).
\end{cases}
\]
It follows from the first inequality that
\[
\mathcal{C}_{a,b}\leq c\P(\mB_1\geq \frac{\sqrt{2}b}{2\sigma})=c e^{-b^2/4\sigma^2},
\]
On the other hand, according to \cite{BianeYor},
\[
\P\left(\max_{s\in[0,1]}\sigma(\mB_s-\underline{\mB}_{[s,1]})\leq \sqrt{2}a\right)=\P\left(\max_{s\in[0,1]}|b_s|\leq \sqrt{2}a/\sigma\right)\leq \P\left(\max_{s\in[0,1]}b_s\leq \sqrt{2}a/\sigma\right).
\]
This shows that
\begin{align*}
\mathcal{C}_{a,b}\leq& c\P\left(\max_{s\in[0,1]}b_s\leq \sqrt{2}a/\sigma\right)
=c \left(1-\exp\left(-2\Big(\frac{\sqrt{2}a}{\sigma}\Big)^2\right)\right)\leq \frac{4c}{\sigma^2}a^2.
\end{align*}
We are now ready to prove the integrability. Observe that
\begin{align*}
\int_0^\infty\frac{\mathcal{C}_{\gamma^{-1/2},\gamma^{-1/2}}}{\gamma}d\gamma =& \int_0^1\frac{\mathcal{C}_{\gamma^{-1/2},\gamma^{-1/2}}}{\gamma}d\gamma+\int_1^\infty \frac{\mathcal{C}_{\gamma^{-1/2},\gamma^{-1/2}}}{\gamma}d\gamma\\
\leq &\int_0^1 c e^{-\frac{1}{4\sigma^2\gamma}}\frac{d\gamma}{\gamma}+\int_1^\infty  \frac{4c}{\sigma^2}\frac{d\gamma}{\gamma^2}=\int_0^1 c e^{-\frac{1}{4\sigma^2\gamma}}\frac{d\gamma}{\gamma}+ \frac{4c}{\sigma^2}.
\end{align*}
By change of variables $t=1/\gamma$,
\begin{align*}
\int_0^1 c e^{-\frac{1}{4\sigma^2\gamma}}\frac{d\gamma}{\gamma}=&\int_1^\infty c e^{-\frac{t}{4\sigma^2}}\frac{dt}{t}<\infty.
\end{align*}
We hence conclude the integrability of $\frac{\mathcal{C}_{\gamma^{-1/2},\gamma^{-1/2}}}{\gamma}$, as well as $\frac{\mathcal{C}_{a\gamma^{-1/2},b\gamma^{-1/2}}}{\gamma}$  for any $a,b>0$.
\end{proof}
\subsection{Results on one-dimensional random walks}

In this section we state technical inequalities that are used all along the paper. The sequence $(S_k,k)$ which appears here is the one defined in \eqref{M2O}. The proofs are postpone Section \ref{A3}.

We start with two well know inequalities  (see \cite{AidekonShi09} for instance) and some basic Facts. There exists constant $c>0$ such that for any $n\geq1$ and $u\geq0$
\begin{equation}\label{mSbd}
\P(\mS_n\geq-u)\leq \frac{c(1+u)}{\sqrt{n}}\textrm{ and } \P(\MS_n\leq u)\leq \frac{c(1+u)}{\sqrt{n}}.
\end{equation}
By Lemma 2.2 in \cite{AidekonShi09}, there exists some constant $c>0$ such that for any $u\geq0$, $b\geq a\geq -u$ and any $n\geq1$,
\begin{equation}\label{mSSbd}
\P(\mS_n\geq-u, a\leq S_n\leq b)\leq \frac{c(1+u)(1+b+u)(1+b-a)}{n^{3/2}}.
\end{equation}

 
\begin{fact}
\begin{enumerate}[fullwidth]
\item For any $u, \alpha\geq0$ and $\forall n\geq1$,
\begin{equation}\label{mSMSbd}
\P_u(\mS_n\geq-\alpha, S_n=\MS_n)\leq \frac{c(1+\alpha+u)}{n}.
\end{equation}
\item 
\begin{enumerate}[fullwidth]
\item For any $n\geq1$, $B>0$ fixed, there exists $c(B)>0$ such that for any $u\geq0$, $-B\sqrt{n}\leq -\alpha\leq0<a<b\leq B\sqrt{n}$,
\begin{equation}\label{mSMSSbd}
\P_u(\mS_n\geq-\alpha, S_n=\MS_n\in[a,b])\leq \frac{c(B)(1+\alpha+u)(b-a)}{n^{3/2}}.
\end{equation}
\item For any $n\geq 1$, $A>0$,
\begin{equation}\label{mSMSSbd0}
\P(\mS_n\geq-\alpha, S_n=\MS_n\geq A)\leq \frac{c(1+\alpha)}{A n^{1/2}}.
\end{equation}
\end{enumerate}
\item For any $a, A,\alpha>0$ and $\forall n>m\geq1$,
\begin{equation}\label{mSMSSbd+}
\P(\mS_n\geq-\alpha, S_n=\MS_n, \MS_{m}-S_n\geq-A, \MS_m-S_m\leq a)\leq \frac{c (1+A)(1+a+A)(1+\alpha)}{m^{1/2}(n-m)^{3/2}}
\end{equation}
\end{enumerate}
\end{fact}

We now state the following Lemma which is mostly a consequence of the above facts.

\begin{lem}
For any $\alpha\geq0$, $0< a\leq b$ and $n\geq1$, we have 
\begin{align}
\P\Big(S_n=\mS_n\geq-\alpha\Big)&\leq \frac{c(1+\alpha)^2}{n^{3/2}},\label{SmSbd}\\
\P\Big(\sum_{1\leq i\leq n}e^{S_i-S_n}\in[a,b],\underline{S}_n\geq-\alpha\Big)&\leq c\frac{(1+\alpha)(1+\log b)(1+\log b-\log a+\log n)}{n},\label{mSeSsum}\\
\E\Big(e^{S_n}; S_n\in[a,b], \underline{S}_n\geq -\alpha\Big)&\leq \frac{ce^b(b+\alpha+1)(1+b-a)(1+\alpha)}{n^{3/2}}.\label{mSeS}
\end{align}
\end{lem}

Following  Lemma focus on asymptotic results.
\begin{lem}
Let $a,b\geq 0$ fixed and $\lim_n\frac{a_n}{\sqrt{n}}=\lim_n\frac{b_n}{\sqrt{n}}=0$. We have the following convergences.
\begin{enumerate}[fullwidth]
\item Moreover, for any $\alpha\geq0$ fixed, 
\begin{equation}\label{mSMScvg}
\lim_n n\P\left(\mS_n\geq-\alpha, S_n>\MS_{n-1}, \max_{1\leq i\leq n}(\MS_i-S_i)\leq a\sqrt{n}+a_n, S_n\geq b\sqrt{n}+b_n\right)=\mathcal{C}_{a,b}\Ren(\alpha),
\end{equation}
where $\mathcal{C}_{a,b}\in (0,\infty)$ is a constant depending on $a$ and $b$, defined in \eqref{cab}.
\item Let $g:[1,\infty)\rightarrow\r_+$ be a uniformly continuous and bounded function. Then,
\begin{multline}\label{sumcvg}
\lim_{n\rightarrow\infty}n\E\left[g\Big(\sum_{j=1}^n e^{S_j-S_n}\Big); \mS_n\geq-\alpha, S_n>\MS_{n-1}, \max_{1\leq i\leq n}(\MS_i-S_i)\leq a\sqrt{n}+a_n, S_n\geq b\sqrt{n}+b_n\right]\\
=\mathcal{C}_{a,b}\Ren(\alpha)\E[g(\Hinf_\infty)]
\end{multline}
\end{enumerate}

\end{lem}

Below we collect some more basic facts.
\begin{fact}
\begin{enumerate}[fullwidth]
\item For any $n\geq1$,
\begin{equation}\label{eSMSbd}
\E[e^{S_n}; \MS_n\leq 0] \leq \frac{c}{n^{3/2}}.
\end{equation}
\item For any $A>0$ and $n\geq1$, 
\begin{equation}\label{eSMSbd+}
\E[e^{S_n}; \MS_n\leq A]\leq \frac{c (1+A)e^A}{n^{3/2}}.
\end{equation}
\item For any $n\geq1$,
\begin{equation}\label{eSbd}
\E\Big[\frac{e^{S_n}}{\sum_{1\leq i\leq n}e^{S_n}}\Big]\leq \frac{c}{n^{1/2}}.
\end{equation}
\item For any $A, \delta>0$ and $\forall n\geq k_2>k_1\geq1$,
\begin{equation}\label{mSmSbd}
\P(\mS_{[k_1,k_2]}\leq A, \mS_{(1+\delta)n}\geq0)\leq \frac{c(1+A)}{\sqrt{\delta n k_1}}.
\end{equation}
\item If $\E[e^{\pm \theta S_1}]<\infty$ for some $\theta>0$ [see \eqref{hyp0+}], then for any $\delta>0$ there exists $c(\delta,\theta)>0$ such that for any $n\geq1$,
\begin{equation}\label{LDP}
\P(\MS_n \geq n^{1+\delta})\leq c(\delta,\theta)e^{-\theta n^{1+\delta}/2}. 
\end{equation}
\item Let $a>0$. With the same hypothesis as above there exists $c_2>0$ such that 
\begin{equation}
 \P(\MS_n \geq a\sqrt { n \log n}, \mS_n \geq - \alpha )\leq \frac{c}{ n^{a^2  c_2}}. \label{Last1}
\end{equation}
\end{enumerate}

\end{fact}
The following corollaries follow from above lemmas.
\begin{cor}
Let $a\geq 0$, $b_n\geq0$ and $\lim_n\frac{a_n}{\sqrt{n}}=\lim_n\frac{b_n}{\sqrt{n}}=0$. Take $n_0=n-n^{1/3}$, then
\begin{align}
\P\Big( \mS_{[n/2, n]}\leq b_n,\mS_n\geq-\alpha, S_n>\MS_{n-1}\Big)&=\frac{o_n(1)}{n}.\label{omSmScvg}
\end{align}
\end{cor}

\begin{lem}
For any $\alpha>0$ fixed and $k_n=o(n^{1/2})$, the following estimates hold uniformly for $n/2\leq k\leq n$, 
 \begin{align}\label{mSMSkbd}
\P\Big( S_{k_n}\in[k_n^{1/3},k_n],\mS_{[k_n,k]}\leq k_n^{1/6}, S_k>\overline{S}_{k-1}, \underline{S}_k\geq-\alpha\Big)=&\frac{o_n(1)}{n},\\
\P\Big(S_{k_n}\notin[k_n^{1/3},k_n], \mS_k\geq-\alpha, S_k=\MS_k\Big)\leq & \frac{c}{nk_n^{1/2}}.\label{kmSMSbd}
\end{align}
\end{lem}

\begin{cor}
\begin{enumerate}[fullwidth]
\item If $\alpha\in\r_+$ is fixed, then for any $\ell\geq1$ and $A\in\r_+$, 
\begin{equation}\label{eSMSmSbd}
\E\left[e^{S_\ell}; \MS_\ell\leq A, \mS_\ell\geq-\alpha\right]\leq \frac{ce^A(1+\alpha)(1+A+\alpha)}{\ell^{3/2}}.
\end{equation}
\item For $A>0$ sufficiently large and any $\ell\geq1$,
\begin{equation}\label{eSMSMMSmSbd}
\E\left[e^{S_\ell-\MS_\ell}; \max_{1\leq k\leq \ell}(\MS_k-S_k)\leq A,\mS_\ell\geq-\alpha\right]\leq  \frac{c(1+\alpha)}{\ell^{7/6}}+ \frac{c(1+\alpha)}{\ell}e^{-c' \frac{\ell}{A^2}}.
\end{equation}
\item For any $A\geq 1$ and $\ell\geq1$,
\begin{equation}\label{eSMSMSmSbd}
\E\left[e^{S_\ell-\MS_\ell}; \MS_\ell\geq A, \mS_\ell\geq-\alpha\right]\leq \frac{c(1+\alpha)}{\ell^{1/2}A}.
\end{equation}
\end{enumerate}

\end{cor}

\subsection{Proofs of \eqref{mSMSbd}--\eqref{omSmScvg} \label{A3}}
We show these results one by one.

\begin{proof}[Proof of \eqref{mSMSbd}] Let $R_k:=S_n-S_{n-k}$ for $0\leq k\leq n$. Clearly, $(R_k, 0\leq k\leq n/2)$ is an independent copy of $(S_k, 0\leq k\leq n/2)$. Hence,
\[
\P_u\Big(\mS_n\geq-\alpha, S_n=\MS_n\Big)\leq \P_u\Big(\mS_{n/2}\geq-\alpha, \underline{R}_{n/2}\geq0\Big)=\P\Big(\mS_{n/2}\geq-\alpha-u\Big)\P\Big( \underline{R}_{n/2}\geq0\Big).
\]
Applying \eqref{mSbd} to both $(S_\cdot)$ and $(R_\cdot)$ yields that
\begin{equation}
\P_u\Big(\mS_n\geq-\alpha, S_n=\MS_n\Big)\leq {c(1+\alpha+u)}{n^{-1}},
\end{equation}
which is exactly \eqref{mSMSbd}.
\end{proof}

\begin{proof}[Proof of \eqref{mSMSSbd}-\eqref{mSMSSbd0}] Using the same arguments as above, as $S_n=S_{n/2}+R_{n/2}$, we get that
\begin{align*}
\P_u\Big(\mS_n\geq-\alpha, S_n=\MS_n\in [a,b]\Big)=&\P\Big(\mS_n\geq-\alpha-u,S_n=\MS_n\in [a-u, b-u]\Big)\\
\leq &\P\Big(\mS_{n/2}\geq-\alpha-u, \underline{R}_{n/2}\geq0, S_{n/2}+R_{n/2}\in [a-u,b-u]\Big)\\
=&\E\Big(\psi(S_{n/2});\mS_{n/2}\geq-\alpha-u\Big),
\end{align*}
where $\psi(x):=\P(\underline{R}_{n/2}\geq0, R_{n/2}\in [(a-u-x)_+,b-u-x])\un_{ \{ -\alpha-u\leq x\leq b-u \} }$. By \eqref{mSSbd},
\[
\psi(x)\leq \frac{c(1+b-u-x)(1+b-a)}{n^{3/2}}\un_{\{ -\alpha-u\leq x\leq b-u \}}.
\]
It follows that
\begin{align}\label{mSmRbd}
\P_u\Big(\mS_n\geq-\alpha, S_n=\MS_n\in [a,b]\Big)\leq&\frac{c(1+b-a)}{n^{3/2}}\E\Big((1+b-u-S_{n/2})_+;\mS_{n/2}\geq-\alpha-u\Big)\\
\leq & \frac{c(1+b-a)}{n^{3/2}}(1+b+\alpha)\P\Big(\mS_{n/2}\geq-\alpha-u\Big),\nonumber
\end{align}
which by \eqref{mSbd} is less than $\frac{c(1+b-a)}{n^{3/2}}\frac{(1+\alpha+u)(1+b+\alpha)}{\sqrt{n}}=O(1)\frac{(1+b-a)(1+\alpha+u)}{n^{3/2}}$ since $b\vee\alpha\leq B\sqrt{n}$. This completes the proof of \eqref{mSMSSbd}. 

Similarly for \eqref{mSMSSbd0}, we have
\begin{multline*}
\P\left(\mS_n\geq-\alpha, \MS_n=S_n\geq A\right)\leq \P\left(\mS_{n/2}\geq-\alpha, \underline{R}_{n/2}\geq0, R_{n/2}+S_{n/2}\geq A\right)\\
\leq  \P\left(\mS_{n/2}\geq-\alpha, \underline{R}_{n/2}\geq0, R_{n/2}\geq A/2\right)+\P\left(\mS_{n/2}\geq-\alpha, \underline{R}_{n/2}\geq0, S_{n/2}\geq A/2\right),
\end{multline*}
which by independence between $(S_i, i\leq n/2)$ and $(R_i, i\leq n/2)$ and \eqref{mSbd}, is bounded by
\[
\frac{c(1+\alpha)\P\left(\underline{S}_{n/2}\geq0, S_{n/2}\geq A/2\right)}{n^{1/2}}+\frac{c\P\left(\mS_{n/2}\geq-\alpha, S_{n/2}\geq A/2\right)}{n^{1/2}}.
\]
By Lemma 2.3 in \cite{AidekonShi09}, there exists a constant $c$ such that for any $\alpha\geq0$,
\[
\sup_{n\geq1}\E\left[|S_n|;{\mS_n\geq-\alpha}\right]\leq c(\alpha+1).
\]
It follows from this lemma and Markov's inequality that for any $\alpha\geq0$,
\[
\P\left(\mS_{n/2}\geq-\alpha, S_{n/2}\geq A/2\right)\leq 2\E\left[\frac{|S_{n/2}|}{A}; {\mS_{n/2}\geq-\alpha}\right]\leq \frac{c(1+\alpha)}{A}.
\]
As a consequence, 
\[
\P\left(\mS_n\geq-\alpha, \MS_n=S_n\geq A\right)\leq \frac{c(1+\alpha)}{A \sqrt{n}}.
\] \end{proof}

\begin{proof}[Proof of \eqref{mSMSSbd+}]
To obtain \eqref{mSMSSbd+}, we consider the two independent random walks $(S_k, 0\leq k\leq m)$ and $(R_k, 0\leq k\leq n-m)$. As $S_n=R_{n-m}+S_{m}$, one immediately sees that
\begin{align*}
&\P\Big(\mS_n\geq-\alpha, S_n=\MS_n, \MS_m-S_n\geq-A, \MS_m-S_m\leq a\Big)\\
\leq &\P\Big(\mS_m \geq -\alpha, \MS_m-S_m\leq a, \underline{R}_{n-m}\geq0, R_{n-m}\in [\MS_m-S_m, \MS_m-S_m+A]\Big)\\
\leq & \E\left[ \P\left(\underline{R}_{n-m}\geq0, R_{n-m}\in [\MS_m-S_m, \MS_m-S_m+A]\Big\vert (S_k, 0\leq k\leq m)\right); {\mS_m \geq -\alpha, \MS_m-S_m\leq a}\right].
\end{align*}
Applying \eqref{mSSbd} to this conditional probability implies that
\begin{align*}
\P\Big(\mS_n\geq-\alpha, S_n=\MS_n, \MS_m-S_n\geq-A, \MS_m-S_m\leq a\Big)\leq & \E\left[c\frac{(1+A)(1+A+\MS_m-S_m)}{(n-m)^{3/2}}; {\mS_m \geq -\alpha, \MS_m-S_m\leq a}\right]\\
\leq & c\frac{(1+A)(1+A+a)}{(n-m)^{3/2}}\P\left(\mS_m\geq-\alpha\right),
\end{align*}
which by \eqref{mSbd} is bounded by
\[
c\frac{(1+\alpha)(1+A)(1+A+a)}{m^{1/2}(n-m)^{3/2}}.
\]
This ends the proof of \eqref{mSMSSbd+}.
\end{proof}

\begin{proof}[Proof of \eqref{SmSbd}] Let $T_k:=S_{n-k}-S_n=-R_k$. Then $(T_k, 0\leq k\leq n)$ is a random walk distributed as $(-S_k, 0\leq k\leq n)$. It follows from \eqref{mSSbd} that
\[
\P\Big(S_n=\mS_n\geq-\alpha\Big)\leq\P\Big(\underline{T}_n\geq 0, T_n\leq \alpha\Big)\leq  \frac{c(1+\alpha)^2}{n^{3/2}}.
\]
\end{proof}

\begin{proof}[Proof of \eqref{mSeSsum}] Observe that $e^{\MS_n-S_n}\leq \sum_{1\leq i\leq n}e^{S_i-S_n}\leq ne^{\MS_n-S_n}$, then
\[
\Big\{\sum_{1\leq i\leq n}e^{S_i-S_n}\in[a,b]\Big\}\subset\{\log a-\log n\leq\overline{S}_n-S_n\leq \log b\}.
\]
We thus bound the left hand side of \eqref{mSeSsum} as follows
\begin{align*}
LHS_\eqref{mSeSsum}:=&\P\Big(\sum_{1\leq i\leq n}e^{S_i-S_n}\in[a,b],\mS_n\geq-\alpha\Big)\leq\P(\log a-\log n\leq\overline{S}_n-S_n\leq \log b, \underline{S}_n\geq-\alpha)\\
=&\sum_{k=1}^n \P\left(\MS_{k-1}<S_k=\MS_n, \log a-\log n\leq\overline{S}_n-S_n\leq \log b, \underline{S}_n\geq-\alpha\right).
\end{align*}
By Markov property at the first hitting time $\MS_n$,
\begin{align*}
LHS_\eqref{mSeSsum}\leq\sum_{k=1}^n \P\left(\underline{S}_k\geq-\alpha, S_k=\overline{S}_k\right)\P\left(\overline{S}_{n-k}\leq0, S_{n-k}\in[-\log b,\log n-\log a]\right). 
\end{align*}
By \eqref{mSMSbd} and \eqref{mSSbd}, we deduce that
\begin{align*}
LHS_\eqref{mSeSsum}&\leq \sum_{k=1}^n \frac{c(1+\alpha)}{k}\frac{(1+\log b)(1+\log b-\log a+\log n)}{(n-k+1)^{3/2}}\\
&\leq \frac{c(1+\alpha)(1+\log b)(1+\log b-\log a+\log n)}{n},
\end{align*}
which ends the proof.
\end{proof}

\begin{proof}[Proof of \eqref{mSeS}] By \eqref{mSSbd}, one sees that
\begin{align*}
LHS_\eqref{mSeS}:=\E\left(e^{S_n};\mS_n\geq-\alpha, S_n\in[a,b]\right)\leq& e^b\P\left( \underline{S}_n\geq-\alpha, S_n\in[a,b]\right)\\
\leq& ce^b {(1+\alpha)(1+b-a)(b+\alpha+1)}{n^{-3/2}}.
\end{align*}
\end{proof}

 
\begin{proof}[Proof of \eqref{mSMScvg}] Consider the two independent random walks $(S_k, 0\leq k\leq n)$ and $(R_k, 0\leq k\leq n)$. One observes that
\[
\max_{1\leq i\leq n}(\MS_i-S_i)=\max\left\{\max_{1\leq i\leq n/2}(\MS_i-S_i), \max_{1\leq i\leq n/2}(R_k-\underline{R}_{[k,n/2]}), \MS_{n/2}-S_{n/2}+\overline{R}_{n/2}-R_{n/2}\right\},
\]
and that
\[
\{\mS_n\geq-\alpha, S_n>\MS_{n-1}\}=\{\mS_{n/2}\geq -\alpha\}\cap\{\overline{R}_{n/2}-R_{n/2}\leq S_{n/2}+\alpha\}\cap\{\underline{R}_{n/2}>0\}\cap\{R_{n/2}>\MS_{n/2}-S_{n/2}\}.
\]
Let
\[
\P_\eqref{mSMScvg}:=\P\left(\mS_n\geq-\alpha, S_n>\MS_{n-1}, \max_{1\leq i\leq i}(\MS_i-S_i)\leq a\sqrt{n}+a_n, S_n\geq b\sqrt{n}+b_n\right).
\]
It is immediate that
\begin{multline*}
\P_\eqref{mSMScvg}=\P\Big(\mS_{n/2}\geq -\alpha, \underline{R}_{n/2}>0, \max_{1\leq i\leq n/2}(\MS_i-S_i)\leq a\sqrt{n}+a_n, R_{n/2}+S_{n/2}> (b\sqrt{n}+b_n)\vee \MS_{n/2},\\
 \overline{R}_{n/2}-R_{n/2}\leq (a\sqrt{n}+a_n-\MS_{n/2}+S_{n/2})\wedge  (S_{n/2}+\alpha), \max_{1\leq i\leq n/2}(R_k-\underline{R}_{[k,n/2]})\leq a\sqrt{n}+a_n \Big)
\end{multline*}
which equals to
\[
\E\left(\Psi_n^{a,b}(\frac{S_{n/2}}{\sqrt{n/2}},\frac{\MS_{n/2}-S_{n/2}}{\sqrt{n/2}}); \mS_{n/2}\geq -\alpha, \max_{1\leq i\leq n/2}(\MS_i-S_i)\leq a\sqrt{n}+a_n\right)
\]
where
\begin{multline*}
\Psi^{a,b}_n(x,h):=\P\Big( \underline{R}_{n/2}>0\Big)\times\P\Big(\frac{R_{n/2}}{\sqrt{n/2}}> (\sqrt{2}b+\frac{b_n}{\sqrt{n/2}})\vee (x+h)-x, \\
\frac{\overline{R}_{n/2}-R_{n/2}}{\sqrt{n/2}}\leq (\sqrt{2} a-h+\frac{a_n}{\sqrt{n/2}})\wedge (x+\frac{\alpha}{\sqrt{n/2}}),  \max_{1\leq i\leq n/2}(R_k-\underline{R}_{[k,n/2]})\leq a\sqrt{n}+a_n\Big\vert \underline{R}_{n/2}>0\Big).
\end{multline*}
Again by invariance principle and \eqref{c+}, as $n\rightarrow\infty$, 
\begin{multline}
\sqrt{n/2}\Psi_n^{a,b}(x,h)\longrightarrow \Psi^{a,b}(x,h)= \\
\cb_2^+\P\left(\sigma m_1> (\sqrt{2}b-x)\vee h, \sigma\overline{ m}_1-\sigma m_1\leq (\sqrt{2} a-h)\wedge x, \max_{0\leq s\leq 1}\sigma(m_s-\underline{m}_{[s,1]})\leq \sqrt{2}a\right).
\end{multline}
Because $\Psi_n^{a,b}(x,h)$ is monotone for $x\geq0$ and $h\geq0$ and $\Psi^{a,b}$ is continuous, by Dini's theorem, we have uniformly for $(x,h)\in\r_+^2$,
\[
\Psi^{a,b}_n(x,h)=\frac{\Psi^{a,b}(x,h)+o_n(1)}{\sqrt{n/2}}.
\]
As a consequence, 
\begin{align*}
&\P_\eqref{mSMScvg}=\E\left(\frac{\Psi^{a,b}(\frac{S_{n/2}}{\sqrt{n/2}},\frac{\MS_{n/2}-S_{n/2}}{\sqrt{n/2}})+o_n(1)}{\sqrt{n/2}}; \mS_{n/2}\geq -\alpha, \max_{1\leq i\leq n/2}(\MS_i-S_i)\leq a\sqrt{n}+a_n\right)\\
=& \frac{1}{\sqrt{n/2}}\E\left(\Psi^{a,b}(\frac{S_{n/2}}{\sqrt{n/2}},\frac{\MS_{n/2}-S_{n/2}}{\sqrt{n/2}})+o_n(1);  \max_{1\leq i\leq n/2}(\MS_i-S_i)\leq a\sqrt{n}+a_n\Big\vert \mS_{n/2}\geq-\alpha\right)\P\left(\mS_{n/2}\geq-\alpha\right).
\end{align*}
Once again by invariance principle and the fact that $\lim_{n\rightarrow\infty}\sqrt{n}\P(\mS_n\geq-\alpha)=\Ren(\alpha)\cb_1^+$, 
\[
\P_\eqref{mSMScvg}=\frac{\Ren(\alpha) \mathcal{C}_{a,b}}{n}+\frac{o_n(1)}{n}, 
\]
with $\mathcal{C}_{a,b}$ defined in \eqref{cab}.
\end{proof}

\begin{proof}[Proof of \eqref{sumcvg}] We turn to consider 
\[
n\E\left[g\Big(\sum_{j=1}^n e^{S_j-S_n}\Big); \mS_n\geq-\alpha, S_n>\MS_{n-1}, \max_{1\leq i\leq n}(\MS_i-S_i)\leq a\sqrt{n}+a_n, S_n\geq b\sqrt{n}+b_n\right].
\]
First, we show that in this case with high probability, $\MS_{n/2}\leq S_n-n^{-1/3}$. In fact, 
\begin{align*}
&\E\left[g\Big(\sum_{j=1}^n e^{S_j-S_n}\Big); \mS_n\geq-\alpha, \MS_{n/2}\geq S_n-n^{-1/3}, S_n>\MS_{n-1}, \max_{1\leq i\leq n}(\MS_i-S_i)\leq a\sqrt{n}+a_n, S_n\geq b\sqrt{n}+b_n\right]\\
&\leq ||g||_\infty \P\left( \mS_n\geq-\alpha, \MS_{n/2}\geq S_n-n^{-1/3}, S_n=\MS_{n}, S_{n/2}-\MS_{n/2}\leq a\sqrt{n}+a_n\right)\\
&\leq c||g||_\infty \frac{a(1+\alpha)}{n^{1+1/6}},
\end{align*}
where the last inequality follows from \eqref{mSMSSbd+}. Now, given $\MS_{n/2}\leq S_n-n^{-1/3}$, $\sum_{j=1}^n e^{S_j-S_n}$ can be replaced by $\sum_{n/2\leq j\leq n}e^{S_j-S_n}$ which is independent of $(S_k; 0\leq k\leq n/2)$. Note that on $\{\MS_{n/2}\leq S_n-n^{-1/3}\}$,
\[
\sum_{n/2\leq j\leq n}e^{S_j-S_n}\leq \sum_{j=1}^n e^{S_j-S_n}\leq ne^{-n^{1/3}}+\sum_{n/2\leq j\leq n}e^{S_j-S_n}.
\]
and that $g$ is uniformly continuous. Hence,
\[
\Big\vert g\Big(\sum_{n/2\leq j\leq n}e^{S_j-S_n}\Big)-g\Big(\sum_{j=1}^n e^{S_j-S_n}\Big)\Big\vert =o_n(1).
\]
Therefore, we deduce that
\begin{align*}
&\E\left[g\Big(\sum_{j=1}^n e^{S_j-S_n}\Big); \mS_n\geq-\alpha, S_n>\MS_{n-1}, \max_{1\leq i\leq n}(\MS_i-S_i)\leq a\sqrt{n}+a_n, S_n\geq b\sqrt{n}+b_n\right]\\
=&\E\left[g\Big(\sum_{n/2\leq j\leq n} e^{S_j-S_n}\Big); \mS_n\geq-\alpha, S_n>\MS_{n-1}, \max_{1\leq i\leq n}(\MS_i-S_i)\leq a\sqrt{n}+a_n, S_n\geq b\sqrt{n}+b_n\right]+\frac{o_n(1)}{n}
\end{align*}
Now we use $(R_{k}, 0\leq k\leq n/2)$ in replace of $(S_n-S_{n-k}, 0\leq k\leq n/2)$ and recount on the same arguments as in the proof of \eqref{mSMScvg}. Thanks to \eqref{mSjointcvg}, \eqref{sumcvg} follows immediately.
\end{proof}

\begin{proof}[Proof of \eqref{eSMSbd}] Let $\uS_n:=-S_n$. Observe that
\begin{equation*}
\E[e^{S_n}; \MS_n\leq 0]=\E[e^{-\uS_n}; \underline{\uS}_n\geq0]\leq \sum_{k=0}^\infty e^{-k}\P\left[ \underline{\uS}_n\geq0, \uS_n\in[k,k+1)\right].
\end{equation*}
Applying \eqref{mSSbd} to $\uS$ implies that
\begin{align*}
\E[e^{S_n}; \MS_n\leq 0]&\leq \sum_{k=0}^\infty e^{-k} \frac{c(1+k)}{n^{3/2}}\leq \frac{c}{n^{3/2}},
\end{align*}
since $\sum_{k\geq0}(1+k)e^{-k}<\infty$.
\end{proof}

\begin{proof}[Proof of \eqref{eSMSbd+}] By applying Markov property at the first hitting time $\MS_n$, one sees that
\begin{align}
\E[e^{S_n}; \MS_n\leq A]=& \sum_{k=0}^n \E[e^{S_n}; \MS_{k-1}<S_k\leq A, S_k\geq \MS_{[k,n]}]\nonumber\\
=&\sum_{k=0}^n \E[e^{S_k}; \MS_{k-1}<S_k\leq A]\E[e^{S_{n-k}}; \MS_{n-k}\leq0]\nonumber\\
=&\sum_{k=0}^n \E[e^{S_k}; \mS_k>0, S_k\leq A]\E[e^{S_{n-k}}; \MS_{n-k}\leq0]\label{eSeS}
\end{align}
where the last equality follows from time-reversing. Next, one observes that for any $k\geq1$, by \eqref{mSSbd},
\begin{align*}
\E[e^{S_k}; \mS_k>0, S_k\leq A]\leq & \sum_{j\in [0, A)\cap\z}e^{j+1}\P\Big(\mS_k>0, S_k\in [j,j+1]\Big)\\
\leq & \frac{c}{k^{3/2}} \sum_{j\in [0, A)\cap\z}e^{j+1}(1+j)\leq \frac{c (1+A) e^{A}}{k^{3/2}},
\end{align*}
since $\sum_{j\in [0, A\sqrt{n})\cap\z}e^{j+1}(1+j)\leq c (A+1)e^{A}$. Plugging this inequality and \eqref{eSMSbd} into \eqref{eSeS} yields that
\begin{align*}
\E[e^{S_n}; \MS_n\leq A]\leq & \sum_{k=0}^n  \frac{c (1+A) e^{A}}{(k+1)^{3/2}(n-k+1)^{3/2}}\leq \frac{c (1+A)e^{A}}{n^{3/2}},
\end{align*}
which is what we need.
\end{proof}

\begin{proof}[Proof of \eqref{eSbd}] We have,
\begin{equation*}
\E\left[\frac{e^{S_n}}{\sum_{1\leq i\leq n}e^{S_i}}\right]\leq\E\left[e^{S_n-\MS_n}\right]=\sum_{k=0}^n\E\left[e^{S_n-S_k}; S_k>\MS_{k-1}, S_k\geq \MS_{[k,n]}\right],
\end{equation*}
then by Markov property and a time reversal for $(S_j, 0\leq j\leq k)$, one gets that
\begin{align*}
\E\left[\frac{e^{S_n}}{\sum_{1\leq i\leq n}e^{S_i}}\right]\leq&\sum_{k=0}^n\P\left(S_k>\MS_{k-1}\right)\E\left[e^{S_{n-k}}; \MS_{n-k}\leq0\right]\\
\leq & \sum_{k=0}^n\P\left(\mS_k>0\right)\E\left[e^{S_{n-k}}; \MS_{n-k}\leq0\right].
\end{align*}
By \eqref{mSbd} and \eqref{eSMSbd}
\begin{equation*}
\E\left[\frac{e^{S_n}}{\sum_{1\leq i\leq n}e^{S_i}}\right]\leq  \sum_{k=0}^n \frac{c}{(k+1)^{1/2}(n-k+1)^{3/2}}\leq \frac{c}{\sqrt{n}}.
\end{equation*}
\end{proof}

\eqref{mSmSbd} follows immediately from Lemma 3 in \cite{Ritter81}.

\begin{proof}[Proof of \eqref{LDP}] For $\theta>0$ such that $\varphi(\theta):=\log \E[e^{\theta S_1}]\in (-\infty,\infty)$, $\{e^{\theta S_n-n\varphi(\theta)}; n\geq0\}$ is a non-negative martingale. The existence of $\theta$ comes from \eqref{hyp0+}. Therefore, by Doob's inequality,
\begin{align*}
\P\left(\MS_n\geq n^{1+\delta}\right)\leq& \P\left(\max_{0\leq k\leq n}e^{\theta S_k-k \varphi(\theta)}\geq e^{\theta n^{1+\delta}-n\varphi(\theta)}\right)\\
\leq& e^{-\theta n^{1+\delta}+n\varphi(\theta)}\E\left[e^{\theta S_n-n\varphi(\theta)}\right]= e^{-\theta n^{1+\delta}+n\varphi(\theta)}.
\end{align*}
For $n$ large enough, $\theta n^{1+\delta}-n\varphi(\theta)\geq \theta n^{1+\delta}/2$. Hence, for any $n\geq1$,
\[
\P\left(\MS_n\geq n^{1+\delta}\right)\leq c(\delta,\theta)e^{-\theta n^{1+\delta}/2}.
\]
\eqref{Last1} can be treated similarly choosing $\theta$ properly as a, decreasing to zero,   function of $n$.
\end{proof}

\begin{proof}[Proof of \eqref{omSmScvg}] Let 
\[
\P_\eqref{omSmScvg}:=\P\Big( \mS_{[n/2, n]}\leq b_n,\mS_n\geq-\alpha, S_n>\MS_{n-1}\Big)
\]
Use again the notation $R_k=S_n-S_{n-k}$, we observe that
\begin{align*}
\P_\eqref{omSmScvg}&=\P\Big(\mS_{n/2}\geq-\alpha, \overline{R}_{n/2}-R_{n/2}\in[(S_{n/2}-b_n)_+, S_{n/2}+\alpha], \underline{R}_{n/2}>0, R_{n/2}> \MS_{n/2}-S_{n/2}\Big)\\
&\leq \E\Big[\mS_{n/2}\geq-\alpha, \hat{f}_n(\frac{S_{n/2}}{\sqrt{n/2}})\Big],
\end{align*}
where
\[
\hat{f}_n(x):=\P\left(\frac{\overline{R}_{n/2}-R_{n/2}}{\sqrt{n/2}}\in\Big[(x-\frac{b_n}{\sqrt{n/2}}), x+\frac{\alpha}{\sqrt{n/2}}\Big], \underline{R}_{n/2}>0\right).
\]
By invariance principle, $\P(\frac{\overline{R}_{n/2}-R_{n/2}}{\sqrt{n/2}}\leq x\vert \underline{R}_{n/2}>0)$ converges to $\P(\overline{\mB}_1-\mB_1\leq x)$ uniformly for $x\in\r_+$. Consequently,
\[
\hat{f}_n(x)= \frac{o_n(1)}{\sqrt{n}},\textrm{ uniformly for } x\in\r_+,
\]
so 
\[
\P_\eqref{omSmScvg} \leq \frac{o_n(1)}{\sqrt{n}} \P\Big(\mS_{n/2}\geq-\alpha\Big)= \frac{o_n(1)}{n}.
\]
\end{proof}

\begin{proof}[Proof of \eqref{mSMSkbd}]
{We need to  obtain an upper bound for $\P\Big( S_{k_n}\in[k_n^{1/3},k_n],\mS_{[k_n,k]}\leq k_n^{1/6}, S_k>\overline{S}_{k-1}, \underline{S}_k\geq-\alpha\Big)$. One sees that by \eqref{omSmScvg}, for any $n_0<k\leq n$,
\begin{align*}
&\P\Big( S_{k_n}\in[k_n^{1/3},k_n],\mS_{[k_n,k]}\leq k_n^{1/6}, S_k>\overline{S}_{k-1}, \underline{S}_k\geq-\alpha\Big)\\
\leq&\P\Big(\min_{k/2< j\leq k}S_j\leq k_n^{1/6}, S_k>\overline{S}_{k-1}, \underline{S}_k\geq-\alpha\Big)+\P\Big(\mS_{[k_n,k/2]}\leq k_n^{1/6},\mS_{(k/2,k]}>k_n^{1/6}, S_k>\overline{S}_{k-1}, \underline{S}_k\geq-\alpha\Big)\\
=&\frac{o_n(1)}{n}+\P\Big(\mS_{[k_n,k/2]}\leq k_n^{1/6}<\mS_{(k/2,k]}, S_k>\overline{S}_{k-1}, \underline{S}_k\geq-\alpha\Big).
\end{align*}
By \eqref{xi1}, to conclude that $\xi_1=o_n(1)$, it suffices to show that uniformly on $k\in [n_0,n]\cap\z$,
\[
\P\Big(\mS_{[k_n,k/2]}\leq k_n^{1/6}<\mS_{(k/2,k]}, S_k>\overline{S}_{k-1}, \underline{S}_k\geq-\alpha\Big)=\frac{o_n(1)}{n}.
\]
Considering the first  hitting time of $\mS_k$ which should be before $k/2$, one has 
\begin{align}
&\P\Big(\mS_{[k_n,k/2]}\leq k_n^{1/6}<\mS_{(k/2,k]}, S_k>\overline{S}_{k-1}, \underline{S}_k\geq-\alpha\Big)\nonumber\\
\leq & \sum_{0\leq j\leq k/2}\P\Big(\mS_{[k_n,k/2]}\leq k_n^{1/6}, S_k>\overline{S}_{k-1}, \mS_{j-1}>S_j=\underline{S}_k\geq-\alpha\Big).\label{xi1P}
\end{align}
For any $k_n/2\leq j\leq k/2$ and $n_0\leq k\leq n$, by Markov property at time $j$,
\begin{align*}
&\sum_{k_n/2\leq j\leq k/2}\P\Big(\mS_{[k_n,k/2]}\leq k_n^{1/6}, S_k>\overline{S}_{k-1}, \mS_{j-1}>S_j=\underline{S}_k\geq-\alpha\Big)\\
\leq&\sum_{k_n/2\leq j\leq k/2}\P\Big(\mS_{j-1}>S_j\geq-\alpha\Big)\P\Big(\mS_{k-j}\geq0, S_{k-j}=\MS_{k-j}\Big)
\end{align*}
which by \eqref{SmSbd} and \eqref{mSMSbd} is bounded by  $\sum_{k_n/2\leq j\leq k/2}\frac{c(1+\alpha)^2}{j^{3/2}n}=\frac{o_n(1)}{n}$. Also when $j\leq k_n/2$, applying Markov property at time $2k/3$ then at time $j$ implies that
\begin{align*}
&\P\Big(\mS_{[k_n,k/2]}\leq k_n^{1/6}, S_k>\overline{S}_{k-1}, \mS_{j-1}>S_j=\underline{S}_k\geq-\alpha\Big)\\
\leq & \P\Big(\mS_{[k_n,k/2]}\leq k_n^{1/6}, \mS_{j-1}>S_j=\underline{S}_{2k/3}\geq-\alpha\Big)\P\Big(S_{k/3}>\MS_{k/3-1}\Big)\\
\leq&\P\Big(\mS_{j-1}>S_j\geq-\alpha\Big)\P\Big(\mS_{7k/12}\geq0, \mS_{[k_n/2,k/2]}\leq k_n^{1/6}+\alpha\Big)\P\Big(S_{k/3}>\MS_{k/3-1}\Big),
\end{align*}
where for the random walk from the time $j$ to $2k/3$, we use the fact that $\{\mS_{[k_n-j,k/2-j]}\leq k_n^{1/6}+\alpha, \mS_{2k/3-j}\geq0\}\subset\{\mS_{7k/12}\geq0, \mS_{[k_n/2,k/2]}\leq k_n^{1/6}+\alpha\}$ as $j\leq k_n/2$.

By time reversal together with \eqref{mSbd}, $\P\Big(\MS_{k/3-1}<S_{k/3}\Big)\leq c/\sqrt{k/3}$. Also, in view of \eqref{SmSbd} and \eqref{mSmSbd}, for any $n_0\leq k\leq n$,
\[
\sum_{j\leq k_n/2}\P\Big(\mS_{[k_n,k/2]}\leq k_n^{1/6}, S_k>\overline{S}_{k-1}, \mS_{j-1}>S_j=\underline{S}_k\geq-\alpha\Big)\leq \sum_{j\leq k_n/2}\frac{c(1+\alpha)^2(1+\alpha+k_n^{1/6})}{(j+1)^{3/2}nk_n^{1/2}}=\frac{o_n(1)}{n}.
\]}
\end{proof}
\begin{proof}[Proof of \eqref{kmSMSbd}]
{Applying Markov property at time $k_n$ yields that
\[
\P_\eqref{kmSMSbd}=\E\Big(\P_{S_{k_n}}\Big(\mS_{k-k_n}\geq-\alpha, S_{k-k_n}=\MS_{k-k_n}\Big),\mS_{k_n}\geq -\alpha, S_{k_n}\notin[k_n^{1/3},k_n]\Big),
\]
and recall that $\P_u$ is for the distribution of the random walk starting from $u$.
 By \eqref{mSMSbd}, $\P_{S_{k_n}}\Big(\mS_{k-k_n}\geq-\alpha, S_{k-k_n}=\MS_{k-k_n}\Big)\leq c(1+\alpha+S_{k_n})/(k-k_n)$. This  yields
\[
\P_\eqref{kmSMSbd}\leq\E\Big[\frac{(1+\alpha+S_{k_n})}{k-k_n}; \mS_{k_n}\geq -\alpha, S_{k_n}\notin[k_n^{1/3},k_n]\Big].
\]
We now split the above expectation into two terms, first by Markov's inequality,
\begin{align*}
\E\Big[\frac{(1+\alpha+S_{k_n})}{k-k_n}; \mS_{k_n}\geq -\alpha, S_{k_n}\geq k_n\Big]&\leq \frac{c}{n(1+\alpha+k_n)^3}\E\Big[(1+\alpha+S_{k_n})^4\Big] \leq \frac{c}{n k_n},
\end{align*}
and also by \eqref{mSSbd}
\begin{align*}
&\E\Big[\frac{(1+\alpha+S_{k_n})}{k-k_n}; \mS_{k_n}\geq -\alpha, S_{k_n}\leq k_n^{1/3}\Big]\leq \sum_{l=-\alpha}^{k_n^{1/3}}\E\Big[\frac{(1+\alpha+S_{k_n})}{k-k_n}; \mS_{k_n}\geq -\alpha, S_{k_n}\in[l,l+1]\Big]\\
\leq &\sum_{l=-\alpha}^{k_n^{1/3}}\frac{c(1+\alpha+l)}{n}\P\Big(\mS_{k_n}\geq -\alpha, S_{k_n}\in[l,l+1]\Big)\leq \frac{c}{nk_n^{1/2}}.
\end{align*}
These two inequalities conclude \eqref{kmSMSbd}.}
\end{proof}
\begin{proof}[Proof of \eqref{eSMSmSbd}]
Arguing over the first time hitting $\MS_\ell$ then by Markov property, we have
\begin{align*}
\E_\eqref{eSMSmSbd}\leq&\sum_{j=1}^\ell\E\left[e^{S_j}; S_j\leq A, \mS_j\geq-\alpha\right]\E\left[e^{S_{\ell-j}}\un_{\{ \MS_{\ell-j}\leq 0\} }\right]\\
 \leq& \sum_{j=1}^\ell\sum_{-\alpha\leq k\leq A}e^{k+1}\P\left[S_j\in[k,k+1], \mS_j\geq-\alpha\right]\E\left[e^{S_{\ell-j}}\un_{\{ \MS_{\ell-j}\leq 0\} }\right].
\end{align*}
By \eqref{eSMSbd} and \eqref{mSSbd}, we have
\[
\E_\eqref{eSMSmSbd}\leq \sum_{j=1}^\ell\sum_{-\alpha\leq k\leq A}e^{k+1} \frac{c(1+\alpha)(1+k+\alpha)}{j^{3/2}(\ell-j+1)^{3/2}}\leq  \frac{ce^A(1+\alpha)(1+A+\alpha)}{\ell^{3/2}}.
\]
\end{proof}
\begin{proof}[Proof of \eqref{eSMSMMSmSbd}]
Arguing over the value of $\Upsilon_S$ implies that
\begin{align*}
\E_\eqref{eSMSMMSmSbd}=&\E\left[e^{S_\ell-\MS_\ell}; \max_{1\leq k\leq \ell}(\MS_k-S_k)\leq A,\mS_\ell\geq-\alpha\right]\\
= & \sum_{j=1}^\ell \E\left[e^{S_\ell-\MS_\ell}; \MS_{j-1}<S_j=\MS_\ell, \max_{1\leq k\leq \ell}(\MS_k-S_k)\leq A,\mS_\ell\geq-\alpha\right],
\end{align*}
which by Markov property at time $j$, is bounded by
\[
\sum_{j=1}^\ell \P\left[\MS_{j-1}<S_j, \max_{1\leq k\leq j}(\MS_k-S_k)\leq A,\mS_j\geq-\alpha\right]\E\left[e^{S_{\ell-j}}\un_{\{ \MS_{\ell-j}\leq0\}}\right].
\]
By \eqref{eSMSbd},
\begin{align}
\E_\eqref{eSMSMMSmSbd}\leq&\sum_{j=1}^\ell \P\left[\MS_{j-1}<S_j, \max_{1\leq k\leq j}(\MS_k-S_k)\leq A,\mS_j\geq-\alpha\right]\frac{c}{(\ell-j+1)^{3/2}}\nonumber\\
=:& \sum_{j=1}^\ell r_{\ell,j}\label{bdRI}.
\end{align}
We split this sum into two parts: $\sum_{j=1}^{\ell-\ell^{1/3}}$ and $\sum_{\ell-\ell^{1/3}\leq j\leq \ell}$. For the first sum,  by \eqref{mSMSbd}, one sees that
\begin{align*}
\sum_{j=1}^{\ell-\ell^{1/3}}r_{\ell,j}&\leq \sum_{j=1}^{\ell-\ell^{1/3}} \P\left(\mS_j\geq-\alpha, S_j=\MS_j\right)\frac{c}{(\ell-j+1)^{3/2}}\\
&\leq \sum_{j=1}^{\ell-\ell^{1/3}} \frac{c(1+\alpha)}{j(\ell-j+1)^{3/2}}\leq \frac{c(1+\alpha)}{\ell^{7/6}}.
\end{align*}
For the second sum, by Markov property at $j/3$ and $2j/3$ then by reversing time,
\begin{align*}
r_{\ell,j}\leq & \P\left(\mS_{j/3}\geq-\alpha\right)\P\left(\max_{k\leq j/3}(\MS_k-S_k)\leq A\right)\P\left(S_{j/3}=\MS_{j/3}\right)\frac{c}{(\ell-j+1)^{3/2}}\\
= & \P\left(\mS_{j/3}\geq-\alpha\right)\P\left(\max_{k\leq j/3}(\MS_k-S_k)\leq A\right)\P\left(\mS_{j/3}\geq0\right)\frac{c}{(\ell-j+1)^{3/2}}.
\end{align*}
It is known by \cite{HuShi15} that for sufficiently large $\lambda>0$, $\P(\max_{1\leq k\leq j}(\MS_k-S_k)\leq \lambda)\leq e^{-c\lfloor \frac{j}{\lfloor \lambda^2\rfloor}\rfloor}$. This together with \eqref{mSbd} implies that for any $n$ large enough,
\begin{align*}
\sum_{\ell-\ell^{1/3}\leq j\leq \ell}r_{\ell,j}&\leq \sum_{\ell-\ell^{1/3}\leq j\leq \ell} \frac{c'(1+\alpha)}{j(\ell-j+1)^{3/2}}e^{-c\frac{j}{A^2}} \leq  \frac{c'(1+\alpha)}{\ell}e^{-c \frac{\ell}{2A^2}},
\end{align*}
which ends the proof.
\end{proof}

\noindent \\ \textit{Proof of} \eqref{eSMSMSmSbd}
By Markov property at time $\Upsilon_S=j$,
\begin{align*}
\E_\eqref{eSMSMSmSbd}:=&\E\left[e^{S_\ell-\MS_\ell};  \MS_\ell\geq A, \mS_\ell\geq-\alpha\right]
= \sum_{j=1}^\ell\P\left(\mS_j\geq-\alpha, \MS_j=S_j\geq A\right)\E\left[e^{S_{\ell-j}}\un_{\{ \MS_{\ell-j}\leq0\} }\right],
\end{align*}
which by \eqref{eSMSbd}, is bounded by $\sum_{j=1}^\ell\P\left(\mS_j\geq-\alpha, \MS_j=S_j\geq A\right) {c}{(\ell-j+1)^{-3/2}}$.
Then by \eqref{mSMSSbd0}, 
$\E_\eqref{eSMSMSmSbd}\leq \sum_{j=1}^\ell\frac{c(1+\alpha)}{ j^{1/2}(\ell-j+1)^{3/2}A}
\leq  \frac{c(1+\alpha)}{\ell^{1/2}A}$. \hfill   $\square$

\noindent \\
Acknowledgment : We are grateful to an anonymous referee for comments that were very useful to help improve the presentation of the paper.

\bibliographystyle{plain}
\bibliography{thbiblio}
\end{document}